\documentclass{amsart}
\usepackage[margin=1in]{geometry}

\usepackage%
[bookmarksdepth=subsubsection, unicode]%
{hyperref}

\usepackage{bm,booktabs,comment,mathtools,multicol,tikz-cd,amssymb,stmaryrd}

\usepackage[final,notref,notcite]{showkeys}

\newtheorem{result}{Result}[section]
	\newtheorem{corollary}[result]{Corollary}
	\newtheorem{conjecture}[result]{Conjecture}
	
	\newtheorem{lemma}[result]{Lemma}
	\newtheorem{proposition}[result]{Proposition}
	
	\newtheorem{theorem}[result]{Theorem}

\theoremstyle{definition}
	\newtheorem{construction}[result]{Construction}
	\newtheorem{definition}[result]{Definition}

\theoremstyle{remark}
	
	\newtheorem{remark}[result]{Remark}

\DeclareMathOperator\Hom{Hom}

\DeclareMathOperator\ind{ind}

\DeclareMathOperator\res{res}
\DeclareMathOperator\Res{Res}

\DeclareMathOperator\SL{SL}

\newcommand{\ov}{\overline}
\newcommand{\co}[1]{\prescript{#1}{}}

\newcommand\acts\curvearrowright

\newcommand\mathdef{\textsf}
\newcommand\tn{\textnormal}

{\renewcommand{\descriptionlabel}%
[1]{\hspace{\labelsep}##1:}
\begin{description}
\setlength\itemsep{1em}}
{\end{description}}

\title{Isocrystals and limits of rigid local Langlands correspondences} 

\author{Peter Dillery}
\thanks{The author was supported by the Brin Postdoctoral Fellowship at the University of Maryland, College Park}
  \address{Department of Mathematics,
  University of Maryland, 4176 Campus Drive,
  College Park, MD 20742-4015, USA}

\begin{document}

\begin{abstract}
We show that, over a nonarchimedean local field, the rigid refined local Langlands correspondence and associated endoscopic character identities for connected reductive $G$ follow if one only has them for all such $G$ with connected center. The strategy is to construct a projective system of central extensions and then take limits of the Langlands correspondences (and endoscopic data) of each group in the system. As an application, we prove the equivalence of the rigid refined local Langlands correspondence and its analogue for isocrystals, generalizing the work of \cite{kaletha18} in the $p$-adic case. 
\end{abstract}

\maketitle

\setcounter{tocdepth}{1}
\tableofcontents


\section{Introduction}\label{Intro}

\subsection{Motivation}\label{Intro1}
The goal of this paper is to show that if one has a construction of the rigid refined local Langlands correspondence (for a nonarchimedean local field $F$), for connected reductive groups $G$ with connected center then the correspondence can be canonically extended to arbitrary connected reductive $G$. This procedure has two steps: First, we define the candidate local Langlands correspondence for arbitrary $G$ using its connected center counterpart(s). Second, we show that the construction defined in the first step satisfies the endoscopic character identities (assuming that they hold in the connected center case). This generalizes the work of \cite{kaletha18}, which establishes these results for $p$-adic fields; our method works uniformly for all nonarchimedean local fields and recovers Kaletha's approach in the $p$-adic case (via using the constant projective system of embeddings, explained below).

Before going further, we recall (a rough version of) the local Langlands correspondence. Given $G$ as above, there is conjectured to be a finite-to-one map 
\begin{equation}\label{introLLC}
\Pi(G(F)) \to \Phi(G)
\end{equation}
from isomorphism classes of smooth, irreducible representations of $G(F)$ on $\mathbb{C}$-vector spaces 
to conjugacy classes of $L$-parameters for $G$. There are numerous approaches to parametrizing the fibers of \eqref{introLLC} (see e.g. \cite[\S 6.3]{Taibi22} for a summary); this paper will focus primarily on the \mathdef{rigid refined local Langlands correspondence} (we abbreviate ``local Langlands correspondence" by ``LLC"), initially defined for local fields of characteristic zero in \cite{kaletha16} and extended to all nonarchimedean local fields in \cite{Dillery1}. For $G$ quasi-split, given an $L$-parameter $W_{F} \times \SL_{2}(\mathbb{C}) \xrightarrow{\varphi} \prescript{L}{}G$ and finite central subgroup $Z \subset G$, the rigid refined LLC predicts that the fiber of \eqref{introLLC} over $[\varphi]$ can be understood via a bijection  
\begin{equation*}
  \Pi_{\varphi}^{Z}  \xrightarrow{\iota_{\mathfrak{w}}} \tn{Irr}(\pi_{0}(S_{\varphi}^{+}))
\end{equation*}
depending on a choice of Whittaker datum $\mathfrak{w}$ for $G$, where $\Pi_{\varphi}^{Z}$ is a finite set of isomorphism classes of representations of $Z$-rigid inner twists of $G$ (see Definition \ref{RITdef}---for now it suffices to think of this as a collection of representations of $G'(F)$ where $G'$ varies along all inner forms of $G$ together with an extra piece of data that ensures automorphisms preserve representations up to isomorphism) which all map to $\varphi$ under \eqref{introLLC} and $S_{\varphi}^{+}$ is the preimage of $Z_{\widehat{G}}(\varphi)$ in $\widehat{G/Z}$. 

It is useful to reduce to $G$ with connected center because in this situation the rigid refined LLC can be obtained from the \mathdef{isocrystal refined LLC} in the sense of \cite[Conjecture F]{kaletha16a}, and vice versa. Recall that the isocrystal refined LLC only captures all inner forms of a group $G$ when $Z(G)$ is connected, whereas its rigid analogue always captures all inner forms (in practice, we choose $G$ to be quasi-split and then realize an arbitrary connected reductive group as one of its rigid inner forms). Our work shows that, nevertheless, the isocrystal LLC suffices for capturing all possible connected reductive $G$, since we show that one can obtain the rigid refined LLC for arbitrary quasi-split $G$ from the rigid refined LLC for a projective system of quasi-split groups $\{G_{z}^{(i)}\}_{i \geq 0}$ which all have connected center, and so their rigid LLC can further be obtained from their isocrystal refined LLC. 

Establishing an equivalence between these two refined LLC's, in addition to showing that the isocrystal refined LLC captures all connected reductive groups, also has the advantage of placing the rigid refined LLC in closer proximity to the Fargues-Scholze construction (as in \cite{FS21}) of the local Langlands correspondence using the geometry of the Fargues-Fontaine curve. We believe that relating these two approaches to the LLC will be profitable---for example, on the Galois side, the sheaves that arise from the rigid refined LLC (cf. \cite[\S 3]{DS23}) are constructible, whereas one expects the isocrystal refined LLC to give coherent sheaves (on the stack of $L$-parameters), even in the case of basic isocrystals. Therefore, the ability to interpolate between the two refined LLC's (restricted to basic isocrystals for the moment) gives a very crude way of understanding some of the above coherent sheaves in terms of significantly more tractable constructible ones.  

When $F$ is a $p$-adic local field the above reduction to connected center and subsequent comparison to the isocrystal refined LLC has been carried out in \cite{kaletha18}. In this situation, one constructs a \mathdef{pseudo $z$-embedding} (cf. Definition \ref{pseudodef}) of $G$ an arbitrary connected reductive group into another such group $G_{z}$ with connected center and then deduces the rigid refined LLC (including the endoscopic character identities) for $G$ using its analogue for $G_{z}$; to the author's knowledge, this approach was first suggested by Kottwitz. The above embedding is a central extension which has the key property that the induced map 
\begin{equation}\label{introisom}
    H^{1}(F, Z(G)) \to H^{1}(F, Z(G_{z}))
\end{equation}
is an isomorphism. For example, this implies that any $L$-parameter $\varphi$ for $G$ can be lifted to an $L$-parameter $\varphi_{z}$ for $G_{z}$ and the induced map $\pi_{0}(S_{\varphi_{z}}^{+}) \to \pi_{0}(S_{\varphi}^{+})$ is an isomorphism for any $Z$ as above. 

However, when $F$ is a local function field, it is in general not possible to centrally embed $G$ into $G_{z}$ with connected center and have the isomorphism \eqref{introisom} (where now by $H^{1}(F, Z(G))$ we mean fppf cohomology). An easy way to see this fact is that the left-hand side of \eqref{introisom} can be infinite while the right-hand side is always finite. In order to circumvent this difficulty we replace the embedding $G \to G_{z}$ with an infinite projective system $\{G \to G_{z}^{(i)}\}$ of compatible embeddings, where each $G_{z}^{(i)}$ has connected center. Each embedding in this system will not be a pseudo $z$-embedding, but a less restrictive notion that we call a \mathdef{weak $z$-embedding} (Definition \ref{weakdef}). Although each individual embedding does not satisfy \eqref{introisom}, the entire system plays the role of a single pseudo $z$-embedding. In particular, one can still lift parameters, compare centralizer subgroups, and relate representations on the ``automorphic side" in order to construct the rigid refined LLC for $G$ and prove the endoscopic character identities (assuming both of these things for each $G_{z}^{(i)}$). 

Constructing the rigid refined LLC for $G$ using the projective system $\{G_{z}^{(i)}\}$ proceeds via a ``limit" of Langlands correspondences, which requires a basic functoriality property (Conjecture \ref{Maartenbis}) of the correspondence along central surjections (in order to pass between $G_{z}^{(j)}$ and $G_{z}^{(i)}$ for $j \geq i$) in the connected center case. An analogous version of this functoriality is also assumed (again, for groups with connected center) in \cite{kaletha18} in order for the construction of the rigid refined LLC loc. cit. in the disconnected center case to be well-defined. 

This aforementioned functoriality also appears in our study of endoscopy, where one transfers an endoscopic datum for $G$ to a system of endoscopic data for $\{G_{z}^{(i)}\}$ (for $i \gg 0$, cf. Proposition \ref{liftend}) and works with the associated projective system $\{H_{z}^{(i)}\}$ of endoscopic groups, where now the centers of each $H_{z}^{(i)}$ need not be connected. We prove that the desired functoriality holds for our above construction of the rigid refined LLC under the assumption that it holds in the connected center case, which enables us to compute the LLC for the endsocopic group $H$ using a limit of correspondences for the system $\{H_{z}^{(i)}\}$. The reduction of the LLC for $H$ to $\{H_{z}^{(i)}\}$ is the key ingredient for deducing the endoscopic character identities for $G$ using those for each $G_{z}^{(i)}$ and endsocopic group $H_{z}^{(i)}$.

Another difficulty particular to the function field case arises when one carries out the rigid-to-isocrystal comparison. As mentioned above, for any nonarchimedean local $F$, when $Z(G)$ is connected every inner twist of a quasi-split $G$ can be reached using isocrystal inner twists (also called \mathdef{extended pure inner twists}), which means that the canonical comparison map (described explicitly in \S \ref{IsoComp}) 
\begin{equation}\label{introcomp}
    B(G)_{\tn{bas}} \to H^{1}_{\tn{bas}}(\mathcal{E}, G)
\end{equation}
from basic $G$-isocrystals to basic $G$-torsors on the Kaletha gerbe (cf. \S \ref{PrelimKal}) while not surjective in general, becomes so after quotienting the right-hand side by the action of $\varinjlim_{Z} H^{1}(\mathcal{E}, Z)$. To transfer \eqref{introcomp} to the Galois side in the $p$-adic case \cite{kaletha18} proves a duality isomorphism 
\begin{equation*}
    H^{1}(\mathcal{E}, Z) \xrightarrow{\sim} Z^{1}(\Gamma, \widehat{Z})^{*},
\end{equation*}
(where $\Gamma = \Gamma_{F^{s}/F}$) by embedding $Z(G)$ strategically into a torus, where one has access to analogous duality results. 

The above torus embedding strategy does not work for local function fields (this is again related to the general failure of \eqref{introisom}). However, we prove the same duality isomorphism directly using unbalanced cup products in \v{C}ech cohomology, following the same general strategy for the proof of the analogous isomorphism for $H_{\tn{fppf}}^{1}(F, Z)$. This duality result is closely related to the Tate-Nakayama isomorphism for $H^{1}(\mathcal{E}, Z \to T)$ proved in \cite{Dillery1}, which motivated the definition of the aforementioned \v{C}ech version of unbalanced cup products (cf. \cite[\S 4.2]{Dillery1}). 

\subsection{Overview}\label{Intro2}
This paper has six sections and one appendix and its structure is as follows. In \S \ref{Prelminaries} we review cohomology, gerbes, rigid inner forms, and the rigid refined LLC. We approach gerbes and torsors on them via \v{C}ech cohomology associated to the cover $\mathrm{Spec}(\overline{F}) \to \mathrm{Spec}(F)$ for the sake of concreteness. The heart of the paper is \S \ref{Construction}, where we define and study certain systems of central embeddings and use them to extend the rigid refined LLC from the connected center case to all connected reductive groups. In \S \ref{Endoscopy} we study the endoscopic properties of the construction in \S \ref{Construction}, using systems of endoscopic data to establish the endoscopic character identities (again, assuming that they hold in the connected center case). 

Having studied the reduction of the rigid refined LLC to groups with connected center, we turn to comparison with the isocrystal refined LLC. As mentioned above, this naturally motivates a version of Tate-Nakayama duality for the gerbe cohomology groups $H^{1}(\mathcal{E}, Z)$ for finite multiplicative $Z$, which is the content of \S \ref{Tateduality}. We compare the two LLC's in \S \ref{Iso}, which (other than \S \ref{IsoFunc}) is essentially just a summary of the characteristic zero situation studied in \cite[\S 5.2]{kaletha18} with some \v{C}ech-cohomological adjustments. Appendix \ref{Braided} discusses crossed modules for the \v{C}ech cohomology associated to an fpqc cover of rings $\mathrm{Spec}(S) \to \mathrm{Spec}(R)$ and then specializes to the case of $R = F$, $S = \overline{F}$ to extend certain $p$-adic duality results to the local function setting (building on \cite[Appendix A]{Dillery2}, which studies complexes of tori in this setting) required for lifting $L$-parameters for $G$ to ones for the groups $G_{z}^{(i)}$. 

\subsection{Notation and terminology}\label{Intro3}
We assume that $F$ is a local field of characteristic $p > 0$, although all arguments in this paper hold for arbitrary nonarchimedean local fields. We fix an algebraic closure $\overline{F}$ of $F$ with associated separable closure $F^{s}$, absolute Galois group $\Gamma = \Gamma_{F^{s}/F}$, Weil group $W_{F}$, and Weil-Deligne group $L_{F} := W_{F} \times \SL_{2}(\mathbb{C})$ (although essentially all non-endoscopic results in this paper still hold when $\mathbb{C}$ is replaced by an arbitrary algebraically closed field $C$ of characteristic zero if one strategically removes the word ``tempered"). 

We call an affine, commutative algebraic group over $F$ \mathdef{multiplicative} if its characters span its coordinate ring over $F^{s}$. When $N$ is a discrete $\Gamma$-module, we denote by $\underline{N}$ the associated \'{e}tale group scheme over $F$. For a morphism of $X \xrightarrow{f} Y$ of $R$-schemes ($R$ is any commutative ring) and $S$ an $R$-algebra, we denote by $f^{\sharp}$ the induced map $X(S) \to Y(S)$. For notational convenience we will often denote the fibered product of two $R$-schemes $X_{1}$, $X_{2}$ by $X_{1} \times_{R} X_{2}$ rather than $X_{1} \times_{\mathrm{Spec}(R)} X_{2}$, and also a faithfully flat (and thus fpqc) cover $\mathrm{Spec}(S) \to \mathrm{Spec}(S)$ by the ring homomorphism $R \to S$. All group schemes are assumed to be affine unless explicitly stated otherwise.

For an algebraic group $G$, $G_{\tn{der}}$ denotes the derived subgroup of $G$, $G^{\circ}$ the connected component of the identity in $G$ with quotient $\pi_{0}(G) = G/G^{\circ}$, $Z(G)$ the center of $G$, $Z_{\tn{der}}$ the center of $G_{\tn{der}}$, and for a subscheme $H \subset G$, $Z_{G}(H)$ the scheme-theoretic centralizer of $H$ in $G$. When working with group schemes over $\mathbb{C}$ we will frequently conflate the scheme with its $\mathbb{C}$-points. For a topological group $H$ we denote by $H^{D}$ the group $\Hom_{\tn{cts}}(H, \mathbb{C}^{\times})$, where $\mathbb{C}^{\times}$ has the analytic topology (note that when $H$ is locally profinite this is the same as giving $\mathbb{C}^{\times}$ the discrete topology). When we just consider the group $\Hom_{\mathbb{Z}}(H, \mathbb{Q}/\mathbb{Z})$ of abstract group homomorphisms we denote it by $H^{*}$. 

\subsection{Acknowledgements}\label{Intro4} The author thanks Tasho Kaletha for suggesting this project. They also thank Sean Cotner, David Schwein, and Alex Youcis for their helpful conversations, as well as Jean-Pierre Labesse for useful email exchanges about crossed modules in the setting of local function fields. The author also thanks Alexander Bertoloni Meli for proofreading an older version of this paper, expository improvements, as well as numerous suggestions concerning direct limits of restrictions of representations.

\section{Preliminaries}\label{Prelminaries}

\subsection{Cohomological basics}\label{PrelimCohom}
For a commutative group scheme $A$ over a ring $R$, the notation $H^{i}(R, A)$ will always denote $H_{\tn{fppf}}^{i}(\mathrm{Spec}(R),A)$. For $G$ a general group scheme over $R$ we define $H^{1}(R, G)$ as the set of isomorphism classes of (fpqc) $G$-torsors over $R$ (using descent data, one verifies that this is indeed a set).

Given an fpqc cover $\mathrm{Spec}(S) \to \mathrm{Spec}(R)$ one can consider the \v{C}ech cohomology sets (groups when $G$ is commutative) $\check{H}^{i}(S/R, G)$ with coefficients in $G$, which are defined for all $i \geq 0$ when $G$ is commutative and for $i=0,1$ for general $G$. See e.g. \cite[\S 6.4]{Poonen17} for the commutative case and \cite[III.3.6]{Giraud71} for the general case. \v{C}ech cohomology will be reviewed further as needed in \S \ref{Iso} and Appendix \ref{Braided}. Recall that the sets $\check{H}^{i}(\overline{F}/F, G)$ coincide with $H^{i}(F, G)$ as defined in the previous paragraph for $R = F$ a field when $G$ is of finite type. Moreover, when $F$ is a nonarchimedean local field and $G=A$ is commutative of finite type, the groups $H^{i}(F, A)$ have a natural locally-compact topology (defined in \cite{Shatz}, cf. also \cite[\S III.6]{Milne06}). We recall the following useful topological property for $i=1$: 

\begin{lemma}
For nonarchimedean $F$, the topological group $H^{1}(F, A)$ with the natural topology as above is compact when $A$ is of multiplicative type.
\end{lemma}

\begin{proof}
This is explained in \cite[\S III.6]{Milne06}.
\end{proof}

We will need some basic \v{C}ech-cohomological notation: Denote by 
\begin{equation*}
\mathrm{Spec}(S) \times_{R} \mathrm{Spec}(S) \xrightarrow{p_{i}} \mathrm{Spec}(S) 
\end{equation*}
the projection onto the $i$th factor for $i=1,2$ (we can and do generalize this notation in the obvious way for sources and targets with more direct factors). 

\subsection{Gerbes}\label{PrelimGerbes}
We attempt to make the gerbes in this paper as concrete as possible (see, e.g., \cite[\S 2]{Dillery1} for a detailed study of their basic properties in the case of the \v{C}ech cover $\mathrm{Spec}(\overline{F}) \to \mathrm{Spec}(F)$). There is a more conceptual approach to this setup which is discussed at length in \cite[\S 2]{Dillery1}. We continue with our fixed fpqc cover $\mathrm{Spec}(S) \to \mathrm{Spec}(R)$.

For a commutative group scheme $A$ over $R$, recall that the \mathdef{classifying stack} $BA$ is the fibered category (in groupoids)
\begin{equation*}
    BA \to \mathrm{Sch}/\mathrm{Spec}(R)
\end{equation*}
whose fiber category $BA(\mathrm{Spec}(R'))$ for an $R$-algebra $R'$ is given by all $A$-torsors over $R'$, with morphisms given by maps of $A$-torsors. We have the following generalization of this object:

\begin{definition}\label{gerbedef}
Given a \v{C}ech $2$-cocycle $a \in A(S \otimes_{R} S \otimes_{R} S)$, we define the fibered category (in groupoids)
\begin{equation*}
\mathcal{E}_{a} \to \mathrm{Sch}/\mathrm{Spec}(R)
\end{equation*}
to have fiber category $\mathcal{E}_{a}(\mathrm{Spec}(R'))$ given by pairs $(T, \psi)$, where $T$ is an $A$-torsor over $\mathrm{Spec}(R' \otimes_{R} S)$ and $\psi$ is an isomorphism of $A$-torsors over $\mathrm{Spec}(R' \otimes S) \times_{R} \mathrm{Spec}(R' \otimes S)$
\begin{equation*}
\psi \colon p_{2}^{*}T \xrightarrow{\sim} p_{1}^{*}T
\end{equation*}
such that $d\psi$ equals translation by $a$ as an $A$-automorphism of $p_{1,2}^{*}T$.
\end{definition}

The above fibered category is an example of a \mathdef{gerbe banded by $A$} (cf. \cite[\S 2.3]{Dillery1}). All gerbes considered in this paper will be, up to isomorphism, of the form given in Definition \ref{gerbedef}. Any category over $\mathrm{Sch}/\mathrm{Spec}(R)$ fibered in groupoids inherits the fpqc topology from $\mathrm{Sch}/\mathrm{Spec}(R)$, and so it makes sense to define $H^{1}(\mathcal{E}_{a}, G)$ for $G$ an $R$-group scheme as the collection of all isomorphism classes of fpqc $G_{\mathcal{E}_{a}}$-torsors over $\mathcal{E}_{a}$. 

When $R = F$ a field and $S = \overline{F}$, we may interpret $H^{1}_{\tn{bas}}(\mathcal{E}_{a}, G)$ concretely as equivalence classes of \mathdef{$a$-twisted (\v{C}ech) cocycles with coefficients in $G$} which are pairs $(c, f)$ with $A \xrightarrow{f} Z(G)$ a morphism of group schemes over $F$ and $c \in G(\overline{F} \otimes_{F} \overline{F})$ such that $dc = f(a)$ (see \cite[\S 2.5]{Dillery1} for the full details). 

\subsection{The Kaletha gerbe}\label{PrelimKal}
Let $R = F$ be nonarchimedean local field. We now specialize the notions of the previous subsection, setting
\begin{equation*}
A = u := \varprojlim_{k \geq 0} \frac{\mathrm{Res}_{E_{n_{k}}/F}(\mu_{n_{k}})}{\mu_{n_{k}}}
\end{equation*}
where the limit is over a totally-ordered cofinal system of natural numbers $\{n_{k}\}_{k \geq 0}$ and a totally-ordered cofinal system $\{E_{n_{k}}\}_{k \geq 0}$ of finite Galois extensions of $F$, the embedding of each $\mu_{n_{k}}$ is given by the diagonal, and the transition map 
\begin{equation*}
   \frac{\mathrm{Res}_{E_{n_{k}}/F}(\mu_{n_{k}})}{\mu_{n_{k}}} \to \frac{\mathrm{Res}_{E_{n_{\ell}}/F}(\mu_{n_{\ell}})}{\mu_{n_{\ell}}} 
\end{equation*}
is a composition of the $E_{n_{k}}/E_{n_{\ell}}$-norm and the $n_{k}/n_{\ell}$-power map. The definition of $u$ is independent of the choice of systems. Denote the $k$th term in the limit defining $u$ by $u_{k}$.

One then computes (\cite[\S 3.1]{Dillery1}) that $H^{1}(F,u) = 0$ and $H^{2}(F, u)$ is canonically isomorphic to $\widehat{\mathbb{Z}}$; we set $\mathcal{E} = \mathcal{E}_{\xi}$ (as in Definition \ref{gerbedef}) for any \v{C}ech $2$-cocycle $\xi$ representing the class corresponding to $-1$ and call $\mathcal{E}$ the \mathdef{Kaletha gerbe}. Any two choices $\mathcal{E}_{\xi}$, $\mathcal{E}_{\xi'}$ differ via a non-canonical isomorphism of gerbes which (due to the vanishing of $H^{1}(F,u)$) induces a canonical isomorphism on cohomology sets
\begin{equation*}
  H^{1}(\mathcal{E}_{\xi}, G) \xrightarrow{\sim}  H^{1}(\mathcal{E}_{\xi'}, G) 
\end{equation*}
for any group scheme $G$, as shown in \cite[\S 2.5]{Dillery1}. For this reason it is harmless to fix such a choice of $\xi$.

Now assume that $G$ is either a connected reductive group or is of multiplicative type. For $Z \subset Z(G)$ a finite subgroup we define $H^{1}(\mathcal{E}, Z \to G)$ to be the isomorphism classes in $H^{1}(\mathcal{E}, G)$ whose image in $H^{1}(\mathcal{E}, G/Z)$ (via the contracted product map) descends to a $G/Z$-torsor over $F$. We then define the set of \mathdef{basic $G$-torsors} on $\mathcal{E}$ as
\begin{equation*}
    H^{1}_{\tn{bas}}(\mathcal{E}, G) = \varinjlim_{Z \subset_{\tn{finite}} Z(G)} H^{1}(\mathcal{E}, Z \to G),
\end{equation*}
with the obvious transition maps. 

For any normal embedding of connected reductive groups over $F$
\begin{equation*}
    1 \to H \to G \to G/H \to 1
\end{equation*}
one always has (by \cite[Proposition III.3.3.1]{Giraud71}) an exact sequence of pointed sets
\begin{equation*}
    1 \to H^{0}(\mathcal{E}, H) \to H^{0}(\mathcal{E}, G) \to H^{0}(\mathcal{E}, G/H) \to H^{1}(\mathcal{E}, H) \to H^{1}(\mathcal{E}, G) \to H^{1}(\mathcal{E}, G/H).
\end{equation*}

Note that $H^{0}(\mathcal{E}, M) = M(F)$ for any group scheme $M$ over $F$. If we insist further that $H \to G$ is a central extension, then for $Z \subset Z(H)$ finite, we have the following further functoriality (abusively denoting the image of $Z$ in $G$ also by $Z$):
\begin{equation}\label{basicgerbeLES}
    1 \to H(F) \to G(F) \to (G/H)(F) \to H^{1}(\mathcal{E}, Z \to H) \to H^{1}(\mathcal{E}, Z \to G) \to H^{1}(F, G/H);
\end{equation}
the fact that the last term is cohomology over $F$ follows from the fact that $Z \subset H$ and all torsors in $H^{1}(\mathcal{E}, Z \to G)$ descend after passing to $G/Z$.

\subsection{The local Langlands correspondence for rigid inner twists}\label{PrelimLLC}
Fix a connected quasi-split reductive group $G$ with finite central subgroup $Z$. We conclude our preliminary section by summarizing the rigid refined LLC as first given in \cite{kaletha16} (for local fields of characteristic zero) and extended in \cite{Dillery1}. First, we recall:

\begin{definition}\label{RITdef}
For a fixed inner twist $G \xrightarrow{\psi} G'$, a \mathdef{$Z$-rigid inner twist enriching $\psi$} is a pair $(\mathcal{T}, \bar{h})$, where $\mathcal{T} \in H^{1}(\mathcal{E}, Z \to G)$ and $\bar{h}$ is an isomorphism of (the $F$-descent of) $\mathcal{T} \times^{G} G_{\tn{ad}}$ with $T_{\psi}$, the $G_{\tn{ad}}$-torsor canonically associated to $\psi$ by taking the fiber over $\psi \in (\underline{\mathrm{Isom}}(G', G)/G_{\tn{ad}})(F)$ in $\underline{\mathrm{Isom}}(G', G)$. Every inner twist has such a (not necessarily unique) enrichment, which we will call a \textsf{rigidification}, for some sufficiently large $Z$. This last fact means that taking $G$ to be quasi-split still allows us to study all possible connected reductive $G'$ by taking rigidifications of quasi-split inner twists. 
\end{definition}

There is a notion of a morphism of rigid inner twists (see \cite[\S 7.1]{Dillery1}) such that $\mathrm{Aut}(\mathcal{T}, \bar{h})$ is canonically isomorphic to $G'(F)$. A \mathdef{representation} of a rigid inner twist $(\mathcal{T}, \bar{h})$ is a triple $((\mathcal{T}, \bar{h}),\pi)$, where $\pi$ is a smooth admissible representation of $G'(F)$ (on a $\mathbb{C}$-vector space), and a \mathdef{morphism} of two such representations is a map between the two rigid inner twists which maps one representation to the other. 

For a tempered $L$-parameter $L_{F} \xrightarrow{\varphi} \prescript{L}{}G$ we define $S_{\varphi}^{+}$ as the preimage of $Z_{\widehat{G}}(\varphi)$ in $\widehat{G/Z}$ and fix a Whittaker datum $\mathfrak{w}$ for $G$. 

\begin{conjecture}\label{rigidLLC1} There is a bijection (depending on $\mathfrak{w}$)
\begin{equation}\label{rigidLLCv1}
  \Pi_{\varphi}^{Z}  \xrightarrow{\iota_{\mathfrak{w}}} \tn{Irr}(\pi_{0}(S_{\varphi}^{+})),
\end{equation}
where $ \Pi_{\varphi}^{Z}$ is a set of isomorphism classes of tempered representations of $Z$-rigid inner twists $((\mathcal{T}, \bar{h}), \pi)$ of $G$. More precisely, we conjecture a commutative diagram
\[
\begin{tikzcd}
\Pi_{\varphi}^{Z} \arrow["\iota_{\mathfrak{w}}"]{r} \arrow{d} & \tn{Irr}(\pi_{0}(S_{\varphi}^{+})) \arrow{d} \\
H^{1}(\mathcal{E}, Z \to G) \arrow{r} & \pi_{0}(Z(\widehat{G/Z})^{+})^{*},
\end{tikzcd}
\]
where $Z(\widehat{G/Z})^{+}$ denotes the preimage of $Z(\widehat{G})^{\Gamma}$ in $\widehat{G/Z}$, the bottom map is a gerbe-theoretic analogue of the Tate-Nakayama isomorphism (see \cite[Theorem 4.11]{kaletha16} and \cite[Theorem 5.10]{Dillery1}), the left-hand column extracts the underlying torsor, and the right-hand column is induced by taking central characters. 
\end{conjecture}

The map $\iota_{\mathfrak{w}}$ is expected to satisfy many additional properties, such as the endoscopic character identities (cf. \cite[\S 7.4]{Dillery1}), which will be discussed at length in \S \ref{Endoscopy} and will be reviewed then (as well as in \S \ref{Iso}). There is also a version of Conjecture \ref{rigidLLC1} which uses \mathdef{extended pure inner twists} instead of rigid inner twists, which will be reviewed in \S \ref{Iso} when we compare the two versions. For $\rho \in \tn{Irr}(\pi_{0}(S_{\varphi}^{+}))$, write $\rho \in \tn{Irr}(\pi_{0}(S_{\varphi}^{+}), \mathcal{T})$ if its image in $H^{1}(\mathcal{E}, Z \to G)$ via the above diagram is $[\mathcal{T}]$.

\begin{remark}
It is also possible to give a version of the above conjectures in which the temperedness hypotheses are dropped, see \cite[\S 4]{DS23}. See also \cite[\S 7]{SZ} for a discussion of how the local Langlands correspondence for general $L$-parameters is related its tempered analogue.
\end{remark}


\section{Rigid refined LLC and disconnected centers}\label{Construction}
The goal of this section is to show that if one has maps $\iota_{\mathfrak{w}}$ satsifying Conjecture \ref{rigidLLCv1} for all connected reductive $G$ with connected center then (assuming a functorality conjecture, also in the connected center case, see \S\S \ref{rigidcomp}, \ref{ConstructionFunc}) there is a canonical construction of these maps for all connected, reductive $G$.

\subsection{Weak $z$-embeddings}\label{ConstructionEmbeddings}
Let $G$ be a connected, quasi-split reductive group over $F$ with fixed finite central subgroup $Z$. Recall the following notion from \cite{kaletha18}:

\begin{definition}\label{pseudodef} A \mathdef{pseudo $z$-embedding} is a connected reductive group $G_z$ defined over $F$ and an embedding $G \to G_{z}$ with normal image such that
\begin{enumerate}
\item{$G_{z}/G$ is a torus;}
\item{$H^{1}(F, G_{z}/G) = 0$;}
\item{The map $H^{1}(F, Z(G)) \to H^{1}(F, Z(G_{z}))$ is bijective.}
\end{enumerate}
We call it a \mathdef{$z$-embedding} if $Z(G_{z})$ is connected and $G_{z}/G$ is an induced torus.
\end{definition}

Unlike in the $p$-adic case, it is in general not possible to find a $z$-embedding for arbitrary $G$ in equal characteristic because of the third part of the definition. As mentioned in the introduction, an easy way to see this impossibility is that the group $H^{1}(F, Z(G))$ is in general infinite (which is not the case in mixed characteristic), whereas $H^{1}(F, Z(G_{z}))$ is always finite for $Z(G_{z})$ connected. Instead, in this paper we work with a subtly different notion. To define it we first observe that given a central embedding $G \hookrightarrow G_z$ 
and an inner twist $G \xrightarrow{\psi} G'$, we have a canonical inner twist of $G_z$. Indeed

\begin{lemma}\label{transfertwist} For any inner twist $G \xrightarrow{\psi} G'$ with corresponding $1$-cocycle $u$ of $\Gamma$ in $G_{\tn{ad}}(F^{s})$, there is an inner twist $G_{z} \xrightarrow{\psi_{z}} G'_{z}$ of $G_{z}$ with corresponding cocycle equal to $u$ such that the following diagram commutes
\[
\begin{tikzcd}
1 \arrow{r} & G \arrow["\psi"]{d} \arrow{r} & G_{z} \arrow["\psi_{z}"]{d} \arrow{r} &  G_{z}/G \arrow["\text{id}"]{d} \arrow{r} & 1 \\
1 \arrow{r} & G' \arrow{r} & G'_{z} \arrow{r} & G_{z}'/G' \arrow{r} & 1.
\end{tikzcd}
\]
\end{lemma}

\begin{proof}
As explained in the argument of \cite[Fact 5.9]{kaletha18}, $G'_{z}$ is defined as $G' \times^{Z(G)} Z(G_{z})$, where $Z(G)$ maps into $G'$ via $\psi$ (this descends to a morphism over $F$ because the differential of $\psi$ is inner). The embedding $G' \to G'_{z}$ is induced by the identity to the first component and the trivial map to the second component, and the map $\psi_{z}$ is induced by $\psi \times \tn{id}$.
\end{proof}
When we have a reductive group $G$ with inner twist  $(G', \psi)$, we always denote by $(G'_z, \psi_z)$ the inner twist of $G_z$ constructed in Lemma \ref{transfertwist}. We can now make the following definition.
\begin{definition}\label{weakdef}
     A \mathdef{weak $z$-embedding} is a connected reductive group $G_z$ defined over $F$ and an embedding $G \hookrightarrow G_z$ with normal image such that 
     \begin{enumerate}
         \item $G_z/G$ is a torus;
         \item $H^1(F, G_{z}/G)=1$;
         \item the natural map 
\begin{equation*}
    H^1(F, G') \rightarrow H^1(F, G'_z)
\end{equation*}
is bijective for every inner form $G'$ of $G$, where $G'_{z}$ is as in Lemma \ref{transfertwist}.
     \end{enumerate}
\end{definition}
We warn the reader that it is not enough to require condition (3) just for $G$ alone, since in general there is no identification of $H^{1}(F, G)$ with $H^{1}(F, G')$ (unless $G$ and $G'$ are pure inner forms of each other). We also note that the above definition makes sense for non quasi-split $G'$, and so we will use it in that more general context as well.
\begin{lemma}{\label{weaklem}}
    Condition (3) in Definition \ref{weakdef} is equivalent to
    \begin{enumerate}
        \item[$(3')$]  the natural map 
        \begin{equation*}
            H^{1}(\mathcal{E}, Z \to G') \to H^{1}(\mathcal{E}, Z \to G'_{z}),
        \end{equation*}
        is bijective for any (equivalently every) finite subgroup $Z \subset Z(G)$ and every inner twist $G'$.
    \end{enumerate}
\end{lemma}
\begin{proof}
Condition $(2)$ implies that $H^1(F, G') \to H^1(F, G'_z)$ is surjective, as is $H^{1}(\mathcal{E}, Z \to G') \to H^{1}(\mathcal{E}, Z \to G'_{z})$ for each $Z$. Indeed, the latter statement follows from the exact sequence
\begin{equation}{\label{gerbeLES}}
G'_z(F) \to (G/G_z)(F) \to H^{1}(\mathcal{E}, Z \to G') \to H^{1}(\mathcal{E}, Z \to G'_{z}) \to H^{1}(F, G/G_z).
\end{equation}

Hence, we just need to show that the injectivity of $H^{1}(\mathcal{E}, Z \to G') \to H^{1}(\mathcal{E}, Z \to G'_{z})$ is equivalent to that of $H^1(F,G') \to H^1(F, G'_z)$. But by equation \eqref{gerbeLES} (and a standard twisting argument) these are both equivalent to the surjectivity of $G'_z(F) \to (G/G_z)(F)$. 
\end{proof}
It is shown in \cite[\S 5]{kaletha18} that any $z$-embedding is a weak $z$-embedding. 

\begin{remark}\label{inducedremark} Definition \ref{weakdef} could be strengthened to insist that: \begin{itemize}
\item{$G_{z}/G$ is an induced torus (this will be the case for all of the embeddings in this paper used to construct the rigid refined LLC, but there will be other situations, for example when we need to interpolate between different systems of embeddings in Proposition \ref{canonLLC}, where this property does not hold);}
\item{$Z(G_{z})$ is connected; this will be the case for our embeddings until we deal with weak $z$-embeddings of endoscopic groups (in \S \ref{Endoscopy}).} 
\end{itemize}
\end{remark}

It will be useful to consider projective systems of weak $z$-embeddings (they will play the role of honest $z$-embeddings in the function field setting). Before constructing such a system, we follow \cite[Proposition 5.2]{kaletha18} to define a ``base weak $z$-embedding" that will serve as the blueprint for the system. The idea to use such an embedding is originally due to Kottwitz.

Let $G$ be a fixed connected reductive group. Let $(T_0, K)$ be such that
    \begin{enumerate}
        \item $T_0$ is an $F$-torus equipped with an embedding $Z(G) \to T_{0}$ defined over $F$
        \item $K/F$ is a finite Galois extension containing the splitting field $K_0$ of $T_0$.
    \end{enumerate}
We use the following notations:
\begin{itemize}
\item $C_0 = T_0/Z(G)$;
\item $C= \Res_{K/F}(C_{0,K})$;
\item $T=C \times_{C_0} T_0$.
\end{itemize}
 Note that by construction, we have a short exact sequence
 \begin{equation*}
     0 \rightarrow Z(G) \rightarrow T \to C \rightarrow 0.
 \end{equation*}
 Indeed, $Z(G)$ embeds into $T$ via the $T_0$ term in the product and when we quotient by $Z(G)$, we are left with $C \times_{C_0} C_0=C$. Then $H^{1}(F, T/Z(G)) = H^{1}(F, C) = 0$, since $C$ is an induced torus. Moreover, the connecting homomorphism $C(F) \to H^{1}(F, Z(G))$ factors through the norm map $C(F) \to C_{0}(F)$; to see this, apply the functoriality of the long exact sequence in group cohomology to the commutative diagram
 \[
 \begin{tikzcd}
0 \arrow{r} & Z(G) \arrow{r} \arrow["\tn{id}"]{d} & T \arrow{d} \arrow{r} & C \arrow["N_{K/K_{0}}"]{d} \arrow{r} & 0 \\
0 \arrow{r} & Z(G) \arrow{r} & T_{0} \arrow{r} & C_{0} \arrow{r} & 0,
 \end{tikzcd}
 \]
 where by $N_{K/K_{0}}$ we mean the map induced by (a product of) the field norm map after fixing a $K_{0}$-splitting of $T_{0}$ which determines $K_{0}$-splittings of $C_{0}$ and $C$.

\begin{definition}{\label{weakembeddingpair}}
Let $(T_0, K)$ be a pair as in the above paragraph. We say that it is a \mathdef{weak embedding pair} if the composition
 \begin{equation*}
(K^{\times})^{n} \xrightarrow{\sim} C(F) \xrightarrow{N_{K/F}} C_0(F) \rightarrow  H^{1}(F, Z(G)) \rightarrow H^{1}(F, G'),
 \end{equation*}
 is trivial for every inner twist $G'$ of $G$ (where the left-hand map is given by our choice of splitting and $n = \mathrm{rk}(C)$).
\end{definition}
\begin{remark}
Given any $T_{0}$ as above it is easy to find a $K$ such that $(T_{0},K)$ is a weak embedding pair (and, by the above discussion, so is any $(T_{0}, K')$ for $K \subset K'$), as follows. First, note that it's enough to show this for any fixed inner twist $G'$, since if we have such a pair for each twist we can take the compositum of the fields (there are finitely many, since $H^{1}(F, G_{\text{ad}})$ is finite). Choosing a maximal torus $S$ of $G'$, we note that the connecting homomorphism $C_{0}(F) \to H^{1}(F, G')$ factors as
\begin{equation}
C_{0}(F) \to H^{1}(F, Z(G)) \to H^{1}(F, S) 
\to H^{1}(F, G'),
\end{equation}
where now the connecting homomorphism to $H^{1}(F, S)$ is a group homomorphism. After splitting the tori over $K_{0}$, the relevant map is identified with the composition 
\begin{equation}\label{weakpairexists}
(K^{\times})^{n} \xrightarrow{N_{K/K_{0}}} (K_{0}^{\times})^{n} \to H^{1}(F, S) \to H^{1}(F, G').
\end{equation}

From here, the desired result follows from the fact that $H^{1}(F,S)$ is finite, so that each of the $n$ maps $K_{0}^{\times} \to H^{1}(F, S)$ has open kernel, and hence by local class field theory we may choose sufficiently large finite $K$ such that the image of the first map in \eqref{weakpairexists} lands in the kernel of the second.

\end{remark}
\begin{construction}{\label{weakexist}}
Let $G$ be a connected reductive group and $(T_0, K)$ a weak embedding pair. There exists a natural weak $z$-embedding $G \hookrightarrow G_z$ associated to $(T_0,K)$.
\end{construction}
\begin{proof}
We define $G_{z}$ to be the push-out of the maps $Z(G) \to G$ and $Z(G) \to T$. It is shown in the proof of \cite[Corollary 5.3]{kaletha18} that the induced maps $G \to G_{z}$ and $T \to G_{z}$ are injective and the latter identifies $T$ with $Z(G_{z})$, and that $\mathrm{Coker}(Z(G) \to T) = C$. 

To prove that $G \to G_{z}$ is a weak $z$-embedding,  we need only show that the induced map
\begin{equation*}
    H^{1}(F, G') \to H^{1}(F, G'_{z}),
\end{equation*}
is an isomorphism for each inner twist $G'$ of $G$. It is automatically surjective since by construction $G'_{z}/G' = C$, which is an induced torus. The exact sequence \eqref{gerbeLES} (and twisting) gives injectivity, since the connecting homomorphism factors through the composition 
\begin{equation*}
C(F) \to H^{1}(F, Z(G')) \to H^{1}(F, G'),
\end{equation*}
which by construction has trivial image.
\end{proof}

We can now construct our projective system of weak $z$-embeddings.

\begin{proposition}\label{weaksystemexist}
Fix an $F$-torus $T_0$. Let $\{K_i\}_{i \geq 1}$ be a cofinal system of finite Galois extensions of $F$ such that each $(T_0, K_i)$ is a weak embedding pair (find an initial $K_{1}$ such that $(T_{0}, K_{1})$ is such a pair and then take any such system of extensions which all contain $K_{1}$). Each pair $(T_0, K_i)$ gives a weak $z$-embedding $G \rightarrow G^{(i)}_z$ as in Construction \ref{weakexist} and  each map $K_i \rightarrow K_j$ induces a map $G^{(j)}_z \to G^{(i)}_z$ that is a central extension by a torus. 
\end{proposition}
The above Proposition gives a projective system  $\{G \to G^{(i)}_z\}$ of weak $z$-embeddings of $G$ (sometimes we just write $\{G^{(i)}_z\}$ if the embedding of $G$ is understood).
\begin{proof}
    For the pair $(T_0, K_i)$, we adopt the notation $C^{(i)}$, $T^{(i)}$, and $G^{(i)}_z$ for the objects $C$, $T$, and $G_z$ associated to this pair. 

    We now construct, for $K_i \rightarrow K_j$ a morphism 
    \begin{equation*}
        G^{(j)}_z \rightarrow G^{(i)}_z.
    \end{equation*}
   It suffices to construct a map $C^{(j)} \rightarrow C^{(i)}$, which we take to be $N_{K_j/K_i}$. Note that this is compatible with the embeddings of $G$, since by construction $G$ maps into $G_{z}^{(i)} = G \times^{Z(G)} T^{(i)}$ by $\tn{id} \times \tn{1}$ and our transition map is defined on the second factor. 

   The kernel of $G^{(j)}_z \rightarrow G^{(i)}_z$ is canonically identified with $\tn{Ker}[T^{(i+1)} \to T^{(i)}]$ and hence is a central extension by a torus (it's easy to see that this aforementioned kernel is the subtorus of elements killed by the $K_{i+1}/K_{i}$-norm). 
\end{proof}

Consider the short exact sequence
\begin{equation}\label{modifiedSES1}
    0 \to Z(G) \to G' \times Z((G')_{z}^{(i)}) \to (G')_{z}^{(i)} \to 0.
\end{equation}
A crucial property of the system $\{(G')_{z}^{(i)}\}$ is:

\begin{lemma}\label{shrinkingimage} The intersection across all $j$ of the image of the connecting homomorphisms $(G')_{z}^{(j)}(F) \to H^{1}(F, Z(G))$ induced by the short exact sequence \eqref{modifiedSES1} is zero. 
\end{lemma}

\begin{proof} To simplify notation we work with $G$ instead of $G'$, since the argument is the same. 
The connecting homomorphism $G_{z}^{(j)}(F) \to H^{1}(F, Z(G))$ factors as the composition $G_{z}^{(j)}(F) \to C^{(j)}(F) \xrightarrow{\delta} H^{1}(F, Z(G))$ due to the commutative diagram
\[
\begin{tikzcd}
0 \arrow{r} & Z(G) \arrow["\Delta"]{r} \arrow["\tn{id}"]{d} & G \times  Z(G_{z}^{(j)}) \arrow{d} \arrow{r} & G_{z}^{(j)} \arrow{d} \arrow{r} & 1 \\
0 \arrow{r} & Z(G) \arrow{r} & Z(G_{z}^{(j)}) \arrow{r} & C^{(j)} \arrow{r} & 0
\end{tikzcd}
\]
and so we deduce the result from the fact that the intersection of the images of the connecting homomorphisms $C^{(j)}(F) \to H^{1}(F, Z(G))$ is zero, by construction of our system $\{G_{z}^{(j)}\}$ (cf. the proof of Construction \ref{weakexist}).
\end{proof}

The above result motivates the following definition:

\begin{definition}\label{goodsysdef} Let $H$ be a connected reductive group, and suppose we have a projective system of compatible weak $z$-embeddings $\{H \to H_{z}^{(i)}\}_{i \geq 0}$. We say that it is a \mathdef{good $z$-embedding system} if:
\begin{enumerate}
\item{Each transition map $H_{z}^{(i+1)} \to H_{z}^{(i)}$ is surjective (note that the compatibility condition on the system already implies that the kernel of each transition map is central);}
\item{The limit connecting homomorphism
\begin{equation*}
   \varprojlim_{i} C^{(j)}(F) \to H^{1}(F, Z(H)) 
\end{equation*}
 is trivial.}
\end{enumerate}
\end{definition}

In particular, any system of weak $z$-embeddings constructed using Proposition \ref{weaksystemexist} is a good $z$-embedding system, and any good $z$-embedding system satisfies Lemma \ref{shrinkingimage}. 

\begin{remark}\label{shrinkingimagebis} As in the discussion preceding Definition \ref{weakembeddingpair}, each connecting homomorphism $C^{(j)}(F) \to H^{1}(F, Z(H))$ factors through the composition $C^{(j)}(F) \to C^{(i)}(F) \to H^{1}(F, Z(H))$ for any $i \leq j$ due to the commutative diagram
 \begin{equation}\label{connhomfacdiag}
 \begin{tikzcd}
0 \arrow{r} & Z(H) \arrow{r} \arrow["\tn{id}"]{d} & Z(H_{z}^{(j)}) \arrow{d} \arrow{r} & C^{(j)} \arrow{d} \arrow{r} & 0 \\
0 \arrow{r} & Z(H) \arrow{r} & Z(H_{z}^{(i)}) \arrow{r} & C^{(i)} \arrow{r} & 0,
 \end{tikzcd}
 \end{equation}
 and functoriality of the long exact sequence in cohomology. This implies that, for any fixed $i$, the intersection over all $j \geq i$ of the image of $C^{(j)}(F)$ in $C^{(i)}(F)$ is contained in the image of $Z(H_{z}^{(i)})(F)$.
\end{remark}





\subsection{Dual constructions}\label{ConstructionDual}
Recall that $W_{F}$ denotes the Weil group of $F$ and $L_{F} = W_{F} \times \SL_{2}(\mathbb{C})$ the Weil-Deligne group. For an arbitrary weak $z$-embedding $G \to G_{z}$ we have the dual exact sequence of groups
\begin{equation}\label{dualSES}
1 \to \widehat{C} \to \widehat{G_{z}} \to \widehat{G} \to 1.
\end{equation}

Denote by $\Psi(G)$ the set of $\widehat{G}$-conjugacy classes of $L$-parameters $L_{F} \to \co{L}{G}$, which we can also view as a subset of $H^{1}(L_{F}, \widehat{G})$ via the projection $\co{L}{G}\to \widehat{G}$. If $\varphi_{z}$ is an $L$-parameter for $\prescript{L}{}G_{z}$ then the map from \eqref{dualSES} gives an $L$-parameter $\varphi$ for $^{L}G$; we'll say in this case that $\varphi_{z}$ \textsf{lifts} $\varphi$ and write $\varphi_{z} \mapsto \varphi$.

The following result says that we can always lift $L$-parameters to weak $z$-embeddings:

\begin{proposition}\label{liftparam} For any $\varphi \in \Psi(G)$ there is a $\varphi_{z} \in \Psi(G_{z})$ with $\varphi_{z} \mapsto \varphi$. Moreover, if $\varphi(W_{F})$ is bounded then we can also choose $\varphi_{z}$ such that $\varphi_{z}(W_{F})$ is bounded.
\end{proposition}

\begin{proof} The proof will adapt the argument used to prove \cite[Corollary 5.13]{kaletha18}. For $\varphi \in \Psi(G)$ denote by $\varphi_{0}$ its restriction to $W_{F}$. We first lift $\varphi_{0}$ to semisimple $\varphi_{0,z}$ such that $\varphi_{0,z}(W_{F})$ is bounded when $\varphi_{0}(W_{F})$ is. This part of the argument is unchanged from the pseudo $z$-embedding case (which is the setting of \cite[Corollary 5.13]{kaletha18}), using the duality results for crossed modules from Appendix \ref{Braided}. 

The long exact sequence of $W_{F}$-modules for the short exact sequence \eqref{dualSES} gives the commutative diagram with exact rows
\[
\begin{tikzcd}
H^{1}(W_{F}, \widehat{C}) \arrow{r} \arrow["\text{id}"]{d} & H^{1}(W_{F}, \widehat{G_{z}}) \arrow{r} \arrow{d} & H^{1}(W_{F}, \widehat{G}) \arrow{r} \arrow{d} & H^{2}(W_{F}, \widehat{C}) \arrow["\text{id}"]{d} \\
H^{1}(W_{F}, \widehat{C}) \arrow{r} & H^{1}(W_{F}, \widehat{G}_{\text{sc}} \to \widehat{G_{z}}) \arrow{r} &  H^{1}(W_{F}, \widehat{G}_{\text{sc}} \to \widehat{G}) \arrow{r} & H^{2}(W_{F}, \widehat{C}),
\end{tikzcd}
\] 
where $H^{1}(\Gamma', H' \to H)$ denotes the cohomology of the $\Gamma'$-braided crossed module $H' \to H$ (see \S \ref{BraidedCrossed} for an exposition of such objects) for a (topological) group $\Gamma'$. To extend $\varphi_{0}$ to $\varphi_{0,z}$, it's enough to find an element $\varphi' \in H^{1}(W_{F}, \widehat{G}_{\text{sc}} \to \widehat{G_{z}})$ whose image in $H^{1}(W_{F}, \widehat{G}_{\text{sc}} \to \widehat{G})$ coincides with the image of $\varphi_{0}$.

It follows from Proposition \ref{appdual1} that $H^{1}(W_{F}, \widehat{G}_{\text{sc}} \to \widehat{G})$ is canonically isomorphic to $Z(G)(F)^{D}$. As explained in the proof of \cite[Corollary 5.13]{kaletha18}, the unitary characters correspond to the elements of $H^{1}(W_{F}, \widehat{G}_{\text{sc}} \to \widehat{G})$ whose image in $H^{1}(W_{F}, \text{Coker}(\widehat{G}_{\text{sc}} \to \widehat{G}))$ is bounded. From here, the identical argument loc. cit. produces such an element $\varphi'$ and shows that the corresponding parameter $\varphi_{0,z}$ has bounded image when $\varphi_{0}$ does. 

We now want to extend $\varphi_{0,z}$ to an admissible cocycle of $L_{F}$. The restriction of $\varphi$ to $\SL_{2}$ is a homomorphism of algebraic groups 
\begin{equation*}
    \SL_{2} \to  (S_{\varphi}^{\circ})_{\text{der}},
\end{equation*}
and we want to lift this to an algebraic homomorphism
\begin{equation*}
    \SL_{2} \to (S_{\varphi_{0,z}}^{\circ})_{\text{der}}.
\end{equation*}
To find such a lift, it is enough to show that the natural map $(S_{\varphi_{0,z}}^{\circ})_{\tn{der}} \to (S_{\varphi_{0}}^{\circ})_{\tn{der}} $ is surjective; indeed, this means that it's an isogeny, and so the morphism lifts because $\SL_{2}$ is simply connected (see, e.g., \cite[Proposition 9.3.2]{Conrad1}).

Via taking the long exact sequence for \eqref{dualSES} associated to the $\varphi_{0,z}$-twisted $L_{F}$-action we obtain 
\begin{equation}\label{twistLES}
    0 \to \widehat{C}^{\Gamma} \to S_{\varphi_{0,z}} \to S_{\varphi_{0}} \to H^{1}(L_{F}, \widehat{C}) \to H^{1}(L_{F}, \varphi_{0,z}, \widehat{G_{z}}).
\end{equation}
Note that the last degree-zero map need not be surjective, so this is where the argument in the proof of the analogous result in \cite{kaletha18} does not work in our situation. However, we have the chain of canonical identifications 
\begin{equation*}
    H^{1}(L_{F}, \widehat{C}) = H^{1}(W_{F}, \widehat{C}) \xrightarrow{\sim} H^{1}(W_{F}, 1 \to \widehat{C}) \xrightarrow{\sim} C(F)^{D}. 
\end{equation*}
Moreover, the group $C(F)^{D}$ has a natural topology and the map induced by the above identification and the connecting homomorphism 
\begin{equation}\label{SL2conn}
    S_{\varphi_{0}} \to H^{1}(L_{F}, \widehat{C}) \xrightarrow{\sim} C(F)^{D}
    \end{equation}
is continuous (the group $H^{1}(W_{F}, \widehat{C})$ has a natural topology making the connecting homomorphism continuous, cf. \cite[\S A.2]{Dillery2}, and the identification of this group with $C(F)^{D}$ is a homeomorphism). 

We note that anything in the image of \eqref{SL2conn} is trivial on the image of $Z(G_{z})(F)$ in $C(F)$, using the commutative diagram
\begin{equation}\label{dualdiagramS3}
\begin{tikzcd}
H^{1}(L_{F}, \varphi_{0,z}, \widehat{G_{z}}) \to \Psi(G_{z}) \arrow{r} & H^{1}(W_{F}, \widehat{G}_{\text{sc}} \to \widehat{G_{z}}) \arrow["\sim"]{r} & Z(G_{z})(F)^{D} \\
H^{1}(W_{F}, \widehat{C}) \arrow{u} \arrow["\sim"]{r} & H^{1}(W_{F}, 1 \to \widehat{C}) \arrow{u} \arrow["\sim"]{r} & C(F)^{D} \arrow{u} \\
S_{\varphi_{0}}; \arrow{u} & & 
\end{tikzcd}
\end{equation}
the top-right isomorphism is again from Proposition \ref{appdual1}, the first column is from the long exact sequence \eqref{twistLES}, and the top-left map is the composition 
\begin{equation*}
    H^{1}(L_{F}, \varphi_{0,z}, \widehat{G_{z}}) \xrightarrow{\sim}  H^{1}(L_{F},\widehat{G_{z}}) \to \Psi(G_{z}), 
\end{equation*}
where the first map is the standard twisting identification from non-abelian cohomology and the second map sends a cocycle to the associated parameter. 

It follows that any character in the image \eqref{SL2conn} factors through the following quotient of $C(F)$:
\begin{equation}\label{characterquot}
  \overline{C(F)} := \frac{ C(F)}{\text{im}[Z(G_{z})(F) \to C(F)]} \hookrightarrow H^{1}(F, Z(G)),
\end{equation}
which, in view of \eqref{characterquot}, is profinite (since  $H^{1}(F, Z(G))$ is and the inclusion in \eqref{characterquot} is a closed immersion). 

We deduce from the above discussion that the map \eqref{SL2conn} factors through the subgroup
\begin{equation*}
    \overline{C(F)}^{D} \hookrightarrow C(F)^{D},
\end{equation*}
and that $\overline{C(F)}^{D}$ is the Pontryagin dual of a profinite group, and is therefore totally disconnected. 

It follows that the image of $S_{\varphi_{0}}^{\circ}$ in $H^{1}(W_{F}, \widehat{C})$ is trivial, and so it lies in the image of $S_{\varphi_{0,z}}$. Since $\widehat{C}^{\Gamma}$ is connected (because $H^{1}(F, C) = 1$, so that $\pi_{0}(\widehat{C}^{\Gamma})$ is trivial by Kottwitz's duality result \cite{Kottwitz86}), we see that the induced map $\pi_{0}(S_{\varphi_{0,z}}) \to \pi_{0}(S_{\varphi_{0}})$ is injective, and therefore $S_{\varphi_{0,z}}^{\circ}$ is the preimage of $S_{\varphi}^{\circ}$ and the natural map $S_{\varphi_{0,z}}^{\circ} \to S_{\varphi}^{\circ}$ is surjective as well. Because this map restricts to a morphism $Z(S_{\varphi_{0,z}}^{\circ}) \to Z(S_{\varphi}^{\circ})$ it remains surjective after passing to derived groups (the map $(S_{\varphi_{0,z}}^{\circ})_{\tn{der}} \times Z(S_{\varphi_{0,z}}^{\circ}) \to S_{\varphi_{0}}^{\circ}$ is surjective, and one only needs the left-hand factor to surject onto commutators), 
giving the desired result. 
\end{proof}

We now fix a good $z$-embedding system $\{G \to G_{z}^{(i)}\}$ (not necessarily of the type constructed in Proposition \ref{weaksystemexist}). In the above proof we mentioned that if $\varphi_{z}$ lifts $\varphi$ then the induced map $S_{\varphi_{z} }\to S_{\varphi}$ need not be surjective (and this is also the case after applying $``+"$ and/or taking $\pi_{0}$). However, the following result (and its corollary) shows that this will be true for our fixed system $\{G_{z}^{(i)}\}$ as long as $i$ is sufficiently large. Following the notation of \S \ref{ConstructionEmbeddings}, we denote $G_{z}^{(i)}/G$ by $C^{(i)}$.

\begin{proposition}\label{limsurj} Given a compatible system of parameters $\varphi_{z}^{(i)}$ for $^{L}G_{z}^{(i)}$ which induce some fixed parameter $\varphi$ for $^{L}G$, the natural map
$$\varinjlim_{i} S_{\varphi_{z}^{(i)}} \to S_{\varphi}$$ is surjective (in the category of abstract abelian groups). 
\end{proposition}

\begin{proof} For any $j$ we have the commutative diagram
\[
\begin{tikzcd}
S_{\varphi_{z}^{(1)}} \arrow{d} \arrow{r} & S_{\varphi} \arrow["\text{id}"]{d} \arrow{r} & H^{1}(L_{F}, \widehat{C^{(1)}})\arrow["\sim"]{r} \arrow{d} & H^{1}(W_{F}, 1 \to \widehat{C^{(1)}}) \arrow{d} \arrow["\sim"]{r} & C^{(1)}(F)^{D} \arrow{d} \\
S_{\varphi_{z}^{(j)}} \arrow{r} & S_{\varphi} \arrow{r} & H^{1}(L_{F}, \widehat{C^{(j)}})\arrow["\sim"]{r} & H^{1}(W_{F}, 1 \to \widehat{C^{(j)}}) \arrow["\sim"]{r} & C^{(j)}(F)^{D}, 
\end{tikzcd}
\]
where the right-most horizontal map is via the duality pairing for $W_{F}$-crossed modules from Proposition \ref{appdual1}.

Fix $x \in S_{\varphi}$ and denote by $\chi_{x}$ the character of $C^{(1)}(F)$ corresponding to $x$ as above. The same arguments as in the proof of Proposition \ref{liftparam} show that the map 
\begin{equation}\label{Ssurjmap}
    S_{\varphi} \to C^{(1)}(F)^{D}
\end{equation}
from the top line of the above diagram factors through the subgroup $\overline{C^{1}(F)}^{D}$, where as above
\begin{equation*}
  \overline{C^{(1)}(F)} := \frac{ C^{(1)}(F)}{\text{im}[Z(G_{z}^{(1)})(F) \to C^{(1)}(F)]} \hookrightarrow H^{1}(F, Z(G)),
\end{equation*}
which is profinite. Since $\chi_{x}$ has open kernel we deduce that $\chi_{x}$ has finite image.

It follows by Remark \ref{shrinkingimagebis} that the image $C^{(j)}(F) \to C^{(1)}(F)$ is eventually contained in the kernel of $\chi_{x}$, and thus that the image of $\chi_{x}$ is zero in $C^{(k)}(F)^{D}$ for all $k \gg 0$. The same is thus true for the image of $x$ in $H^{1}(W_{F}, \widehat{C^{(k)}})$ for $k \gg 0$, giving the result.
\end{proof}

We can now deduce the eventual surjectivity of $S_{\varphi_{z}^{(i)}} \to S_{\varphi}$ as well as some related results that will be useful for our comparison of the LLC:

\begin{corollary}\label{keycor} In the setting of the previous Proposition, there is some $j \gg 0$ such that $S_{\varphi_{z}^{(j)}} \to S_{\varphi}$ is surjective. In such a case, the natural map $S_{\varphi_{z}^{(j)}}^{+} \to S_{\varphi}^{+}$ is also surjective, and after applying $\pi_{0}$ it's an isomorphism.
\end{corollary}

\begin{proof} Each morphism $S_{\varphi_{0,z}^{(j)}} \to S_{\varphi_{0}}$ is a morphism of algebraic groups, and hence has closed (in particular, constructible) image. One can give $S_{\varphi_{0}}$, a spectral space, the constructible topology, which is quasi-compact and has the property that every constructible subset is open and closed (the constructible topology is, by definition, the coarsest topology with this property). The images of each of the above homomorphisms thus gives an open cover, which has a finite subcover, giving the first claim. 

For $j$ giving surjectivity, the fact that $S_{\varphi_{z}^{(j)}}^{+} \to S_{\varphi}^{+}$ is surjective follows easily from the surjectivity of $\widehat{G_{z}/Z} \to \widehat{G/Z}$, but we give the argument for completeness. Fix $y \in S_{\varphi}^{+}$ with image $\bar{y} \in S_{\varphi}$ and $\bar{y}_{z} \in S_{\varphi_{z}^{(j)}}$ a preimage of $\bar{y}$; we may take a preimage $y_{z}'$ in $\widehat{G_{z}^{(j)}/Z}$ of $\bar{y}_{z}$, and by construction its image in $\widehat{G/Z}$ differs from $y$ by an element $c \in \widehat{Z} \hookrightarrow \widehat{G/Z}$. Since the embedding of $\widehat{Z}$ into $\widehat{G/Z}$ is compatible with its analogous embedding in $\widehat{G_{z}^{(j)}/Z}$ (in particular, it has trivial intersection with $\widehat{C^{(j)}}$), we may lift $c$ to $\widehat{G_{z}^{(j)}/Z}$ (denoting the lift also by $c$) and replace $y'_{z}$ by $y_{z} := y'_{z} \cdot c$, which lies in $S_{\varphi_{z}^{(i)}}^{+}$ and is the desired preimage. 

The final claim follows from applying the right-exact functor $\pi_{0}$ to the short exact sequence
$$0 \to \widehat{C^{(j)}}^{\Gamma} \to S_{\varphi_{z}^{(j)}}^{+} \to S_{\varphi}^{+} \to 0$$ (the only subtle part of this short exact sequence is surjectivity, which comes from the second statement), noting that $\widehat{C^{(j)}}^{\Gamma}$ is connected.
\end{proof}

\subsection{Construction of rigid refined LLC}\label{rigidcomp}
Continuing with $G$ and a good $z$-embedding system $\{G \to G_{z}^{(i)}\}$ as above, we show that one can construct the rigid refined LLC for $G$ using its analogue for each $G_{z}^{(i)}$ (which, in view of Proposition \ref{weaksystemexist} can be taken to have connected center).

\begin{remark}
It may seem more natural and easier at this stage to take a good $z$-embedding system $\{G \to G_{z}^{(i)}\}$ from Proposition \ref{weaksystemexist}. However, later in the paper we will be working with projective systems of endoscopic groups which will be good $z$-embedding systems that do not arise as in the aforementioned proposition.
\end{remark}

Recall the rigid refined LLC (Conjecture \ref{rigidLLCv1}) which predicts, for a fixed Whittaker datum $\mathfrak{w}$ of $G$, finite central subgroup $Z$, and tempered $L$-parameter $L_{F} \to \prescript{L}{}G$, that there is a bijection 
\begin{equation}\label{rigidLLC}
  \Pi_{\varphi}^{Z}  \xrightarrow{\iota_{\mathfrak{w}}} \tn{Irr}(\pi_{0}(S_{\varphi}^{+})),
\end{equation}
where $ \Pi_{\varphi}^{Z} \subset \Pi^{\tn{temp},Z}$ is the set of isomorphism classes of representations of $Z$-rigid inner twists $((\mathcal{T}, \bar{h}), \pi)$ of $G$ such that $\pi \in \Pi_{\varphi}(G^{\mathcal{T}}),$ where $G^{\mathcal{T}}$ denotes the inner form of $G$ corresponding to $\mathcal{T}$. When we want to emphasize the group $G$, we write the above map as $\iota_{\mathfrak{w},G}$. 

To ensure that our construction of rigid refined LLC is well-defined, we need some standard functoriality assumptions about the rigid refined LLC. The relevant setting is a morphism $G_{1} \xrightarrow{\eta} G_{2}$ of quasi-split connected reductive groups with central kernel and abelian cokernel (in particular, $\eta$ induces an isomorphism on adjoint groups). Note that any Whittaker datum $\mathfrak{w}$ for $G_{2}$ induces, via $\eta$, a canonical Whittaker datum for $G_{1}$ (because the derived subgroups are isogenous), which we will also denote by $\mathfrak{w}$. 

A $Z$-rigid inner twist $(\mathcal{T}, \bar{h})$ for $G_{1}$ gives rise to a $Z' := \eta(Z)$-rigid inner twist $(\mathcal{T}' := \mathcal{T} ^{\times G_{1}} G_{2}, \bar{h}')$ of $G_{2}$. Fix such a rigid inner twist $(\mathcal{T}, \bar{h})$ (thus determining $(\mathcal{T}', \bar{h}')$) and denote the corresponding inner forms by $G_{1}'$, $G_{2}'$. Moreover, letting $\psi$ and $\psi'$ be the twisting isomorphisms associated to $(\mathcal{T}, \bar{h})$ and $(\mathcal{T}', \bar{h}')$, we see that $\psi' \circ \eta \circ \psi^{-1}$ descends to a morphism $G_{1}' \xrightarrow{\eta'} G_{2}'$ satisfying the same properties as the map $\eta$. For simplicity of exposition, we assume that the induced map $H^{1}(\mathcal{E}, Z \to G_{1}) \xrightarrow{\eta} H^{1}(\mathcal{E}, Z' \to G_{2})$ is an isomorphism for any $Z$ (this will be the case for all of our applications in this paper). For a representation $\dot{\pi}$ of a rigid inner twist enriching $G_{2}'$, we will abusively use $\dot{\pi} \circ \eta'$ to denote the pullback of the underlying representation of $G_{2}'(F)$ to $G_{1}'(F)$. 

\begin{conjecture}\label{Maartenbis}
If $G_{1} \xrightarrow{\eta} G_{2}$, $(\mathcal{T}, \bar{h})$, $\mathfrak{w}$ are as above and have connected centers, then for a fixed tempered parameter $\varphi$ for $G_{2}$ and $\rho \in \tn{Irr}(\pi_{0}(S_{\varphi}^{+}), \mathcal{T}')$ such that $\iota_{\mathfrak{w},G_{2}}^{-1}(\varphi, \rho)\circ \eta'$ is irreducible, 
 the representation $((\mathcal{T}, \bar{h}), \iota_{\mathfrak{w},G_{2}}^{-1}(\varphi, \rho)\circ \eta')$ lies in the compound $L$-packet $\Pi^{Z'}_{\prescript{L}{}\eta \circ \varphi}$. 
 Moreover, when the induced map 
\begin{equation}\label{Maartenmap} \pi_{0}(S_{\varphi}^{+}) \xrightarrow{\prescript{L}{}\eta} \pi_{0}(S^{+}_{\prescript{L}{}\eta \circ \varphi})
\end{equation}
is bijective, we have (as representations of rigid inner twists)
\begin{equation*}
   ((\mathcal{T}, \bar{h}), \iota_{\mathfrak{w},G_{2}}^{-1}(\varphi, \rho)\circ \eta') = \iota_{\mathfrak{w},G_{1}}^{-1}(\prescript{L}{}\eta \circ \varphi, \rho \circ \prescript{L}{}\eta). 
\end{equation*}
\end{conjecture}

The above functoriality Conjecture is weaker than \cite[Conjecture 2.15]{KalethaICM}, which is already known in many cases, as explained loc. cit. 
In fact, for our purposes, Conjecture \ref{Maartenbis} (for groups with connected center) can, assuming the compatibility of the rigid refined LLC and the Fargues-Scholze LLC (cf. \S \ref{IsoFunc}), be weakened to the following statement, which is an analogous functoriality assumption as made in \cite[\S 5.2]{kaletha18}:

\begin{conjecture}\label{Maartenbisbis} If $G_{1} \xrightarrow{\eta} G_{2}$, $(\mathcal{T}, \bar{h})$, and $\mathfrak{w}$ are as in Conjecture \ref{Maartenbis} (in particular, they have connected centers), then for a fixed tempered parameter $\varphi$ for $G_{2}$, $\rho \in \tn{Irr}(\pi_{0}(S_{\varphi}^{+}), \mathcal{T}')$ such that $\iota_{\mathfrak{w},G_{2}}^{-1}(\varphi, \rho)\circ \eta'$ is irreducible with $L$-parameter $\tilde{\varphi}$ which lifts to an $L$-parameter for $^{L}G_{2}$  
and the induced map
\begin{equation*} \pi_{0}(S_{\varphi}^{+}) \xrightarrow{\prescript{L}{}\eta} \pi_{0}(S^{+}_{\prescript{L}{}\eta \circ \varphi})
\end{equation*}
is bijective, we have (as representations of rigid inner twists)
\begin{equation*}
   ((\mathcal{T}, \bar{h}), \iota_{\mathfrak{w},G_{2}}^{-1}(\varphi, \rho)\circ \eta') = \iota_{\mathfrak{w},G_{1}}^{-1}(\prescript{L}{}\eta \circ \varphi, \rho \circ \prescript{L}{}\eta). 
\end{equation*}
\end{conjecture}

\begin{remark}\label{weakremark}
An analogous parameter-lifting assumption can be placed in Conjecture \ref{Maartenbis} to weaken it as follows: Assume in addition to the previous hypotheses that the parameter $\tilde{\varphi}$ for $^{L}G_{1}$ corresponding to $\iota_{\mathfrak{w},G_{2}}^{-1}(\varphi, \rho)\circ \eta'$  lifts to an $L$-parameter for $^{L}G_{2}$ up to modifying by an element of $H^{1}(W_{F}, Z(\widehat{G}_{1}))$. Then $((\mathcal{T}, \bar{h}), \iota_{\mathfrak{w},G_{2}}^{-1}(\varphi, \rho)\circ \eta')$ lies in the compound $L$-packet $\Pi^{Z'}_{(\prescript{L}{}\eta \circ \varphi) \cdot \hat{z}}$ for some $\hat{z} \in H^{1}(W_{F}, Z(\widehat{G}_{1}))$ (note that $(\prescript{L}{}\eta \circ \varphi) \cdot \hat{z}$ is necessarily tempered). 
 Moreover, when the induced map 
\begin{equation} \pi_{0}(S_{\varphi}^{+}) \xrightarrow{\prescript{L}{}\eta}  \pi_{0}(S^{+}_{\prescript{L}{}\eta \circ \varphi}) = \pi_{0}(S^{+}_{(\prescript{L}{}\eta \circ \varphi) \cdot \hat{z}})
\end{equation}
is bijective, we have (as representations of rigid inner twists)
\begin{equation*}
   ((\mathcal{T}, \bar{h}), \iota_{\mathfrak{w},G_{2}}^{-1}(\varphi, \rho) \circ \eta') = \iota_{\mathfrak{w},G_{1}}^{-1}((\prescript{L}{}\eta \circ \varphi) \cdot \hat{z}, \rho \circ \prescript{L}{}\eta).
\end{equation*}

What's more, if there are compatible weak $z$-embeddings $G \to G_{i}$ and the twisted lift of $\tilde{\varphi}$ to $^{L}G_{2}$ induces the same parameter for $^{L}G$ as $\tilde{\varphi}$, then the parameter for $^{L}G$ induced by $\tilde{\varphi}$ is the same as the one induced by $\varphi$. We choose not to use this weakening for clarity of exposition, but will indicate where and how certain arguments can be modified for this weaker statement (note also that, if one assumes compatibility of Fargues-Scholze and the rigid refined LLC, then this remark is unnecessary, since we can just use Conjecture \ref{Maartenbisbis}).
\end{remark}

We postpone the proof that Conjecture \ref{Maartenbisbis} suffices for our purposes to Corollary \ref{conjsuffices} (the argument does not rely on any constructions using Conjecture \ref{Maartenbis}, so there is no danger of circularity). The reader that does not want to assume compatibility of rigid refined LLC with Fargues-Scholze should assume Conjecture \ref{Maartenbis} for this paper.

We will now construct the rigid refined LLC for (quasi-split) $G$ assuming: \begin{enumerate}
\item{The rigid refined LLC for all groups in the system $\{G_{z}^{(i)}\};$}
\item{Conjecture \ref{Maartenbis} (or \ref{Maartenbisbis})}
\item{Conjecture \ref{Maartenbis} (or \ref{Maartenbisbis}) without the connected center assumption for the transition morphisms in the system $\{G_{z}^{(i)}\}$.}
\end{enumerate}

\begin{remark}\label{assumprem}
We will show (cf. Proposition \ref{funcprop1}) that in fact condition (3) is a consequence of condition (2), since if Conjecture \ref{Maartenbis} (or \ref{Maartenbisbis}) holds in the connected center case, it holds in general (using the LLC we construct below).
\end{remark}

We clarify the logical structure of the following construction since it is somewhat delicate. The below procedure gives a construction of the rigid refined LLC for $G$ using its analogue for each group in the system $\{G_{z}^{(i)}\}$. For the rest of this paper, this procedure will be the ``explicit" construction of the rigid refined LLC for any $G$ with disconnected center, and when we list assumption (3) above, we mean using this construction. We will show in Proposition \ref{canonLLC} and Corollary \ref{fullunique} that this construction is independent of the system $\{G_{z}^{(i)}\}$; in particular, if $G$ already has connected center, then our construction using any system as above is just the connected rigid refined LLC for $G$ given in assumption (1). This uniqueness fact will also be useful in \S \ref{Endoscopy}, when the system consists of groups with potentially disconnected centers, where we will use the LLC coming from a good $z$-embedding system.

In \S \ref{Endoscopy} we will relate the aforementioned construction to the endoscopic character identities, verifying that it is indeed a construction of the rigid refined LLC for $G$. Denote the transition map $G_{z}^{(j)} \to G_{z}^{(i)}$ in our fixed system by $p_{j,i}$. We start with the following elementary observation:

\begin{lemma}\label{Spluslem} For $\varphi_{z}^{(i)}$ an $L$-parameter for $G_{z}^{(i)}$, the natural map
$$\pi_{0}(S_{\varphi_{z}^{(j)}}^{+}) \xrightarrow{\prescript{L}{}p_{j+1,j}} \pi_{0}(S_{\prescript{L}{}p_{j+1,j} \circ \varphi_{z}^{(j)}}^{+})$$ is an isomorphism for all $j \gg 0$, where $\varphi_{z}^{(j)} := \prescript{L}{}p_{j,i} \circ \varphi_{z}^{(i)}$.
\end{lemma}

\begin{proof} This follows immediately from Corollary \ref{keycor} and the commutativity of the diagram
\[
\begin{tikzcd}
\pi_{0}(S_{\varphi_{z}^{(j)}}^{+}) \arrow["\sim"]{rd} \arrow["\prescript{L}{}p_{j+1,j}"]{rr} &  & \pi_{0}(S_{\prescript{L}{}p_{j+1,j} \circ \varphi_{z}^{(j)}}^{+}) \arrow["\sim"]{ld} \\
& \pi_{0}(S_{\varphi}^{+}), & 
\end{tikzcd}
\]
where $\varphi$ denotes the induced parameter of $G$ and the two diagonal maps are induced by the dual of the embedding of $G$, which are isomorphisms by Corollary \ref{keycor}.
\end{proof}



Fix a $Z$-rigid inner twist $(\mathcal{T}, \bar{h})$ of $G$ with corresponding inner form $G'$. By construction of the system $\{G_{z}^{(i)}\}$ we automatically get a compatible family of $Z$-rigid inner twist $(\mathcal{T}^{(i)}, \bar{h}^{(i)})$ with underlying inner forms $(G')_{z}^{(i)}$, which form a good $z$-embedding system $\{G' \to (G')_{z}^{(i)}\}$ (cf. Lemmas \ref{transfertwist} and \ref{weaklem}). 

The following result (which was brought to our attention by Alexander Bertoloni Meli) will be used repeatedly:
\begin{theorem}(\cite{Silberger79})\label{Silberger} If $G_{1} \xrightarrow{\phi} G_{2}$ is an isogeny of connected reductive groups over a nonarchimedean local field then the pullback of any smooth, irreducible representation of $G_{2}(F)$ by $\phi$ is a finite direct sum of smooth irreducible representations of $G_{1}(F)$.
\end{theorem}

The proof of the above result still holds if one weakens ``isogeny" to the condition that $G_{2}/\phi(G_{1})$ is a torus and $G_{2}(F)/\phi(G_{1}(F))$ is compact, and so we will cite it for this generalization as well. 

We can now begin the construction of the rigid refined LLC for $G$ using its analogue for the $G_{z}^{(i)}$. Start with $\dot{\pi}$ an irreducible smooth representation of $(\mathcal{T}, \bar{h})$ with underlying representation $\pi$ of $G'(F)$. We can suppress the $(\mathcal{T}, \bar{h})$ notation for now because we already have our canonical family of rigid inner twists and corresponding weak $z$-embeddings for all $i$.
 
Fix some $i \geq 1$; by Frobenius reciprocity, we know that $\pi$ is a subrepresentation of $$\res^{(G')_{z}^{(i)}(F)}_{G'(F)}(\ind^{(G')_{z}^{(i)}(F)}_{G'(F)}(\pi)).$$   Note that, since $\pi$ is an irreducible smooth representation of $G'(F)$ the representation $\ind^{(G')_{z}^{(i)}(F)}_{G'(F)}(\pi)$ is admissible and semisimple (by \cite{Bushnell90}), so we have a decomposition
$$\ind^{(G')_{z}^{(i)}(F)}_{G'(F)}(\pi) = \bigoplus_{S_{\pi, i}} \rho^{m_{\rho}}$$ into irreducible smooth representations $\rho$ of $(G')_{z}^{(i)}(F)$. Moreover, the restriction to $G'(F)$ of each $\rho$ decomposes as a finite direct sum of irreducible smooth representations, by Theorem \ref{Silberger}.  

Choose any $\rho \in S_{\pi, i}$ such that $\pi$ is a subrepresentation of $\res^{(G')_{z}^{(i)}(F)}_{G'(F)}(\rho)$ (our construction will not depend on this choice, cf. Lemma \ref{uniqueness}). By Theorem \ref{Silberger}, the restriction of $\rho$ to any $G^{(j)}_{z}(F)$ for $j \geq i$ decomposes as direct sum of irreducible smooth representations. Moreover, Theorem \ref{Silberger} also tells us that there is some $j \gg i$ and irreducible smooth representation $\rho_{j}$ of $(G')_{z}^{(j)}(F)$ (which is a summand of $\rho \circ p_{j,i}$) such that $\rho_{j}|_{G'(F)}$ contains $\pi$ as a summand and, more importantly, such that $\rho_{j} \circ p_{k,j}$ is an irreducible representation of $(G')^{(k)}_{z}(F)$ for all $k \geq j$. In this setting, we have the following result:

\begin{lemma}\label{keycor1} Let $\rho_{j} = (\rho_{j}, V)$ be a smooth representation of $(G')_{z}^{(j)}(F)$ such that $\rho_{j} \circ p_{k,j}$ is irreducible for all $k \geq j$. Then $\rho_{j}|_{G'(F)}$ is irreducible.
\end{lemma}

\begin{proof}
Work with $G' = G$ to simplify notation. We will replace the projective system $\{G_{z}^{(k)}\}_{k \geq j}$ with $$\{\mathcal{G}_{z}^{(k)}\} := \{G_{z}^{(k)} \times Z(G_{z}^{(j)})\},$$ where $G$ is embedded in the first factor of each of these groups. Note that there is a canonical way to extend $\rho_{j} \circ p_{k,j}$ to an irreducible representation of $G_{z}^{(k)}(F) \times Z(G_{z}^{(j)})(F)$.

Assume by contradiction that $\rho_{j}|_{G(F)}$ is not irreducible; in particular, we may find a nonzero vector $v$ such that the $G(F)$-subrepresentation $W$ generated by $v$ is a proper subspace of $V$. Define, for $k \geq j$,
\begin{equation*} H^{(k)} = \{h \in \mathcal{G}_{z}^{(k)}(F) | h\cdot W \not\subset W\} = \{h \in \mathcal{G}_{z}^{(k)}(F) | h\cdot v \notin W\};
\end{equation*}
note that $H^{(k)}$ is a closed subset of $\mathcal{G}_{z}^{(k)}(F)$, since it's the preimage of $V \setminus W$ under the ``action-on-$v$" map. It's also closed under right translation by the subgroup $G(F) \times Z(G_{z}^{(j)})(F)$.

The image of $H^{(k)}$ in $G_{z}^{(j)}(F)$, denoted by $M^{(k)}$, is also closed, since the projection map $\mathcal{G}_{z}^{(k)}(F) \to G_{z}^{(j)}(F)$ is open and $H^{(k)}$ contains the full preimage of $M^{(k)}$. Note that we really mean the image in $G_{z}^{(j)}(F)$, not $\mathcal{G}_{z}^{(j)}(F)$---we have used the $\mathcal{G}^{(k)}_{z}$ to ensure that $Z(G_{z}^{(j)})(F)$ is contained in the image of the map $\mathcal{G}^{(k)}_{z}(F) \to G^{(j)}_{z}(F)$. We also have that $M^{(k)}$ is preserved by right-translation by the subgroup $G(F) \cdot Z(G_{z}^{(j)})(F)$.

 We claim now that $\bigcap_{k > j} M^{(k)} = \emptyset$; to see this, it suffices to show that \begin{equation}\label{int1}
 \bigcap_{k > j} \text{im}[\mathcal{G}^{(k)}_{z}(F) \to G^{(j)}(F)] = G(F) \cdot Z(G_{z}^{(j)})(F),
 \end{equation}
 since if we knew the equality \eqref{int1} then an element in  $\bigcap_{k > j} M^{(k)}$ would yield an element of $G(F)$ that sends $v$ into $V \setminus W$ (since $Z(G_{z}^{(j)})(F)$ acts by a character). The characterization \eqref{int1} of the intersection follows immediately from the fact that the intersection of the images of
 \begin{equation*}
 \mathcal{G}^{(k)}_{z}(F) \to G^{(j)}_{z}(F) \to H^{1}(F, Z(G))
 \end{equation*}
 are trivial across all $k$, since one checks that these are the same as the images of the connecting homomorphisms $G^{(k)}_{z}(F) \to H^{1}(F, Z(G))$ from the short exact sequences 
 $$1 \to Z(G) \to G \times Z(G^{(k)}_{z}) \to G^{(k)}_{z} \to 1,$$ which tend to zero by definition of a good $z$-embedding system (cf. Lemma \ref{shrinkingimage} and the discussion following Definition \ref{goodsysdef}).

Having established $\bigcap_{k > j} M^{(k)} = \emptyset$, if we denote by $\overline{M}^{(k)}$ the closed image of $M^{(k)}$ in $H^{1}(F, Z(G))$, then also $\bigcap_{k> j} \overline{M}^{(k)}_{\rho} = \emptyset$, because $M^{(k)}$ is the full preimage in $G_{z}^{(j)}(F)$ of $\overline{M}^{(k)}$, by right $G(F) \cdot Z(G_{z}^{(j)})(F)$-invariance. 

Since $H^{1}(F, Z(G))$ is compact and $\bigcap_{k > j} \overline{M}^{(k)} = \emptyset$, we must have 
$$\bigcap_{k' \geq k > j}  \overline{M}^{(k)} = \overline{M}^{(k')} = \emptyset$$ for some $k' \geq j$, by the finite intersection property, which obviously implies that $M^{(k')} = \emptyset$, implying that $\mathcal{G}_{z}^{(k)}(F) \cdot v$ is contained in $W$ and therefore that $W$ is $\mathcal{G}_{z}^{(k)}(F)$-stable, contradicting the irreducibility of $\rho_{j} \circ p_{k,j}$.
\end{proof}

For applications to the rigid LLC, it is desirable to show that any irreducible smooth representation $\pi$ of $G'(F)$ arises uniquely in the above manner. To guarantee uniqueness, we need the following result, which has a similar flavor to Lemma \ref{keycor1}:
 
 \begin{lemma}\label{uniqueness} Let $(\rho, V), (\rho', V')$ be irreducible smooth representations of some $(G')^{(i)}_{z}(F)$ with the same central characters whose restrictions to $G'(F)$ are isomorphic irreducible representations (this immediately implies their pullbacks by $p_{j,i}$ are also irreducible for any $j \geq i$). Then for all $j \gg i$ the pullbacks of $\rho$ and $\rho'$ to $(G')^{(j)}_{z}(F)$ are isomorphic.
 \end{lemma}
 
 \begin{proof} Replace $G'$ by $G$ to simplify notation. Let $f \colon V \xrightarrow{\sim} V'$ be an isomorphism of vector spaces realizing the isomorphism at the level of restrictions to $G(F)$. Note that, for any $g \in G^{(i)}_{z}(F)$, the isomorphism $^{g}f$ is also an isomorphism of $G(F)$-representations, so by Schur's Lemma the action of $G^{(i)}_{z}(F)$ on $f$ is determined by a (continuous) character $\chi_{f} \colon G^{(i)}_{z}(F) \to \mathbb{C}^{\times}$, where the target has the discrete topology. 
 
Define the subset $$H_{f} := \{g \in G_{z}^{(i)}(F) | \exists v \in V \text{ such that } f(g^{-1} \cdot v) \neq g^{-1} \cdot f(v) \},$$ which is a closed subset of $G_{z}^{(i)}(F)$, since it's the preimage under $\chi_{f}$ of the closed subset $\mathbb{C}^{\times} \setminus \{1\}$
 
By construction, $H_{f}$ is stable under left-translation by $G(F) \cdot Z(G_{z}^{(i)})(F)$ (the central part uses our assumption that the representations have the same central characters). Now consider the image of $\bar{H}_{f}$ of $H_{f}$ under the (closed) quotient map with compact target 
\begin{equation*}
G^{(i)}_{z}(F) \to [G(F) \cdot Z(G_{z}^{(i)})(F)] \backslash G_{z}^{(i)}(F).
\end{equation*}
 
Using that 
\begin{equation*} \bigcap_{j \geq i} [p_{j,i}(G_{z}^{(j)}(F)) \cdot Z(G_{z}^{(i)}(F))] = G(F) \cdot Z(G_{z}^{(i)})(F)
\end{equation*} (this equality is from the proof of Corollary \ref{keycor1}), we deduce that $\bigcap_{j \geq i} ([p_{j,i}(G_{z}^{(j)}(F)) \cdot Z(G_{z}^{(i)}(F))] \cap H_{f}) = \emptyset,$ or else $f$ would not be $G(F)$-equivariant. Set $H_{f,j} = [p_{j,i}(G_{z}^{(j)}(F)) \cdot Z(G_{z}^{(i)}(F))] \cap H_{f}$, which is also a closed subset of $G_{z}^{(i)}(F)$ (cf. the proof of Corollary \ref{keycor1}) left-invariant under translation by $G(F) \cdot Z(G_{z}^{(i)})(F)$, with image in $[G(F) \cdot Z(G_{z}^{(i)})(F)] \backslash G_{z}^{(i)}(F)$ denoted by $\bar{H}_{f,j}$, which is also closed.
 
The fact that $\bigcap_{j \geq i} H_{f,j} = \emptyset$ together with the above left-invariance implies that $\bigcap_{j \geq i} \bar{H}_{f,j} = \emptyset$, and hence, by the finite intersection property, that there is some $k \geq i$ with $\bigcap_{k \geq j \geq i} \bar{H}_{f,j} = \emptyset$. For this $k$ we immediately also have $\bigcap_{k \geq j \geq i} H_{f,j} = H_{f,k} = \emptyset$, and thus $f$ is $G^{(\ell)}_{z}(F)$-equivariant for any $\ell \geq k$, as desired.
 \end{proof}

We can now finish giving the construction of the rigid refined LLC for $G$. Continuing with the above notation, Lemma \ref{keycor1} implies that the restriction of $\rho_{j}$ (as defined above) to $G'(F)$ is $\pi$. Fix a Whittaker datum $\mathfrak{w}$ for $G$, which determines Whittaker data for each $G^{(j)}_{z}$, denoted by $\mathfrak{w}_{z}^{(j)}$. 

Set $\dot{\pi}^{(k)} := ((\mathcal{T}^{(k)}, \bar{h}^{(k)}), \rho_{j} \circ p_{k,j})$ and $\dot{\pi} = ((\mathcal{T}, \bar{h}), \pi)$. We then define
\begin{equation}\label{limLLC} \iota_{\mathfrak{w}}(\dot{\pi}) =  \varinjlim_{k \geq j} \iota_{\mathfrak{w}_{z}^{(j)}}(\dot{\pi}^{(k)});
\end{equation}
note that this direct limit is well-defined (and, in particular, eventually constant) by combining Conjecture \ref{Maartenbis} and Lemma \ref{Spluslem}.

\begin{remark}\label{weakremark2} If one wants to weaken Conjecture \ref{Maartenbis} as in Remark \ref{weakremark}, then each limit term in the equation \eqref{limLLC} needs to be modified via twisting by a character of $(G')_{z}^{(k)}(F)$ arising from the element of $H^{1}(W_{F}, Z(G_{z}^{(k)}))$ required to lift parameters up to twisting as in Remark \ref{weakremark} (which exists by Proposition \ref{liftparam}) via the map $H^{1}(W_{F}, Z(G_{z}^{(k)})) \to (G')_{z}^{(k)}(F)^{D}$ constructed by Langlands (as explained in \cite{Borel77}), assuming the compatibility of the rigid refined LLC with the Langlands correspondence for tori (cf. \cite[\S 6.1.1]{Taibi22}).
\end{remark}

Observe that the above map is well-defined (that is, it does not depend on the choice of $\rho$, and $\rho_{j}$ is uniquely determined up to isomorphism) by Lemma \ref{uniqueness}, finishing the construction. We have the following stronger version of uniqueness:

\begin{proposition}\label{canonLLC} The construction of the refined LLC given above does not depend on the choice of good $z$-embedding system $\{G_{z}^{(i)}\}$ where each $G_{z}^{(i)}$ has connected center. 
\end{proposition}

\begin{proof} 
Consider two good $z$-embedding systems $\{G_{z}^{(i)}\}$ and $\{H_{z}^{(j)}\}$ Setting $J_{i,j} := G_{z}^{(i)} \times^{Z(G)} Z(H_{z}^{(j)})$, it is straightforward to check that, for any fixed $i$, the system $\{G_{z}^{(i)} \to J_{i,j}\}_{j}$ consists of compatible weak $z$-embeddings with surjective transition maps, and for $j \gg 0$ (depending on $i$) the map $G^{(i)}_{z} \to J_{i,j}$ is in fact a \textbf{pseudo} $z$-embedding (the convergence at a finite level is due to the connectedness of $Z(G^{(i)}_{z})$). We note that the cokernels $J_{i,j}/G_{z}^{(i)}$ may not be induced (cf. Remark \ref{inducedremark}). 

For any fixed $j$, the composition $$H_{z}^{(j)} \xrightarrow{\sim} \frac{G \times Z(H_{z}^{(j)})}{Z(G)} \to \frac{G_{z}^{(i)} \times Z(H_{z}^{(j)})}{Z(G)} = J_{i,j}$$ yields a compatible system of weak $z$-embeddings $\{H_{z}^{(j)} \to J_{i,j}\}_{i}$ with surjective transition maps.

Recall that for any fixed irreducible smooth representation $\pi$ of $G(F)$ we can find $i \gg 0$ such that $\pi$ is the restriction of $\pi^{(i)}$ an irreducible representation of $G_{z}^{(i)}(F)$ and, if $\varphi_{z}^{(i)}$ is the corresponding parameter for $G_{z}^{(i)}$ and $\varphi$ the induced parameter for $G$, such that the natural map
\begin{equation*}
  \pi_{0}(S_{\varphi_{z}^{(i)}}^{+}) \to \pi_{0}(S_{\varphi}^{+})  
\end{equation*}
is an isomorphism. 

We can also find $j \gg 0$ such that $G_{z}^{(i)} \to J_{i,j}$ is a pseudo $z$-embedding, and hence $\pi^{(i)}$ extends (uniquely up to a choice of extension of central characters) to an irreducible representation $\rho^{(j)}$ of $J_{i,j}(F)$. Note that $\rho^{(j)}|_{H_{z}^{(j)}(F)}$ is automatically irreducible, since this is the case for $\rho^{(j)} |_{G(F)} = \pi$ and $G(F)$ is a subgroup of $H_{z}^{(j)}(F)$. After possibly enlarging $j$, we may assume that the induced map $\pi_{0}(S_{\varphi_{z,H}^{(j)}}^{+}) \to \pi_{0}(S_{\varphi}^{+})$ is an isomorphism for any $\varphi_{z,H}^{(j)}$ a parameter for $H_{z}^{(j)}$ lifting $\varphi$.  

After choosing appropriate rigidifcations of $\pi$, the construction of refined LLC using the system $\{G_{z}^{(i)}\}$ yields 
\begin{equation*}
\iota_{\mathfrak{w}}(\dot{\pi}) = \iota_{\mathfrak{w}_{G,z}^{(i)}}(\dot{\pi}^{(i)}) = \iota_{\mathfrak{w}^{i,j}}(\dot{\rho}^{(j)}),
\end{equation*}
($\mathfrak{w}^{i,j}$ denotes the induced Whittaker datum for $J_{i,j}$) where the second equality is due to Conjecture \ref{Maartenbis} for groups with connected center, using that we may lift the parameter $\varphi_{z}^{(i)}$ to a parameter $\varphi_{i,j}$ for $J_{i,j}$ and the induced map between centralizers
\begin{equation*}
   \pi_{0}(S_{\varphi_{i,j}}^{+}) \to \pi_{0}(S_{\varphi_{z}^{(i)}}^{+}) 
\end{equation*}
is an isomorphism because $G_{z}^{(i)} \to J_{i,j}$ is a pseudo $z$-embedding (see \cite[\S \S 5.1, 5.2]{kaletha18} for an explanation of these facts). 

 The construction of refined LLC using $\{H_{z}^{(j)}\}$ yields $\iota_{\mathfrak{w}_{H,z}^{(j)}}(\dot{\rho^{(j)}|_{H_{z}^{(j)}(F)}})$ by Lemma \ref{uniqueness}. We now obtain the desired compatibility from applying Conjecture \ref{Maartenbis} to the embedding $H_{z}^{(j)} \to J_{i,j}$, using that 
\begin{equation*}
 \pi_{0}(S_{\varphi_{i,j}}^{+}) \to \pi_{0}(S_{\varphi_{z,H}^{(j)}}^{+})
 \end{equation*}
(where $\varphi_{z,H}^{(j)}$ is the parameter of $^{L}H_{z}^{(j)}$ induced by $\varphi_{i,j}$)  is an isomorphism, since they're both (compatibly) isomorphic to $\pi_{0}(S_{\varphi}^{+})$, by construction.
\end{proof}

\begin{remark}
The connected center assumption is not necessary; Corollary \ref{fullunique} gives the strongest version of this result (cf. Remark \ref{assumprem}). The assumption is included here only for clarity of exposition. The above Proposition also holds if one uses the twisted version of the formula \eqref{limLLC} as explained in Remark \ref{weakremark2} adapted to the weaker hypotheses of Remark \ref{weakremark}.
\end{remark}

\subsection{General functoriality}\label{ConstructionFunc}
This subsection focuses on the proof of the result:

\begin{proposition}\label{funcprop1}
If Conjecture \ref{Maartenbis} (resp. \ref{Maartenbisbis}) holds for all $G_{1} \xrightarrow{\eta} G_{2}$ with connected center then, using our construction of the refined LLC for general $G$ given above, Conjecture \ref{Maartenbis} (resp. \ref{Maartenbisbis}) holds for any $G_{1} \xrightarrow{\eta} G_{2}$. 
\end{proposition}

\begin{proof}
We prove this for Conjecture \ref{Maartenbis}---the version for Conjecture \ref{Maartenbisbis} follows from an identical argument. Fix an $L$-parameter $\varphi$ for $^{L}G_{2}$, as well as a Whittaker datum $\mathfrak{w}$ for $G_{2}$ which induces a Whittaker datum, also denoted by $\mathfrak{w}$, for $G_{1}$. We also fix a rigid inner twist $(G_{1}',\mathcal{T}, \bar{h})$ of $G_{1}$ with underlying inner form $G_{1}'$ and $\rho \in \mathrm{Irr}(\pi_{0}(S_{\varphi}^{+}), \mathcal{T} \times^{G_{1}} G_{2})$. 

Construct a good $z$-embedding system (with connected centers) from Proposition \ref{weaksystemexist} for $G_{1}$, denoted by $\{G_{1} \to H_{z}^{(i)}\}$, along with the usual associated notation. Form the system of embeddings $\{G_{2} \to \bar{H}_{z}^{(i)}\}$, where $\bar{H}_{z}^{(i)} := H_{z}^{(i)} \times^{G_{1}} G_{2} = H_{z}^{(i)} \times^{Z(G_{1})} Z(G_{2})$. 

We claim that $Z(\bar{H}_{z}^{(i)})$ is connected; since $Z(H_{z}^{(i)})$ is connected, it suffices to show that $Z(G_{2})/\eta(Z(G_{1}))$ is connected. Replacing $G_{1}$ by $\eta(G_{1})$ immediately reduces to the case where $\eta$ is an inclusion, where it induces an equality on derived groups, and hence we have (using the central isogeny decomposition)
\begin{equation*}
\frac{Z(G_{2})}{Z(G_{1})} = \frac{Z(G_{2})^{\circ} \times^{Z(G_{2})^{\circ} \cap G_{2,\mathrm{der}}} G_{2,\mathrm{der}}}{Z(G_{1})^{\circ} \times^{Z(G_{1})^{\circ} \cap G_{2,\mathrm{der}}} G_{2,\mathrm{der}}} \cong \frac{Z(G_{2})^{\circ}/(Z(G_{2})^{\circ} \cap G_{2,\mathrm{der}})}{Z(G_{1})^{\circ}/(Z(G_{1})^{\circ} \cap G_{2,\mathrm{der}})},
\end{equation*}
and the right-most term is connected, giving the claim.

By construction, we have 
\begin{equation*}
\frac{\bar{H}_{z}^{(i)}}{G_{2}} = \frac{H_{z}^{(i)}}{G_{1}} = C^{(i)},
\end{equation*}
which we explicitly described (along with the induced transition maps) in \S \ref{ConstructionEmbeddings}. In particular, we have for any $j \geq 1$ the usual factorization of the connecting homomorphism (using an analogue of the diagram \eqref{connhomfacdiag})
\begin{equation}\label{funcfac}
    C^{(j)}(F) \to C^{(1)}(F) \to H^{1}(F, Z(G_{2}));
\end{equation}
as we range over our cofinal system of finite Galois extensions of $F$, we also get a cofinal system of finite separable extensions of $K_{1}$ (where $K_{1}$ is from the first weak embedding pair as in Proposition \ref{weaksystemexist} for our choice of good $z$-embedding system) and so, from the explicit description of $C^{(j)}(F) \to C^{(1)}(F)$ in terms of field norms, we see that the system $\{G_{2} \to \bar{H}_{z}^{(i)}\}$ satisfies conditions (1) and (2) of Definition \ref{goodsysdef}. Thus, to verify that $\{G_{2} \to \bar{H}_{z}^{(i)}\}$ is a good $z$-embedding system we only need to check that (for $j \gg 0$) the map $G_{2} \to \bar{H}_{z}^{(j)}$ is actually a weak $z$-embedding, which also follows easily from the factorization \eqref{funcfac}.

For each $i$ the map $H_{z}^{(i)} \xrightarrow{\eta_{i}} \bar{H}_{z}^{(i)}$ satisfies the hypotheses of Conjecture \ref{Maartenbis} and both groups have connected center. Denote by $\{\bar{\varphi}_{z}^{(i)}\}$ a compatible system of lifts of $\varphi$ to each $^{L}\bar{H}_{z}^{(i)}$ and choose $j \gg 0$ so that the natural maps (for $k \geq j$)
\begin{equation*}
\pi_{0}(S_{\bar{\varphi}_{z}^{(k)}}^{+}) \to \pi_{0}(S_{\varphi}^{+})
\end{equation*}
are isomorphisms, and such that the same holds for $\bar{\varphi}_{z}^{(k)}$, $\varphi$ replaced by $\prescript{L}{}\eta_{k} \circ \bar{\varphi}_{z}^{(k)}$, and $^{L}\eta \circ \varphi$. We have for any $\rho_{k} \in \mathrm{Irr}(\pi_{0}(S_{\bar{\varphi}_{z}^{(k)}}^{+}), \mathcal{T}^{(k)} \times^{H_{z}^{(k)}} \bar{H}_{z}^{(k)})$ that $\iota_{\mathfrak{w},\bar{H}_{z}^{(k)}}^{-1}(\bar{\varphi}_{z}^{(k)}, \rho_{k}) \circ \eta_{k}'$ is irreducible for $k \gg 0$, since $(\varinjlim_{k} \iota_{\mathfrak{w},\bar{H}_{z}^{(k)}}^{-1}(\bar{\varphi}_{z}^{(k)}, \rho_{k}) \circ \eta_{k}')|_{G_{1}(F)} = \iota_{\mathfrak{w},G_{2}}^{-1}(\rho) \circ \eta'$ is. We thus have by assumption that the representation $((\mathcal{T}^{(k)}, \bar{h}_{z}^{(k)}), \iota_{\mathfrak{w},\bar{H}_{z}^{(k)}}^{-1}(\bar{\varphi}_{z}^{(k)}, \rho_{k}) \circ \eta_{k}')$ lies in the compound $L$-packet $\Pi_{\prescript{L}{}\eta_{k} \circ \bar{\varphi}_{z}^{(k)}}^{Z'}$. Moreover, when $\pi_{0}(S_{\varphi}^{+}) \to \pi_{0}(S_{\prescript{L}{}\eta \circ \varphi}^{+})$ is an isomorphism, so is each $\pi_{0}(S_{\bar{\varphi}_{z}^{(k)}}^{+}) \to \pi_{0}(S_{\prescript{L}{}\eta \circ \bar{\varphi}_{z}^{(k)}}^{+})$, and hence by assumption
\begin{equation*}
((\mathcal{T}^{(k)}, \bar{h}_{z}^{(k)}), \iota_{\mathfrak{w},\bar{H}_{z}^{(k)}}^{-1}(\bar{\varphi}_{z}^{(k)}, \rho_{k}) \circ \eta_{k}') = \iota_{\mathfrak{w}, H_{z}^{(k)}}^{-1}(\prescript{L}{}\eta_{k} \circ \bar{\varphi}_{z}^{(k)}, \rho_{k} \circ \prescript{L}{}\eta).
\end{equation*}

We obtain immediately obtain the desired result by applying the construction of LLC for $G_{1}$ and $G_{2}$ given in \S \ref{rigidcomp} (using the good $z$-embedding systems $\{H_{z}^{(i)}\}$ and $\{\bar{H}_{z}^{(i)}\}$, respectively, via Proposition \ref{canonLLC}).
\end{proof}

The analogous result holds for the LLC from Remark \ref{weakremark2} using the same argument.

\begin{corollary}\label{fullunique}
Proposition \ref{canonLLC} holds for any good $z$-embedding system $\{G_{z}^{(i)}\}$ (not necessarily with connected centers).
\end{corollary}

\section{Endoscopy}\label{Endoscopy}
For $G$ a quasi-split connected reductive group we fix for once and for all $\{G \to G_{z}^{(i)}\}$ a system of weak $z$-embeddings as constructed in Proposition \ref{weaksystemexist} (in particular,  it is a good $z$-embedding system and each $Z(G_{z}^{(i)})$ is connected). To facilitate comparison with the refined isocrystal LLC it is necessary to define a system of finite central subgroups as follows: We set $Z_{n}$ to be the preimage of $(Z(G)/Z_{\text{der}})[n]$ in $Z(G)$, take $G_{n} := G/Z_{n}$ and $G^{(i)}_{z,n} := G_{z}^{(i)}/Z_{n}$. Note that each $G_{z,n} \to G_{z,n}^{(i)}$ is still a weak $z$-embedding, since $G_{z,n}^{(i)}/G_{n} = C^{(i)}$ and the connecting homomorphism $C^{(i)}(F) \to H^{1}(F, G_{n})$ factors through $C^{(i)}(F)\xrightarrow{0} H^{1}(F, G)$ (and the situation is completely analogous for each inner twist $G'$). Moreover, the system $\{G_{n} \to G_{z,n}^{(i)}\}$ is of the type constructed in Proposition \ref{weaksystemexist}; more precisely, each $(T_{0}/Z_{n}, K_{i})$ is a weak embedding pair for $G_{n}$, and the weak $z$-embedding $G_{n} \to G_{z,n}^{(i)}$ can be canonically identified with the one obtained from $G_{n}$ and $(T_{0}/Z_{n}, K_{i})$ as explained in \S \ref{ConstructionEmbeddings}. Set $\bar{G} = \varinjlim_{n} G_{n}$, $\widehat{\bar{G}} = \varprojlim_{n} \widehat{G_{n}}$, $\bar{G}^{(i)}_{z} = \varinjlim G^{(i)}_{z,n}$, and $\widehat{\bar{G}^{(i)}_{z}} = \varprojlim \widehat{G_{z,n}^{(i)}}$. 

\subsection{Systems of endoscopic data}\label{EndoscopySystems}

Let $\dot{\mathfrak{e}} = (H, \mathcal{H}, \dot{s}, \xi)$ be a refined endoscopic datum for $G$; as in \cite[\S 5.1.3]{kaletha18}, our goal is to construct a compatible system of refined endoscopic data $\{ \dot{\mathfrak{e}}_{z}^{(j)}\}$ for all $j \gg 0$, where, as the notation suggests, $\dot{\mathfrak{e}}_{z}^{(j)}$ is a refined endoscopic datum for $G_{z}^{(j)}$. This construction will allow us to study the endoscopy of $G$ using the endoscopy for the system $\{G_{z}^{(j)}\}$. 

For any $i$ and $n$ fixed note that there is a short exact sequence 
\begin{equation*}
0 \to \widehat{C^{(i)}}^{\Gamma} \to Z(\widehat{G_{z,n}^{(i)}})^{+} \to Z(\widehat{G_{n}})^{+} \to 0
\end{equation*}
(cf. the proof of Corollary \ref{keycor} for the ``$+$"-aspect of this sequence). The surjectivity in the above short exact sequence still holds in the weak $z$-embedding setting, since it only relies on the surjectivity of $G_{z}^{(i)}(F) \to C^{(i)}(F)$, cf. the proof of \cite[Lemma 5.10]{kaletha18}, which is one of the properties that carries over in the weak case. This surjectivity extends to the inverse limits over $n$, since the kernel of the above map is constant (so the relevant derived inverse limit vanishes), giving a lift $\dot{s}_{z}^{(i)}$ of $\dot{s}$. The image of $\dot{s}_{z}^{(i)}$ in $\varinjlim_{j \geq i} \widehat{\bar{G}_{z}^{(j)}}$ automatically gives such a lift for any $j$, so we denote the resulting direct limit of refined data simply by $\dot{s}_{z}$. Note that one could take $i=1$ to find such a lift, but allowing for general $i$ gives more flexibility that will be used later.

Recall that by definition $\mathcal{H}$ is a split extension of $W_{F}$ by $\widehat{H} \hookrightarrow \mathcal{H}$, and we have $\xi(\mathcal{H}) = Z_{\widehat{G}}(s)^{\circ}$, where $s$ is the image of $\dot{s}$ in $\widehat{G}$. Moreover, if we define $\mathcal{H}_{z}^{(i)}$ as the pullback of $\mathcal{H} \to \prescript{L}{}G$ and $^{L}G_{z}^{(i)} \to \prescript{L}{}G$, then we have a natural injection $\widehat{C^{(i)}} \to \mathcal{H}_{z}^{(i)}$ realizing the left-hand group as the kernel of the surjection $\mathcal{H}_{z}^{(i)} \to \mathcal{H}$. Denote the other projection by $$\mathcal{H}_{z}^{(i)} \xrightarrow{\xi_{z}^{(i)}} \prescript{L}{}G_{z}^{(i)}.$$
Also define $$H_{z}^{(i)} := H \times^{Z(G)} Z(G_{z}^{(i)}),$$ which is well-defined since $Z(G)$ is canonically embedded in any endoscopic group $H$, and set $$\dot{\mathfrak{e}}_{z}^{(i)} := (H_{z}^{(i)}, \mathcal{H}_{z}^{(i)}, \dot{s}_{z}, \xi_{z}^{(i)}).$$

As the notation suggests, we now show:

\begin{proposition}\label{liftend} For all $j \gg i$ (recall that $i$ is fixed), the following facts hold: The tuple $\dot{\mathfrak{e}}_{z}^{(j)}$ is a refined endoscopic datum for $G_{z}^{(j)}$, and $\{H \to H_{z}^{(k)}\}_{k \geq j}$ is a good $z$-embedding system.
\end{proposition}

We remark that this is the stage where we need the more general definition of a weak $z$-embedding, since $Z(H_{z}^{(j)})$ need not be connected. Note also that there is a subtle difference from the mixed characteristic case, wherein the maps $H \to H_{z}^{(j)}$ could actually fail to be weak $z$-embeddings if $j \geq i$ is too small, and the tuple $\dot{\mathfrak{e}}_{z}^{(j)}$ could fail to give a refined endoscopic datum (because, as we will see, the element $\dot{s}_{z}^{(j)}$ could fail to lie in $Z(\widehat{H_{z}^{(j)}})^{+}$ if $j$ is too small).

\begin{proof} We have the short exact sequence 
\begin{equation*}
1 \to H \to H_{z}^{(j)} \to C^{(j)} \to 1,
\end{equation*}
and so the only thing to show for the weak $z$-embedding claim is the injectivity on $H^{1}$ for $j \gg 0$ (for the group $H$ and also its inner forms). We have the exact sequence
$$H_{z}^{(i)}(F) \to C^{(i)}(F) \to H^{1}(F, H) \to H^{1}(F, H_{z}^{(i)}) \to 0,$$ where the last map is the one in question, which may not be injective. 
Recall also the exact sequence $$Z(G_{z}^{(j)})(F) \to C^{(j)}(F) \to H^{1}(F, Z(G)) \to H^{1}(F, Z(G_{z}^{(j)})),$$ which we studied at length in the previous section. 
Note that the connecting homomorphism $C^{(j)}(F) \to H^{1}(F, H)$ factors through the above connecting homomorphism to $H^{1}(F, Z(G))$, since the lift of any $x \in C^{(j)}(F)$ may be taken in $Z(G_{z}^{(j)})(\overline{F})$, and hence also through the composition $$C^{(j)}(F) \xrightarrow{N_{K_{j}/K_{i}}} C^{(i)}(F) \to H^{1}(F, Z(G)) \to H^{1}(F, H)$$ (where the middle map is the connecting homomorphism for the same sequence with $j=i$ and we are picking a splitting of $T_{0}$ as in \S \ref{ConstructionEmbeddings} to make sense of the norm maps).

From here we can use a familiar argument: The kernel of $C^{(i)}(F) \to H^{1}(F, H)$ is an open subgroup since the right-hand set is finite, and hence for $j \gg 0$ the connecting homomorphism $C^{(j)}(F) \to H^{1}(F, H)$ vanishes, giving the desired injectivity (after twisting). Applying the identical argument for each inner twist $H'$ (and taking the maximum of the $j$'s) shows that $H \to H_{z}^{(k)}$ is a weak $z$-embedding for $k \gg 0$. To see that this gives a good $z$-embedding system, we note that are obvious surjective transition maps between these groups compatible with the embedding of $H$, and the remaining condition of Definition \ref{goodsysdef} (condition (2)) is a straightforward verification. 

It remains to show that each $\dot{\mathfrak{e}}_{z}^{(j)}$ is a refined endoscopic datum for $j \gg 0$. As in the proof of \cite[Lemma 5.14]{kaletha18}, we may assume that $\mathcal{H} \subset \prescript{L}{}G$ and $\xi$ is the inclusion, which means the analogue is also true for each $\mathcal{H}_{z}^{(j)}$, $\prescript{L}{}G_{z}^{(j)}$, and $\xi_{z}^{(j)}$. It is clear that each $H_{z}^{(j)}$ is quasi-split because $H$ is, and the identical argument in the mixed characteristic case shows that $Z_{\widehat{G_{z}^{(j)}}}(s_{z}^{(j)})^{\circ}$ is a dual group for $H_{z}^{(j)}$, where $s_{z}^{(j)}$ is the image of $\dot{s}_{z}$ in $\widehat{G_{z}^{(j)}}$, and that $\mathcal{H}_{z}^{(j)}$ is an extension of $\widehat{H^{(j)}_{z}}$ by $W_{F}$. Moreover, the identifications $\widehat{H_{z}^{(j)}}\xrightarrow{\sim} Z_{\widehat{G_{z}^{(j)}}}(s_{z}^{(j)})^{\circ}$ may be chosen so that each map $\widehat{H_{z}^{(j)}} \to \widehat{H_{z}^{(k)}}$ (for $k \geq j$) induced by $\widehat{G_{z}^{(j)}} \to \widehat{G_{z}^{(k)}}$ is dual to the projection $H_{z}^{(k)} \to H_{z}^{(j)}$. The extensions can also be chosen to be compatible with the dual maps $\widehat{H_{z}^{(j)}} \to \widehat{H_{z}^{(k)}}$ and the maps $\mathcal{H}^{(j)}_{z} \to \mathcal{H}^{(k)}_{z}$ induced by pulling back $\widehat{G_{z}^{(j)}} \to \widehat{G_{z}^{(k)}}$.

We now show that the extension $0 \to \widehat{H^{(j)}_{z}} \to \mathcal{H}_{z}^{(j)} \to W_{F} \to 0$ obtained in the above paragraph is (compatibly as $j$ varies) split. Composing a splitting of the extension $\mathcal{H}$ with $\xi$ gives an $L$-homomorphism $W_{F} \xrightarrow{a} \prescript{L}{}G$, which by Proposition \ref{liftparam} (whose proof works for general $L$-homomorphisms) lifts to an $L$-homomorphism $W_{F} \to \prescript{L}{}G_{z}^{(j)}$ (we can choose such a lift for $j=i$ and then take the lift for $j \geq i$ to be the one given by composing). Since the image of $a$ lies in $\xi(\mathcal{H})$, the image of this lift lies in the preimage of $\mathcal{H}$ in $\prescript{L}{}G_{z}^{(j)}$, which is exactly $\xi_{z}^{(j)}(\mathcal{H}_{z}^{(j)})$. Since $\xi_{z}^{(j)}$ is an isomorphism of topological groups onto its image, postcomposing this extension of $a$ with $(\xi_{z}^{(j)})^{-1}$ yields is a continuous $L$-homomorphism from $W_{F}$ to $\mathcal{H}_{z}^{(j'}$ splitting our original extension. Using the same extension of $a$ for each $k \geq j$ gives the desired compatible sequence of splittings. 

The only step in showing that $\dot{\mathfrak{e}}_{z}^{(j)}$ is an endoscopic datum where one needs to move higher up in the system is showing that $s_{z}^{(k)} \in Z(\widehat{H_{z}^{(k)}})^{\Gamma}$ for $k \gg i$. Recall that $s_{z}^{(k)}$ is a lift of $s \in Z(\widehat{H})^{\Gamma}$, and we have the exact sequence $0 \to \widehat{C}^{(k)} \to Z(\widehat{H^{(k)}_{z}}) \to Z(\widehat{H}) \to 0$, which further yields another exact sequence
\begin{equation*}
0 \to (C^{(k)})^{\Gamma} \to Z(\widehat{H_{z}^{(k)}})^{\Gamma} \to  Z(\widehat{H})^{\Gamma} \to H^{1}(W_{F}, \widehat{C}^{(k)}).
\end{equation*}

The identical argument as in the proof of Proposition \ref{limsurj} (which does not use the connectedness of $Z(G_{z}^{(i)})$) shows that $s$ lies in the image of $Z(\widehat{H_{z}^{(k)}})^{\Gamma}$ for all $k \gg 0$, and since the fibers of $Z(\widehat{H_{z}^{(k)}})^{\Gamma} \to  Z(\widehat{H})^{\Gamma}$ are $(C^{(k)})^{\Gamma}$-torsors, we get that $s_{z}^{(k)} \in Z(\widehat{H_{z}^{(k)}})^{\Gamma}$ for $k \gg 0$, completing this step.

The final verification required is that the action of $W_{F}$ on $\widehat{H^{(k)_{z}}}$ induced by the extension $\mathcal{H}_{z}^{(k)}$ agrees with the standard $L$-group structure induced by $H^{(k)}_{z}$, which can be taken verbatim from the mixed characteristic case. Thus, $\dot{\mathfrak{e}}_{z}^{(k)}$ is a refined endoscopic datum for $k \gg 0$. 
\end{proof}

We re-emphasize that, although one can still define the tuple $\dot{\mathfrak{e}}_{z}^{(j)}$ for any $j \geq i$, it can fail to be a refined endoscopic datum if $j$ is too small. An important property of the above construction is:

\begin{lemma} For $\dot{\mathfrak{e}}$ as above, the family of refined endoscopic data $\{\dot{\mathfrak{e}}_{k}\}_{k \gg 0}$ is unique up to (compatible) isomorphism of refined endoscopic data.
\end{lemma}

\begin{proof} The only choice made in the construction is the lift $\dot{s}_{z}^{(i)}$ of $\dot{s}$, and any two choices yield the same isomorphism classes by the proof of \cite[Fact 5.15]{kaletha18}.
\end{proof}

In contrast to the above slightly delicate procedure of lifting endoscopic data from $G$ to $\{G_{z}^{(i)}\}$, it is easy to pass from a refined endosocopic datum $(H_{z}^{(i)}, \mathcal{H}_{z}^{(i)}, \dot{s}_{z}^{(i)}, \xi_{z}^{(i)})$ for some $G_{z}^{(i)}$ to a refined endoscopic datum for $G$, as well as for any $G_{z}^{(j)}$ with $j \geq i$, as follows. We explain the latter first: Since $\xi_{z}^{(i)}(\widehat{H_{z}^{(i)}})$ contains $\widehat{C^{(i)}}$ we have a dual surjection $H^{(i)}_{z} \to C^{(i)}$ and define $H^{(j)}_{z} := H_{z}^{(i)} \times_{C^{(i)}} C^{(j)}$, so that $\widehat{H^{(j)}_{z}} = \widehat{H_{z}^{(i)}} \times^{\widehat{C^{(i)}}} \widehat{C^{(j)}}$ and we take $\mathcal{H}_{z}^{(j)} = \mathcal{H}_{z}^{(i)}  \times^{\widehat{C^{(i)}}} \widehat{C^{(j)}}$, where in the last two groups $\widehat{C^{(i)}}$ is embedded in the left-hand factor via $(\xi_{z}^{(i)})^{-1}$. We take $\xi_{z}^{(j)}$ to be the inclusion induced by $\xi_{z}^{(i)}$ and $\widehat{C^{(j)}} \hookrightarrow \prescript{L}{}G_{z}^{(j)}$ and $\dot{s}_{z}^{(j)}$ the image of $\dot{s}_{z}^{(i)}$; it is easy to check that this defines a refined endoscopic datum for $G_{z}^{(j)}$.

To construct a refined endoscopic datum for $G$, first let $H$ be the kernel of the aforementioned map $H_{z}^{(i)} \to C^{(i)}$ and set $\widehat{H} = \text{im}(\widehat{H_{z}^{(i)}} \to \widehat{G})$. Taking $\dot{s}$ to be the image of $\dot{s}_{z}^{(i)}$ and $\xi$ the composition of $\xi_{z}^{(i)}$ and the natural map $\widehat{G_{z}^{(i)}} \to \widehat{G}$ gives a tuple $(H, \mathcal{H}, \dot{s}, \xi)$ which one verifies is a refined endoscopic datum. Applying the construction discussed in Proposition \ref{liftend} (using $\dot{s}_{z}^{(i)}$ as the lift of $\dot{s}$) yields a family of $4$-tuples $\{(H_{z}^{(j)}, \mathcal{H}_{z}^{(j)}, \dot{s}^{(j)}_{z}, \xi_{z}^{(j)})\}_{j \geq i}$, and it is easy to check that for all $j \geq i$ this agrees with the datum constructed from $(H_{z}^{(i)}, \mathcal{H}_{z}^{(i)}, \dot{s}_{z}^{(i)}, \xi_{z}^{(i)})$ in the previous paragraph.


In view of the above discussion, it makes sense to define the set 
$$\varinjlim_{i} (\{\text{Refined endoscopic data for $G^{(i)}_{z}$}\}/\text{Isom}),$$
and we have:

\begin{corollary}\label{endcor} The map from Proposition \ref{liftend} (and the discussion preceding it) defines a bijection between isomorphism classes of refined endoscopic data for $G$ and the above inductive limit. Moreover, the image of a datum under this map is an elliptic datum (by which we mean each term in the inductive system is elliptic) if and only if the original datum is. 
\end{corollary}

\begin{proof} The first part is immediate from the above discussion. The second follows from \cite[Lemma 5.16]{kaletha18}.
\end{proof}

\subsection{Transfer factors}\label{EndsocopyTransfer}

In order to compare transfer factors for the data $\dot{\mathfrak{e}}$ and $\{\dot{\mathfrak{e}}_{j}\}$ we still need to pass between $z$-pairs for $G$ and systems of $z$-pairs for the $G_{z}^{(j)}$, which we turn to now. This is almost identical to the mixed characteristic case, but we summarize it here for completeness. Fix $\dot{\mathfrak{e}} = (H, \mathcal{H}, \dot{s}, \xi)$ and $\dot{\mathfrak{e}}_{z}^{(i)} = (H_{z}^{(i)}, \dot{s}_{z}^{(i)}, \xi_{z}^{(i)})$ which correspond to each other under the bijection of Corollary \ref{endcor} and $\mathfrak{z}_{z}^{(i)} = (H_{z,1}^{(i)}, \xi_{z,1}^{(i)})$ a $z$-pair for $\dot{\mathfrak{e}}_{z}^{(i)}$. After possibly enlarging $i$, we can assume that $H \to H_{z}^{(i)}$ is a weak $z$-embedding (and so this is also the case for any $j \geq i$). We define a $z$-pair $\mathfrak{z}_{1}$ for $\dot{\mathfrak{e}}$ by defining $H_{1} = H_{z,1}^{(i)} \times_{H_{z}^{(i)}} H$ and since the composition $$\mathcal{H}_{z}^{(i)} \to \prescript{L}{}H_{z,1}^{(i)} \to \prescript{L}{}H_{1}$$ factors uniquely through $\mathcal{H}$ we can define $\xi_{1} \colon \mathcal{H} \to \prescript{L}{}H_{1}$ as the induced map. Moreover, we can define a $z$-pair $\mathfrak{z}_{z}^{(j)}$ for any $j \geq i$ by taking $$H_{z,1}^{(j)} = H_{z,1}^{(i)} \times_{H_{z}^{(i)}} H_{z}^{(j)} = H_{1} \times^{Z(G)} Z(G_{z}^{(j)})$$ and $\xi_{z,1}^{(j)}$ the map
$$\mathcal{H}^{(j)}_{z} = \mathcal{H} \times_{\prescript{L}{}G} \prescript{L}{}G_{z}^{(j)}\xrightarrow{\xi_{1} \times \text{id}} \prescript{L}{}H^{(1)}_{z,1} \times_{\prescript{L}{}G} \prescript{L}{}G_{z}^{(j)} = \prescript{L}{}H_{z,1}^{(j)}.$$

We now give an analogue of \cite[Lemma 5.17]{kaletha18}:

\begin{proposition}\label{liftzpair} For any $j \geq i$, the following statements hold: \begin{enumerate} \item The map $H_{1} \to H_{z,1}^{(j)}$ has cokernel $C^{(j)}$, is a weak $z$-embedding, and $\{H_{1} \to H_{z,1}^{(k)}\}_{k \geq j}$ is a good $z$-embedding system.

\item The maps $H_{1} \to H$ and $H_{z,1}^{(j)} \to H_{z}^{(j)}$ are $z$-extensions.

\item $(H_{1}, \xi_{1})$ is a $z$-pair for $\dot{\mathfrak{e}}$, $(H_{z,1}^{(j)}, \mathfrak{z}_{z,1}^{(j)})$ is a $z$-pair for each $\dot{\mathfrak{e}}_{z}^{(j)}$ (the latter is an endoscopic datum since $j \geq i$, cf. Proposition \ref{liftend}), and the natural map 
$$\varinjlim_{k} (\{\tn{$z$-pairs for $\dot{\mathfrak{e}_{k}}$}\}/\tn{Isom}) \to  \{\tn{$z$-pairs for $\dot{\mathfrak{e}}$}\}/\tn{Isom}$$
has fibers which are torsors under the group $\varinjlim_{k} Z^{1}(W_{F}, \widehat{C^{(k)}})$ which acts on each $z$-pair $\mathfrak{z}_{z}^{(k)}$ by multiplication on the second component.

\end{enumerate}
\end{proposition}

\begin{proof} It is easy to check that if $K$ is the kernel of $H_{z,1}^{(i)} \to H_{z}^{(i)}$, then we have the short exact sequences
\begin{equation}\label{zextSES}
0 \to K \to H_{1} \to H \to 0, \hspace{1mm} 0 \to K \to H_{z,1}^{(j)} \to H_{z}^{(j)} \to 0
\end{equation} for any $j$, and also that the quotient of $H_{z,1}^{(j)}$ by $H_{1}$ is $C^{(j)}$. It follows that the natural map $H^{1}(F, H_{1}) \to H^{1}(F, H_{z,1}^{(j)})$ is surjective, and it is injective because the map $H^{(j)}_{z,1}(F) \to C^{(j)}(F)$ factors as a composition of two surjections $$H^{(j)}_{z,1}(F) \to H^{(j)}_{z}(F) \to C^{(j)}(F),$$ where the first map is surjective because $K$ is induced (by construction of a $z$-extension) and the second is surjective because we have chosen $j \geq i$. The rest of statement (1) above is a straightforward verification. 

For (2), the sequences \eqref{zextSES} imply that all relevant kernels are induced, and the fact that $H_{1} \to H_{1,z}^{(i)}$ and each $H_{1} \to H_{z,1}^{(j)}$ are embeddings with abelian cokernels means that they all have isomorphic derived subgroups. As in the mixed characteristic case, statement (3) follows from the fact that the kernels of each $\mathcal{H}_{z}^{(j)} \to \mathcal{H}$ and $^{L}H_{z,1}^{(j)} \to \prescript{L}{}H_{1}$ are $\widehat{C^{(j)}}$.
\end{proof}

Now we can compare normalized transfer factors. Recall that the maps $T \mapsto T \cdot Z(G_{z}^{(j)})^{\circ}$ and $B \to B \cdot Z(G_{z}^{(j)})^{\circ}$ induce bijections between sets of $F$-splittings and Whittaker data between the two groups for any $j$, and we can use this to pass from a Whittaker datum $\mathfrak{w}$ for $G$ to a (compatible) family of Whittaker data $\{\mathfrak{w}_{z}^{(j)}\}$, as in Section \ref{rigidcomp}. We can now deduce the main result of this section:

\begin{proposition}\label{lifttransf} Let $\dot{\mathfrak{e}}$ be a fixed refined endoscopic datum for $G$ with corresponding system of refined endoscopic data $\{\dot{\mathfrak{e}}_{z}^{(j)}\}_{j \geq i}$ for some $i \geq 0$, as in Proposition \ref{liftend}. Fix also a rigid inner twist $(\psi, \mathcal{T}, \bar{h})$ of $G$ with underlying group $G'$, which, via Lemma \ref{transfertwist}, gives a compatible family of rigid inner twists $(\psi_{z}^{(j)}, \mathcal{T}_{z}^{(j)}, \bar{h}_{z}^{(j)})$ of each $G_{z}^{(j)}$, and a $z$-pair $\mathfrak{z} = (H_{1}, \xi_{1})$ for $\dot{\mathfrak{e}}$ which corresponds to the system of $z$-pairs $\{\mathfrak{z}_{z}^{(j)}\}_{j \geq i}$ as in Proposition \ref{liftzpair}. 

For any $\gamma_{1} \in H_{1}(F)$ strongly $G'$-regular related to $\delta' \in G'(F)$, we have
$$\Delta'[\dot{\mathfrak{e}}, \mathfrak{z}, \mathfrak{w}, (\mathcal{T}, \bar{h})](\gamma_{1}, \delta') = \Delta'[\dot{\mathfrak{e}}_{z}^{(j)}, \mathfrak{z}_{z}^{(j)}, \mathfrak{w}_{z}^{(j)}, (\mathcal{T}_{z}^{(j)}, \bar{h}_{z}^{(j)})](\gamma_{1}, \delta')$$ for any $j \geq i$. In the above, $\Delta'$ denotes the normalized absolute transfer factor originally defined in \cite[\S 5.3]{kaletha16} in mixed characteristic and extended to equal characteristic in \cite[\S 7.2, 7.3]{Dillery1}.
\end{proposition}

\begin{proof} The mixed characteristic proof (as in \cite[Lemma 5.18]{kaletha18}) works verbatim here for any fixed $j$.
\end{proof}

As an application, we can now relate the endoscopic character identities for $\{G_{z}^{(i)}\}$ to those for $G$. We first recall the statement of these aforementioned identities, and for notational ease we state them for $x := (G', \mathcal{T}, \bar{h})$ a rigid inner twist of $G$. Fix an $L$-parameter $\varphi$ with refined endoscopic datum $\dot{\mathfrak{e}} = (H, \mathcal{H}, \dot{s}, \xi)$ and $z$-pair $\mathfrak{z} = (H_{1}, \xi_{1})$. Define the virtual character $$\Theta_{\varphi,x}^{\dot{s}} := e (G') \sum_{\pi \in \Pi_{\varphi}(G')} \langle \pi, \dot{s} \rangle \Theta_{\pi},$$ where $e(G')$ is Kottwitz's sign (as in \cite{kottwitz83}), $\Theta_{\pi}$ is the Harish-Chandra character of $\pi$, and the pairing is the one induced by $\iota_{\mathfrak{w}, x}$. For $\dot{s} = 1$ we expect this to be a stable  distribution independent of $\mathfrak{w}$ and $x$. 

Our formulation of the endoscopic character identity has the following form: for any strongly-regular semisimple $\delta' \in G'(F)$, we have the equality 
\begin{equation}\label{endchar}
\Theta_{\varphi,x}^{\dot{s}}(\delta') = \sum_{[\gamma_{1}]} \Delta'[\dot{\mathfrak{e}}, \mathfrak{z}, \mathfrak{w}, (\mathcal{T}, \bar{h})](\gamma_{1}, \delta')\Delta_{IV}(\gamma_{1}, \delta')^{-2}\Theta^{1}_{\xi_{1} \circ \varphi, \tn{triv}}(\gamma_{1}),
\end{equation}
where the sum runs over the stable conjugacy classes $[\gamma_{1}]$ in $H_{1}$ of all strongly-regular semisimple elements $\gamma_{1} \in H_{1}(F)$ and $\tn{triv}$ denotes the trivial rigid inner twist. To relate stable conjugacy classes in $H_{1}$ and $H_{z,1}^{(j)}$ we need:

\begin{lemma}\label{stableconj} If $\gamma_{1} \in H_{z,1}^{(j)}(F)$ is related to $\delta' \in G'(F)$ then $\gamma_{1}$ actually lies in $H_{1}(F)$.
\end{lemma}

\begin{proof} It's enough to show that the image $\gamma$ of $\gamma_{1}$ in $H_{z}^{(j)}(F)$ lies in $H(F)$. For a strongly-regular semisimple element $j \in J(F)$ for a reductive group $J$, denote by $T^{j}$ its (scheme-theoretic) centralizer. By definition there is an admissible embedding $T^{\gamma} \xrightarrow{h} (G')_{z}^{(j)}$ sending $\gamma$ to $\delta'$. We're done if we can show that $h^{-1}(G') = T^{\gamma} \cap H$; the right-to-left containment follows from the fact that $$h(T^{\gamma} \cap (H_{z}^{(j)})_{\text{der}}) \subset T^{\delta'} \cap [(G')_{z}^{(j)}]_{\text{der}} \subset G'$$ and the equality $T^{\gamma} \cap H = [T^{\gamma} \cap (H_{z}^{(j)})_{\text{der}}] \cdot Z(G)$ (using also that $h$ is the identity on $Z(G_{z}^{(j)})$).

The other containment can be checked after base-change to $\overline{F}$, where it follows from $H_{z}^{(j)}(\overline{F}) = Z(G_{z}^{(j)})(\overline{F}) \cdot H(\overline{F})$, $(G')_{z}^{(j)}(\overline{F}) = Z(G_{z}^{(j)})(\overline{F}) \cdot G'(\overline{F})$, and the fact that $h$ is the identity on $Z(G_{z}^{(j)})$.
\end{proof}

For our fixed rigid inner twist $x$, refined endoscopic datum $\dot{\mathfrak{e}}$ and $z$-pair $\mathfrak{z}$ we have a compatible system of: Rigid inner twists $\{x_{z}^{(j)}\}_{j \geq 0}$, refined endoscopic data $\{\dot{\mathfrak{e}}_{z}^{(j)}\}_{j \gg 0}$, and $z$-pairs $\{\mathfrak{z}_{z}^{(j)}\}$ as in Propositions \ref{liftend}, \ref{liftzpair}.

\begin{proposition} Let $\varphi$ be a tempered $L$-parameter for $G$ with a family of tempered parameters $\{\varphi_{z}^{(i)}\}$ lifting $\varphi$ (this always exists, cf. Proposition \ref{liftparam}). Then if the endoscopic character identity \eqref{endchar} holds for $(G')_{z}^{(j)}$ for the parameter  $\varphi_{z}^{(j)}$, rigidification $x_{z}^{(j)}$, endoscopic datum $\dot{\mathfrak{e}}_{z}^{(j)}$ and $z$-pair $\mathfrak{z}_{z}^{(j)}$ for all $j \gg 0$, then it holds for $G$, $\varphi$, $\dot{\mathfrak{e}}$, and $\mathfrak{z}$ using the bijection $\iota_{\mathfrak{w}}$ constructed from $\iota_{\mathfrak{w}_{z}^{(j)}}$ for each $G_{z}^{(j)}$ in Section \ref{rigidcomp}.
\end{proposition}

\begin{proof} As far as the left-hand side of \eqref{endchar} goes, the value $\Theta_{\varphi,x}^{\dot{s}}(\delta')$ equals any $\Theta_{\varphi_{z}^{(j)},x_{z}^{(j)}}^{\dot{s}_{z}^{(j)}}(\delta')$ by construction of the correspondence for $G'$. For the right-hand side, we first note that by Lemma \ref{stableconj}, all strongly regular semisimple stable conjugacy classes in any $H_{z,1}^{(j)}$ that are related to $\delta'$ (the only ones which have non-zero contribution to the sum in \eqref{endchar}) can represented be elements $\gamma_{1}$ of $H_{1}(F)$, and then by Proposition \ref{lifttransf}, we have $$\Delta'[\dot{\mathfrak{e}}, \mathfrak{z}, \mathfrak{w}, (\mathcal{T}, \bar{h})](\gamma_{1}, \delta') = \Delta'[\dot{\mathfrak{e}}_{z}^{(j)}, \mathfrak{z}_{z}^{(j)}, \mathfrak{w}_{z}^{(j)}, (\mathcal{T}_{z}^{(j)}, \bar{h}_{z}^{(j)})](\gamma_{1}, \delta').$$ The $\Delta_{IV}$ factors agree by part of the proof of \cite[Lemma 5.18]{kaletha18}, and $\Theta^{1}_{\xi_{1} \circ \varphi}(\gamma_{1})$ equals any $\Theta^{1}_{\xi_{z,1}^{(j)} \circ \varphi_{z}^{(j)}}(\gamma_{1})$, again by our construction of the refined local Langlands correspondence in Section \ref{rigidcomp} (for $H_{1}$, using the good $z$-embedding system $\{H_{z,1}^{(j)}\}$, cf. Proposition \ref{liftzpair}).
\end{proof}

\begin{remark}\label{weakremark3} If one uses the twisted modification of the LLC from Remark \ref{weakremark2} adapted to the assumptions in Remark \ref{weakremark}, then the above proof still works, replacing each $\varphi_{z}^{(j)}$ with a twist by something in $H^{1}(W_{F}, Z(G_{z}^{(j)}))$ whose induced character of $(G')_{z}^{(j)}(F)$ is trivial on $G'(F)$. 
\end{remark}

\section{Cohomology of finite group schemes on $\mathcal{E}$}\label{Tateduality}
\label{sec:cohom}
\subsection{The setup}\label{TatedualitySetup}
Our goal is to prove the following result, for $Z$ a finite multiplicative $F$-group scheme. For this section we can replace $\mathbb{C}$ by $C$ an arbitrary algebraically closed field of characteristic zero. 

\begin{proposition}\label{Zduality} There is a canonical, functorial isomorphism 
\begin{equation}\label{Zdualeq}
H^{1}(\mathcal{E}, Z) \xrightarrow{\sim} Z^{1}(\Gamma, \widehat{Z})^{*},
\end{equation}
where $\widehat{Z} := \mathrm{Hom}(X^{*}(Z), \mathbb{Q}/\mathbb{Z}) \cong \mathrm{Hom}(X^{*}(Z), C^\times).$
\end{proposition}

\begin{remark} We briefly explain the utility of Proposition \ref{Zduality} in the context of this paper. Recall that $H^{1}(\mathcal{E}, Z)$ acts on the fibers of
\begin{equation*}
H^{1}(\mathcal{E}, Z \to G) \to H^{1}(F, G_{\tn{ad}})
\end{equation*}
by multiplication and, in fact, this makes them into $H^{1}(\mathcal{E}, Z)$-torsors. The above result is used for understanding how the irreducible representation of $\pi_{0}(S_{\varphi}^{+})$ associated via the rigid refined LLC to a representation $(x, \pi)$ of a rigid inner twist $x$ changes if we replace $x$ by $x'$ another rigid inner twist inducing the same inner form. More precisely, if $x' = zx$ for $z \in H^{1}(\mathcal{E}, Z)$, then (assuming the existence of rigid refined LLC) the representation of $\pi_{0}(S_{\varphi}^{+})$ corresponding to $(x', \pi)$ must be the one for $(x, \pi)$ twisted by a character canonically obtained from the image of $z$ in $Z^{1}(\Gamma, \widehat{Z})^{*}$ via \eqref{Zdualeq}. For the full details of this construction, see \cite[\S 6.3]{kaletha18} and \cite[\S 4.4]{DS23}.

Another advantage of Proposition \ref{Zduality} (and, more specifically, understanding the effect that changing the rigidifcation of an inner form has on rigid refined LLC) becomes apparent when comparing rigid refined LLC to isocrystal refined LLC, which will be our task in \S \ref{Iso}. The latter approach replaces the cohomology set $H_{\tn{bas}}^{1}(\mathcal{E}, G)$ with $B(G)_{\tn{bas}}$, which is, as recalled in detail in \S \ref{IsoReview}, the set of isomorphism classes of basic $G$-isocrystals. We will construct in \S \ref{IsoComp} (cf. \eqref{cohomcomp}) a map
\begin{equation}\label{rigtoiso1}
B(G)_{\tn{bas}} \to H^{1}_{\tn{bas}}(\mathcal{E}, G)
\end{equation}
which allows one to interpolate between the two versions of the refined LLC. Although this map is in general far from surjective, when $Z(G)$ is connected it will be surjective up to multiplication by $H^{1}(\mathcal{E}, Z)$. Therefore, combining the comparison for elements in the image of \eqref{rigtoiso1} with Proposition \ref{Zduality} allows for a ``complete" comparison in the connected center case. 
\end{remark}

We follow the general framework of \cite[\S 6.1, 6.2]{kaletha18}. As in loc. cit., the first step is giving an interpretation of $H^{1}(\mathcal{E}, Z)$ in terms of Galois modules. 

Recall that for the absolute Galois group $\Gamma$ of $F$ and a finite discrete $\Gamma$-module $M$ (on which $\Gamma$ acts continuously), we can define \mathdef{Tate $-2$-cochains (resp. cocycles)}, denoted by $\widehat{C}^{-2}(\Gamma, M)$ (resp. $\widehat{Z}^{-2}(\Gamma, M)$), as follows: First, recall that for finite Galois $E/F$ such that $\Gamma_E$ acts trivially on $M$, a $-2$-cochain is a function $\Gamma_{E/F} \xrightarrow{f} M$, and it is a $1$-cycle (or a $-2$-cocycle) if $\sum_{g \in \Gamma_{E/F}} (g^{-1} -1)f(g) = 0$; denote the $-2$-cochains, cocycles, and coboundaries by $\widehat{C}^{-2}(\Gamma_{E/F}, M)$, $\widehat{Z}^{-2}(\Gamma_{E/F}, M)$, and $\widehat{B}^{-2}(\Gamma_{E/F}, M)$, respectively. For $K/E/F$ there are obvious transition maps on all of the above groups, and we define 
\begin{equation}\label{Hminus2}
\widehat{H}^{-2}(\Gamma, M) = \varprojlim \frac{\widehat{Z}^{-2}(\Gamma_{E/F}, M)}{\widehat{B}^{-2}(\Gamma_{E/F}, M)}, \hspace{1mm} \frac{\widehat{C}^{-2}(\Gamma, M)}{\widehat{B}^{-2}(\Gamma, M)} : =\varprojlim \frac{\widehat{C}^{-2}(\Gamma_{E/F}, M)}{\widehat{B}^{-2}(\Gamma_{E/F}, M)},
\end{equation}
as we vary over all finite Galois extensions through which the action of $\Gamma$ factors. 

For \textit{any} discrete $\Gamma$-module $M$, we define $\widehat{Z}^{-1}(\Gamma, M)$ as the colimit $\varinjlim \widehat{Z}^{-1}(\Gamma_{K/F}, M)$ with respect to the obvious inclusions, similarly with $\widehat{H}^{-1}(\Gamma, M)$. Note that if our group of coefficients $M$ is finite (and hence the norm will eventually act by zero), we simply have $\widehat{Z}^{-1}(\Gamma, M) = M.$

\begin{remark} In the $p$-adic case, Langlands shows in \cite{Lan83} that the left-hand limit in \eqref{Hminus2} stabilizes at a finite level and \cite[\S 6.1]{kaletha18} shows that the right-hand limit also stabilizes. However, they need not stabilize in the local function field case (for example, using duality, one sees that the lack of stabilization for the left-hand limit is equivalent to the fact that inflation $H^{1}(\Gamma_{E/F}, M) \xrightarrow{\tn{inf}} H^{1}(\Gamma, M)$ is in general not surjective for any finite $E/F$). 
\end{remark}

\begin{remark} For any torus $Z \hookrightarrow S$ defined over $F$, the group $\frac{X_{*}(S/Z)}{X_{*}(S)}$ is canonically isomorphic to $X^{*}(Z)^{*} := \Hom(X^{*}(Z), \mathbb{Q}/\mathbb{Z}) \cong \Hom(X^{*}(Z), C^\times)$; it will frequently be convenient to use this identification, and we do so without comment in the following.
\end{remark}

We have from \cite[Theorem 4.8]{kaletha16}, \cite[Theorem 4.10, Proposition 6.4]{Dillery1}) (and local class field theory) that for any torus $Z \hookrightarrow S$, the diagram
\begin{equation}\label{bigdiag1}
\begin{tikzcd}
0 \arrow{r} & H^{1}(F, Z) \arrow{r} \arrow{d} & H^{1}(\mathcal{E}, Z) \arrow{d} \arrow{r} & \Hom_{F}(u,Z) \arrow{d} \\
0 \arrow{r} & H^{1}(F, S) \arrow{r} \arrow{d} & H^{1}(\mathcal{E}, Z \to S) \arrow{r} \arrow{d} & \Hom_{F}(u,Z) \arrow{d} \\
& H^{1}(F,S) \arrow{r} & H^{1}(F, S/Z) \arrow{d}  \arrow{r} & H^{2}(F, Z) \arrow{d} \\
& & 0 & 0
\end{tikzcd}
\end{equation}
is isomorphic to
\begin{equation}\label{dualdiagram1}
\begin{tikzcd}
0 \arrow{r} & \widehat{H}^{-2}(\Gamma, \frac{X_{*}(S/Z)}{X_{*}(S)}) \arrow{r} \arrow{d} & ? \arrow{d} \arrow{r} & \widehat{Z}^{-1}(\Gamma, \frac{X_{*}(S/Z)}{X_{*}(S)}) \arrow{d} \\
0 \arrow{r} & \widehat{H}^{-1}(\Gamma, X_{*}(S)) \arrow{r} \arrow{d} & \frac{\widehat{Z}^{-1}(\Gamma, X_{*}(S/Z))}{\widehat{B}^{-1}(\Gamma, X_{*}(S))} \arrow{r} \arrow{d} & \widehat{Z}^{-1}(\Gamma, \frac{X_{*}(S/Z)}{X_{*}(S)}) \arrow{d} \\
& \widehat{H}^{-1}(\Gamma, X_{*}(S)) \arrow{r} &  \widehat{H}^{-1}(\Gamma, X_{*}(S/Z))  \arrow{r} & \widehat{H}^{-1}(\Gamma, \frac{X_{*}(S/Z)}{X_{*}(S)})  \arrow{d} \\
& & & 0,
\end{tikzcd}
\end{equation}
and we need to both determine what goes in missing ?-spot and show that it is (compatibly) isomorphic to $H^{1}(\mathcal{E}, Z)$. 

Returning to the two diagrams, it is easy to check that the group 
\begin{equation}\label{dualmodule} \frac{\widehat{C}^{-2}(\Gamma, \frac{X_{*}(S/Z)}{X_{*}(S)})}{\widehat{B}^{-2}(\Gamma, \frac{X_{*}(S/Z)}{X_{*}(S)})},
\end{equation} 
fits into the missing spot in the second diagram (the only non-immediate map is the vertical one; it is defined by lifting $1$-chains to $1$-chains valued in $X_{*}(S/Z)$ and then taking the differential). There are no issues with exactness after passing to the limits implicit in \eqref{dualdiagram1} because all of the groups in the relevant projective systems are finite. It remains to construct the isomorphism of \eqref{dualmodule} with $H^{1}(\mathcal{E}, Z)$.

In the $p$-adic case \cite{kaletha18} constructs the isomorphism by choosing a torus $Z \hookrightarrow S$ such that the canonical map $H^{1}(\mathcal{E}, Z) \to H^{1}(\mathcal{E}, Z \to S)$ is an isomorphism, as is its counterpart 
 $$\frac{\widehat{C}^{-2}(\Gamma, \frac{X_{*}(S/Z)}{X_{*}(S)})}{\widehat{B}^{-2}(\Gamma, \frac{X_{*}(S/Z)}{X_{*}(S)})} \to \frac{\widehat{Z}^{-1}(\Gamma, X_{*}(S/Z))}{\widehat{B}^{-1}(\Gamma, X_{*}(S))},$$
 which one can do in light of:

\begin{proposition}\label{clevertorus}(\cite[Proposition 5.2]{kaletha18}) For $F$ a $p$-adic local field, one can choose a torus $Z \to S$ such that $H^{1}(F, S/Z) = 0$ and $H^{1}(F, Z) \to H^{1}(F, S)$ is bijective. 
\end{proposition}

From there the desired isomorphism comes from the analogous one for $H^{1}(\mathcal{E}, Z \to S)$; one checks that it is independent of the choice of $S$ and functorial in $Z$, see \cite[\S 6.1]{kaletha18}.

\subsection{Unbalanced cup products}\label{TatedualityCup}
We briefly recall cup products and unbalanced cup products in the setting of \v{C}ech cohomology which will be used to prove the above duality isomorphism over function fields. The only novel definition in this subsection is of the unbalanced cup product for $-2$-cochains.

For $S/R$ a finite flat extension (not necessarily étale) of rings, $S'/R$ a Galois extension contained in $S$, and $G$ a commutative $R$-group scheme, define the group $C^{n}(S/R, G)$ to be $G(S^{\bigotimes_{R}(n+1)})$, and  $C^{n,i}(S/R, S', G)$ to be the subgroup $G(S^{\bigotimes_{R} (n+1-i)} \otimes_{R} (S'^{\bigotimes_{R} i}))$. Setting $\Gamma' := \text{Aut}_{R\text{-alg}}(S')$, our goal is to define an unbalanced cup product $$C^{n,i}(S/R, S', G) \times C_{\text{Tate}}^{-i}(\Gamma', H(S')) \xrightarrow{\underset{S'/R}{\sqcup}} C^{n-i}(S/R, J)$$ for two commutative $R$-group schemes $G, H$ and $R$-pairing $P \colon G \times H \to J$ for $J$ another commutative $R$-group scheme, $i=1,2$, where as above $C_{\text{Tate}}^{-1}(\Gamma', H(S')) = H(S')$, $C_{\text{Tate}}^{-2}(\Gamma', H(S'))$ is functions $\Gamma' \to H(S')$. We now drop the ``Tate" subscript for these degrees for notational convenience. Note that in previous work \cite[\S 4.2]{Dillery1}, only the definition for $i=1$ was given. We start by defining the standard cup product in \v{C}ech cohomoloogy:

\begin{definition}\label{fppfcup} We recall the general cup product on \v{C}ech cochains as defined in \cite[\S 3]{Shatz}, continuing with the notation from above. For $n,m \geq 1$, we have the morphism of $R$-algebras $S^{\bigotimes_{R}n} \otimes_{R} S^{\bigotimes_{R} m}  \xrightarrow{\theta} S^{\bigotimes_{R} (n+m-1)}$ defined on simple tensors by $$(a_{1} \otimes \dots \otimes a_{n}) \otimes (b_{1} \otimes \dots \otimes b_{m}) \mapsto a_{1} \otimes \dots \otimes a_{n-1} \otimes a_{n}b_{1} \otimes b_{2} \otimes \dots \otimes b_{m}.$$
The cup product 
\begin{equation*}
    C^{n-1}(S/R, G) \times C^{m-1}(S/R, H) \to C^{n+m-2}(S/R, J)
\end{equation*}
is defined to be the composition of $$G(S^{\bigotimes_{R} n}) \times H(S^{\bigotimes_{R} m}) \xrightarrow{(\text{id} \otimes 1)^{\sharp} \times (1 \otimes \text{id})^{\sharp}} G(S^{\bigotimes_{R}n} \otimes_{R} S^{\bigotimes_{R} m}) \times H(S^{\bigotimes_{R}n} \otimes_{R} S^{\bigotimes_{R} m}) = (G \times H)(S^{\bigotimes_{R}n} \otimes_{R} S^{\bigotimes_{R} m})$$ with the map $$(G \times H)(S^{\bigotimes_{R}n} \otimes_{R} S^{\bigotimes_{R} m}) \xrightarrow{\theta^{\sharp}} (G \times H)(S^{\bigotimes_{R} (n+m-1)}) \xrightarrow{P} J(S^{\bigotimes_{R} (n+m-1)}).$$
\end{definition}

With the above definition in hand, we resume our construction of the unbalanced cup product. We have a homomorphism of $R$-algebras $\lambda \colon S^{\bigotimes_{R} n} \otimes_{R} S' \to \prod_{\Gamma'} S^{\bigotimes_{R} n}$ defined on simple tensors by $$a_{i,1} \otimes \dots \otimes a_{i,n+1} \mapsto (a_{i,1} \otimes \dots \otimes a_{i,n} \cdot ^{\sigma}a_{i,n+1} )_{\sigma \in \Gamma'}.$$ Moreover, for any $R$-group scheme $J$, we have a canonical identification $J(\prod_{\Gamma'} S^{\bigotimes_{R} n}) \to \prod_{\Gamma'} J(S^{\bigotimes_{R} n})$; 

We first recall the definition for $-1$-cochains Define, for $a \in G(S^{\bigotimes_{R} n} \otimes_{R} S')$ and $b \in H(S')$, $$a \tilde{\underset{S'/R}{\sqcup}} b = \lambda^{\sharp}(a \cup b^{(0)}) \in \prod_{\Gamma'} J(S^{\bigotimes_{R} n}).$$ In the above formula we are using the cup product from Definition \ref{fppfcup} (which in this case is the pairing $P$ applied to $(a, p_{n+1}^{\sharp}(b^{(0)})) \in G(S^{\bigotimes_{R} (n+1)}) \times H(S^{\bigotimes_{R} (n+1)})$) and $b^{(0)}$ denotes the element $b \in H(S')$ viewed as a $0$-cochain.

We now apply the group homomorphism $N \colon \prod_{\Gamma'} J(S^{\bigotimes_{R} n}) \to J(S^{\bigotimes_{R} n})$ obtained by taking the sum of all elements on the left-hand side, and the resulting pairing 
$$C^{n,1}(S/R, S', G) \times C^{-1}(\Gamma', H(S')) \to J(S^{\bigotimes_{R} n})$$ is $\mathbb{Z}$-bilinear.  Indeed, $$(a +a') \tilde{\underset{S'/R}{\sqcup}} b = \lambda^{\sharp}[(a+a') \cup b^{(0)}] = \lambda^{\sharp}(a\cup b^{(0)} + a' \cup b^{(0)}) = \lambda^{\sharp}(a \cup b^{(0)}) + \lambda^{\sharp}(a' \cup b^{(0)}).$$ This will be our desired pairing, denoted by $a \underset{S'/R}{\sqcup} b$. 

Note that we could also have defined $a \underset{S'/R}{\sqcup} b$ by first applying $\lambda^{\sharp}$ to $a$, obtaining $\lambda^{\sharp}(a) = (a_{\sigma})_{\sigma \in \Gamma'} \in \prod_{\Gamma'} G(S^{\bigotimes_{R} n})$ and then taking the sum
$$a \underset{S'/R}{\sqcup} b = \sum_{\sigma \in \Gamma'} a_{\sigma} \cup (^{\sigma}b^{(0)});$$ 
it is straightforward to verify that this agrees with the first definition.

This second description of the unbalanced cup product gives a more obvious way of extending the definition to $-2$-cochains $b \colon \Gamma' \to H(S')$. For $a \in G(S^{\bigotimes_{R} n} \otimes_{R} S' \otimes_{R} S')$ and $b \colon \Gamma' \to H(S')$, we first take 
\begin{equation*}
\lambda^{\sharp}(a) = (a_{\underline{\sigma}})_{\underline{\sigma} \in (\Gamma')^{2}} \in \prod_{(\Gamma')^{2}} G(S^{\bigotimes_{R} n}),
\end{equation*}
where now $\lambda$ is the ring isomorphism
\begin{equation*} S^{\bigotimes_{R} n} \otimes_{R} S' \otimes_{R} S' \to \prod_{(\Gamma')^{2}} S^{\bigotimes_{R} n}, \hspace{1mm} a_{1} \otimes \dots \otimes a_{n} \otimes b_{1} \otimes b_{2} \mapsto (a_{1} \otimes \dots a_{n-1} \otimes a_{n} \cdot ^{\sigma_{1}}b_{1} \cdot ^{(\sigma_{1}\sigma_{2})}b_{2})_{(\sigma_{1}, \sigma_{2}) \in (\Gamma')^{2}}.
\end{equation*}
We now define
\begin{equation*}
a \underset{S'/R}{\sqcup} b := \sum_{\underline{\sigma} = (\sigma_{1}, \sigma_{2}) \in (\Gamma')^{2}} a_{\underline{\sigma}} \cup (^{\sigma_{1}\sigma_{2}}b(\sigma_{1}\sigma_{2})^{(0)}),
\end{equation*}
where, as above, $^{\sigma_{1}\sigma_{2}}b(\sigma_{1}\sigma_{2})^{(0)}$ indicates that we are viewing the element $^{\sigma_{1}\sigma_{2}}b(\sigma_{1}\sigma_{2}) \in H(S')$ as a $0$-cochain.

A key property of the unbalanced cup product is the following (cf. \cite[Proposition 4.4]{Dillery1}):

\begin{proposition}\label{bigcup} For $a \in C^{n,i}(S/R, S', G)$ and $b \in C^{-i}(\Gamma',H(S'))$, $i=1,2$, we have $$d(a \underset{S'/R}{\sqcup} b) = (da) \underset{S'/R}{\sqcup} b + (-1)^{n} (a \underset{S'/R}{\sqcup} db).$$
\end{proposition}
In the above result, the unbalanced cup product is defined to be the usual cup product when the right-hand argument is not a negative Tate cochain (this occurs in the $i=1$ case).

\begin{proof} This is proved loc. cit. for $i=-1$, and a straightforward adaptation of the argument gives the result for $i=-2$ as well.
\end{proof}

\subsection{Duality via cup products}\label{TatedualityCupDual}

We will give a functorial isomorphism between $H^{1}(\mathcal{E}, Z)$ and the abelian group \eqref{dualmodule} in the function field case, which is the essential ingredient in the proof of Proposition \ref{Zduality}, by a different method from that of \cite{kaletha18}; there is no hope of an analogue of Proposition \ref{clevertorus} in this situation (even if one tries to replace the torus by a pro-torus in order to preserve injectivity, surjectivity is lost due to the failure of the universal residue symbol in local class field theory to be surjective). Our argument will work for any non-archimedean local field.

The proof will require some further recollections from \cite{Dillery1}; recall from Section \ref{PrelimKal} that we have fixed a system of finite Galois extensions $\{E_{k}/F\}$ and natural numbers $\{n_{k}\}$ which are totally ordered and co-final with respect to the limit defining $u$. 

We assume that, for our fixed group $Z$, that $Z$ is split over $E_{k}$, $\text{exp}(Z) \mid n_{k}$, and $\widehat{Z}^{-1}(\Gamma_{E_{k}/F}, X^{*}(Z)^{*})$ are constant (i.e., are all just equal to $X^{*}(Z)^{*}$) for all $k$ (we can always arrange for this to be the case by shifting the system, since our constructions only depend on its limit anyway, as explained in \cite[\S 4]{Dillery1}). We can define a \v{Cech} $2$-cochain (from \cite[\S 4.3]{Dillery1}), denoted $\widetilde{l_{k}c_{k}} \in \mathbb{G}_{m}(\overline{F}^{\bigotimes_{F} 3})$, by taking a certain choice $c_{k}$ of representative for the canonical class $H^{2}(\Gamma_{E_{k}/F}, E_{k}^{\times})$ and then strategically defining an $n_{k}$th root $\widetilde{l_{k}c_{k}}$; both of these choices will be compatible for varying $k$. This cochain has the property that it takes values in the subgroup $\mathbb{G}_{m}(\overline{F} \otimes_{F} E_{k} \otimes_{F} E_{k}) \subset \mathbb{G}_{m}(\overline{F}^{\bigotimes_{F} 3})$ and hence its differential $d(\widetilde{l_{k}c_{k}})$ is a $3$-cocycle lying in the subgroup $\mu_{n_{k}}(\overline{F} \otimes_{F} \overline{F} \otimes_{F} E_{k} \otimes_{F} E_{k}) \subset \mathbb{G}_{m}(\overline{F}^{\bigotimes_{F} 4})$; both of these inclusions will be used for taking unbalanced cup products.

We define the ``evaluation" homomorphism $$\mu_{n_{k},E_{k}} \xrightarrow{\delta_{e}} u_{k,E_{k}}$$ given on character modules by $$\mathbb{Z}/n_{k}\mathbb{Z}[\Gamma_{E_{k}}/F]_{0} \to \mathbb{Z}/n_{k}\mathbb{Z}, \sum_{\sigma} c_{\sigma}[\sigma] \mapsto c_{e}.$$ Note that it is killed by the $\Gamma_{E_{k}/F}$-norm map. We also sometimes denote this  map by $\delta_{e,k}$ if we are working with multiple terms in the projective system.

There is an isomorphism 
\begin{equation*}
\Hom_{F}(u_{k}, Z) \xrightarrow{\sim} \widehat{Z}^{-1}(\Gamma_{E_{k}/F}, X^{*}(Z)^{*}) = \widehat{Z}^{-1}(\Gamma, X^{*}(Z)^{*}) = X^{*}(Z)^{*}
\end{equation*}
(the last two equality are because of our assumptions on $k$) via precomposing with $\delta_{e}$, and then one obtains the isomorphism $\widehat{Z}^{-1}(\Gamma, X^{*}(Z)^{*}) \xrightarrow{\sim} \Hom_{F}(u, Z)$ by taking the direct limit of the finite-level isomorphisms, which are easily seen to agree as $k$ varies (because of how we have chosen our system relative to $Z$, it follows that every map $u \to Z$ factors through any $u_{k}$, in particular for $k=1$).

A key property of the cochains $\widetilde{l_{k}c_{k}}$ is the following:

\begin{proposition}(\cite[Lemma 4.8]{Dillery1}) Set $$\xi_{k} := d(\widetilde{l_{k}c_{k}}) \sqcup_{E_{k}/F} \delta_{e} \in u_{k}(\overline{F}^{\bigotimes_{F} 3}).$$ Then $[\xi_{k}]$ equals the image of the canonical class $[\xi] \in H^{2}(F, u)$ in $H^{2}(F, u_{k})$.
\end{proposition}
 
In the following we will work with the gerbe $$\mathcal{E} := \varprojlim_{k}\mathcal{E}_{\xi_{k}},$$ which, in view of the preceding Lemma, represents the canonical class in $H^{2}(F,u)$ (recall from \S \ref{PrelimKal} that we may use any such gerbe). The transition maps will be defined shortly. Fix a cochain $y \in \widehat{C}^{-2}(\Gamma, X^{*}(Z)^{*})$, with image $y_{k} \in \widehat{C}^{-2}(\Gamma_{E_{k}/F}, X^{*}(Z)^{*})$. We then define a $1$-cochain by taking the unbalanced cup product
\begin{equation}\label{zykeq}
z_{y_{k}}:=  d(\widetilde{l_{k}c_{k}}) \sqcup_{E_{k}/F} y_{k} \in Z(\overline{F} \otimes_{F} \overline{F});
\end{equation}
note that this is well-defined because of our above remark about $d(\widetilde{l_{k}c_{k}})$, and we are using the obvious pairing 
\begin{equation*} \mu_{n_{k}} \times \underline{\Hom(X^{*}(Z), \frac{1}{n_{k}}\mathbb{Z}/\mathbb{Z})} \to Z
\end{equation*}
in the above cup product.

The $1$-cochain \eqref{zykeq} will be part of our data in the duality isomorphism between $H^{1}(\mathcal{E}, Z)$ and \eqref{dualmodule}; we want to extract from the cochain $y$ a $\xi$-twisted $Z$-torsor, or, equivalently, a $\xi$-twisted cocycle valued in $Z$ as defined in \S \ref{PrelimGerbes} (cf. \cite[\S 2.4, \S 2.5]{Dillery1} for more details), so we still need to give a homomorphism $u \to Z$ over $F$. For this we first take the image $$dy = dy_{k} \in \widehat{Z}^{-1}(\Gamma_{E_{k}/F}, X^{*}(Z)^{*});$$
note that, by construction, the image in the colimit $\widehat{Z}^{-1}(\Gamma, X^{*}(Z)^{*})$ is independent of the choice of $k$.

We then denote by $\lambda_{y_{k}}$ the unique morphism $u_{k} \to Z$ such that $\lambda_{y_{k}} \circ \delta_{e}= dy_{k}$. Again, we emphasize that these $\lambda_{y_{l}}$ are compatible as $l$ varies, by construction.

\begin{lemma}\label{twistedtors} The pair $(-z_{y_{k}}, \lambda_{y_{k}})$ is a $\xi_{k}$-twisted cocycle valued in $Z$. Moreover, its image in $H^{1}(\mathcal{E}_{\xi_{k}}, Z)$ is independent of the coset of $y_{k}$ modulo $\widehat{B}^{-2}(\Gamma_{E_{k}/F}, X^{*}(Z)^{*})$ up to equivalence of $\xi_{k}$-twisted cocycles valued in $Z$.
\end{lemma}

\begin{remark} The sign change is due to the sign coming from Proposition \ref{bigcup}.
\end{remark}

\begin{proof} The only thing to show for the first part is that $$d(-z_{y_{k}}) = \lambda_{y_{k}}(\xi_{k}).$$ We compute using Proposition \ref{bigcup} that 
$$d(-z_{y_{k}}) = d(\widetilde{l_{k}c_{k}}) \sqcup_{E_{k}/F} dy_{k} =  d(\widetilde{l_{k}c_{k}}) \sqcup_{E_{k}/F} (\lambda_{y_{k}} \circ \delta_{e}).$$ Note that $\lambda_{y_{k}}$, by construction, is defined over $F$, and hence we have (using \cite[Lemma 4.6]{Dillery1})
$$d(\widetilde{l_{k}c_{k}}) \sqcup_{E_{k}/F} (\lambda_{y_{k}} \circ \delta_{e}). = [d(\widetilde{l_{k}c_{k}}) \sqcup_{E_{k}/F} \delta_{e}] \cup \lambda_{y_{k}}^{(0)},$$ where we have added the superscript ``$(0)$" to denote the fact that in the last term we are viewing $\lambda_{y_{k}}$ as a $0$-chain rather than a $-1$-chain.

The first claim then follows from the equality
$$[d(\widetilde{l_{k}c_{k}}) \sqcup_{E_{k}/F} \delta_{e}] \cup \lambda_{y_{k}}^{(0)} = \lambda_{y_{k}}([d(\widetilde{l_{k}c_{k}}) \sqcup_{E_{k}/F} \delta_{e}] ) = \lambda_{y_{k}}(\xi_{k}).$$

For the second claim, suppose that we replace $y$ by $y + dx$ for $x \in \widehat{C}^{-3}(\Gamma, X^{*}(Z)^{*})$ (defined in the obvious way). Then the morphism $\lambda_{y_{k}}$ is unchanged, and we compute that 
$$d(\widetilde{l_{k}c_{k}}) \sqcup_{E_{k}/F} (y_{k} + dx_{k}) = z_{y_{k}} + d(\widetilde{l_{k}c_{k}}) \sqcup_{E_{k}/F} dx_{k},$$ and we want to show that this is equivalent to $(z_{y_{k}}, \lambda_{y_{k}})$ as $\xi_{k}$-twisted cocycles valued in $Z$.

To do this, we reintroduce an $F$-torus $Z \hookrightarrow S$; we may lift the cochain $dx_{k} \in \widehat{C}^{-2}(\Gamma_{E_{k}/F}, X^{*}(Z)^{*})$ to a cochain $dx_{k}^{(S)} \in  \widehat{C}^{-2}(\Gamma_{E_{k}/F}, X_{*}(S))$, and from here we compute that
$$ d(\widetilde{l_{k}c_{k}}) \sqcup_{E_{k}/F} dx_{k} = d(\widetilde{l_{k}c_{k}} \sqcup_{E_{k}/F} dx_{k}^{(S)}),$$ where in the right-hand expression $\widetilde{l_{k}c_{k}} \sqcup_{E_{k}/F} dx_{k}^{(S)} \in S(\overline{F})$. The aforementioned unbalanced cup product is well-defined because $\widetilde{l_{k}c_{k}} \in C^{2,2}(\overline{F}/F, E_{k}, \mathbb{G}_{m})$ and $dx_{k}^{(S)}$ is a $-2$-cochain. It was necessary to introduce the torus $S$ because the cochain $\widetilde{l_{k}c_{k}}$ is valued in $\mathbb{G}_{m}$.

It follows that the composition
$$\widehat{C}^{-2}(\Gamma_{E_{k}/F}, X^{*}(Z)^{*}) \to H^{1}(\mathcal{E}, Z) \to H^{1}(\mathcal{E}, Z \to S)$$ 
kills the subgroup $\widehat{B}^{-2}(\Gamma_{E_{k}/F}, X^{*}(Z)^{*})$ for any torus $Z \to S$. For any nonzero $v \in H^{1}(\mathcal{E}_{\xi_{k}}, Z)$, the proof of Proposition \ref{clevertorus} in \cite{kaletha18} shows that we may find a torus $Z \hookrightarrow S_{v}$ such that $v$ is not in the image of the map $(S_{v}/Z)(F) \to H^{1}(\mathcal{E}_{\xi_{k}}, Z)$ (which factors through $H^{1}(F, Z)$), and thus the image of $v$ in $H^{1}(\mathcal{E}_{\xi_{k}}, Z \to S)$ is nonzero. Combining this fact with the vanishing of the above composition for any $S$ gives the second part of the desired result.
\end{proof}

At this point, we need to recall the precise definition of the transition maps $$\mathcal{E}_{\xi_{k+1}} \to \mathcal{E}_{\xi_{k}}.$$
Since the image of $\xi_{k+1}$ in $u_{k}$ is not identically equal to $\xi_{k}$, it is necessary to introduce the $1$-cochains 
$$\alpha_{k} := (\widetilde{l_{k}c_{k}} \sqcup_{E_{k}/F} \delta_{e,k})^{-1} \cdot (\widetilde{l_{k+1}c_{k+1}} \sqcup_{E_{k+1}/F} p_{k+1,k} \circ \delta_{e,k+1}).$$ When we write $\delta_{e,k}$ in the above cup products, we mean the canonical extension of the previous map $\delta_{e,k}$ to a homomorphism (still killed by the norm)
\begin{equation*}
(\mathbb{G}_{m})_{E_{k}} \to (\frac{\mathrm{Res}_{E_{k}/F}(\mathbb{G}_{m})}{\mathbb{G}_{m}})_{E_{k}}.
\end{equation*}
The key property of these cochains is:

\begin{proposition}(\cite[Lemma 4.8]{Dillery1}) The cochains $\alpha_{k}$ are valued in $u_{k}$ and satisfy
$$d\alpha_{k} = p_{k+1,k}(\xi_{k+1})\xi_{k}^{-1}.$$
\end{proposition}

We can then define the transition maps via translation by $\alpha_{k}$ (cf. \cite[Construction 2.38]{Dillery1}). At the level of twisted cocycles, pullback by this transition map sends $(z, \lambda)$ to $(z \cdot \lambda(\alpha_{k}), \lambda \circ p_{k+1,k})$.

\begin{corollary} The assignment $y = (y_{k})_{k} \mapsto (-z_{y_{k}}, \lambda_{y_{k}})$ gives a well-defined map
$$\frac{\widehat{C}^{-2}(\Gamma, X^{*}(Z)^{*})}{\widehat{B}^{-2}(\Gamma, X^{*}(Z)^{*})} \to H^{1}(\mathcal{E}, Z).$$
\end{corollary}

\begin{proof} The map is defined on each term of the limit by pulling back the torsor defined in Lemma \ref{twistedtors} to a torsor on $\mathcal{E}$. We claim that for any $y_{l}$, $l > k$, the corresponding isomorphism class of torsors coincides with the image of $y_{k}$. One checks easily that $\lambda_{y_{k+1}} = \lambda_{y_{k}} \circ p_{k+1,k}$, so the content of this claim is that one twisted cocycle pulls back to the other (up to equivalence) via the above transition maps. 

We need to study the expression 
$$[d(\widetilde{l_{k}c_{k}}) \sqcup_{E_{k}/F} y_{k}] \cdot \lambda_{y_{k}}(\alpha_{k}) \cdot [d(\widetilde{l_{k+1}c_{k+1}}) \sqcup_{E_{k+1}/F} y_{k+1}]^{-1} \in Z(\overline{F} \otimes_{F} \overline{F}),$$
and we are done if we can show that it is a coboundary (in particular, by the proof of Lemma \ref{twistedtors}, it suffices to show that it is a coboundary in $Z^{1}(\mathcal{E}_{\xi_{k+1}}, Z \to S)$ for any torus $Z \hookrightarrow S$.)

Fixing a torus $Z \hookrightarrow S$, we compute (cf. the proof of the first part of Lemma \ref{twistedtors}) that 
$$\lambda_{y_{k}}(\widetilde{l_{k}c_{k}} \sqcup_{E_{k}/F} \delta_{e,k}) = \widetilde{l_{k}c_{k}} \sqcup_{E_{k}/F} dy_{k}^{(S)},$$
and identically with $k+1$. It immediately follows that our main expression
$$[d(\widetilde{l_{k}c_{k}}) \sqcup_{E_{k}/F} y_{k}] \cdot \lambda_{y_{k}}(\alpha_{k}) \cdot [d(\widetilde{l_{k+1}c_{k+1}}) \sqcup_{E_{k+1}/F} y_{k+1}]^{-1} = d(\beta_{k} \cdot \beta_{k+1}^{-1}),$$
where $$\beta_{i} := \widetilde{l_{i}c_{i}} \sqcup_{E_{i}/F} y_{i}^{(S)}$$ for $i=k, k+1$, giving the result.
\end{proof}

In summary, we have constructed the candidate morphism
$$\frac{\widehat{C}^{-2}(\Gamma, X^{*}(Z)^{*})}{\widehat{B}^{-2}(\Gamma, X^{*}(Z)^{*})} \to H^{1}(\mathcal{E}, Z),$$
and we will now turn to showing that it is an isomorphism. A main component of our proof is:

\begin{lemma} For any torus $Z \hookrightarrow S$, the following diagram commutes
\[
\begin{tikzcd}
\frac{\widehat{C}^{-2}(\Gamma, X^{*}(Z)^{*})}{\widehat{B}^{-2}(\Gamma, X^{*}(Z)^{*})} \arrow{r} \arrow{d} & H^{1}(\mathcal{E}, Z) \arrow{d} \\
\frac{\widehat{Z}^{-1}(\Gamma, X_{*}(S/Z))}{\widehat{B}^{-1}(\Gamma, X_{*}(S))} \arrow{r} & H^{1}(\mathcal{E}, Z \to S).
\end{tikzcd}
\]
The map is also compatible with the duality isomorphism $\widehat{H}^{-2}(\Gamma, X^{*}(Z)^{*}) \xrightarrow{\sim} H^{1}(F, Z)$.
\end{lemma}

\begin{proof} This is a routine calculation, which we leave as an exercise (using basic properties of the unbalanced cup product developed in \cite[\S 4.2]{Dillery1}).
\end{proof}

\begin{theorem}\label{finiteduality}
The morphism 
$$\frac{\widehat{C}^{-2}(\Gamma, X^{*}(Z)^{*})}{\widehat{B}^{-2}(\Gamma, X^{*}(Z)^{*})} \to H^{1}(\mathcal{E}, Z)$$
constructed above is an isomorphism.
\end{theorem}

\begin{proof} This map is compatible with all the duality isomorphisms between the two diagrams \eqref{bigdiag1} \eqref{dualdiagram1}, so the result follows by the five-lemma.
\end{proof}

\subsection{Relation to the dual group}\label{TatedualityDualgp}
We are now able to finish proving Proposition \ref{Zduality}. Unlike the previous subsection, this part of the result is identical to its analogue in \cite[\S  6.2]{kaletha18}, and we summarize the situation here. Fixing $Z \hookrightarrow S$, we set $\overline{S} := S/Z$, as usual.

We claim that taking the Pontryagin dual of the diagram \eqref{dualdiagram1} yields:
\begin{equation}\label{dualdiagram2}
\begin{tikzcd}
1 & H^{1}(F, \widehat{Z}) \arrow{l} & Z^{1}(\Gamma, \widehat{Z}) \arrow{l} & \frac{\widehat{C}^{0}(\Gamma, \widehat{Z})}{\widehat{B}^{0}(\Gamma, \widehat{Z})} \arrow["-d"]{l} \\
1 & \widehat{H}^{0}(\Gamma, \widehat{S}) \arrow["-\delta"]{u} \arrow{l} & \pi_{0}(\widehat{\overline{S}}^{\Gamma,+}) \arrow["-d"]{u} \arrow{l} & \frac{\widehat{C}^{0}(\Gamma, \widehat{Z})}{\widehat{B}^{0}(\Gamma, \widehat{Z})} \arrow{l} \arrow{u} \\
& \widehat{H}^{0}(\Gamma, \widehat{S}) \arrow{u} & \widehat{H}^{0}(\Gamma, \widehat{\overline{S}}) \arrow{l} \arrow{u} & \widehat{H}^{0}(\Gamma, \widehat{Z}). \arrow{l} \arrow{u}
\end{tikzcd}
\end{equation}
All of the above descriptions of the Pontryagin duals come from the identifications induced by $X_{*}(S) = X^{*}(\widehat{S})$, $X_{*}(\overline{S}) = X^{*}(\widehat{\overline{S}})$. The fact that the description holds for the terms in the top row involving $\widehat{Z}$ (in the function field case, the corresponding terms in \eqref{dualdiagram1} are inverse limits that do not stabilize) follows from applying inverse limits to \cite[Lemma 6.1]{kaletha18} and then using the fact that 
$$\varprojlim Z^{1}(\Gamma_{E_{k}/F}, \widehat{Z})^{*} = [\varinjlim Z^{1}(\Gamma_{E_{k}/F}, \widehat{Z})]^{*} = Z^{1}(\Gamma, \widehat{Z}).$$

\noindent Combining the above diagram with Theorem \ref{finiteduality} finishes the proof of Proposition \ref{Zduality}.

\section{Isocrystal comparison}\label{Iso}
This section discusses the passage from the rigid refined LLC to the isocrystal refined LLC and proceeds identically (modulo minor \v{C}ech/gerbe-theoretic modifications) here as in the mixed-characteristic case, and, as such, the goal of this section is to summarize the analogous discussion in \cite[\S 4]{kaletha18} (with the exception of \S \ref{IsoFunc} which justifies the sufficiency of Conjecture \ref{Maartenbisbis}).

\subsection{Review of isocrystal LLC}\label{IsoReview}
Denote by $\mathbb{D}$ the pro-multiplicative group with character group $\mathbb{Q}$ and $\mu$ the one with character group $\mathbb{Q}/\mathbb{Z}$; recall that the \mathdef{Kottwitz gerbe} (sometimes also called the \mathdef{Dieudonn\'{e} gerbe}) is the fpqc $\mathbb{D}$-gerbe, denoted by $\tn{Kott}$, corresponding to the canonical class in $H^{2}(F, \mathbb{D})$ defined as the image of $1 \in H^{2}_{\tn{fppf}}(F, \mu) = \varprojlim_{n} H^{2}(F, \mu_{n}) = \widehat{\mathbb{Z}}$. For $G$ a connected reductive group, we define $B(G)_{\tn{bas}}$ as the set of isomorphism classes of $G$-torsors on $\tn{Kott}$ whose $\mathbb{D}$-action is induced by a homomorphism $\mathbb{D} \to Z(G)$; these can also be viewed concretely as \mathdef{basic $G$-isocrystals} (see e.g. \cite[\S 3]{Kottwitz97}). There is a canonical map $B(G)_{\tn{bas}} \to H^{1}(F, G_{\tn{ad}})$ that is surjective whenever $Z(G)$ is connected, as well as a canonical ``Tate-Nakayama duality" isomorphism 
\begin{equation}
B(G)_{\tn{bas}} \xrightarrow{\kappa} [Z(\widehat{G})^{\Gamma}]^{*}
\end{equation}
which is functorial in $G$. An \mathdef{extended pure inner twist} of $G$ is a pair $(\psi, x_{\tn{iso}})$ where $G' \xrightarrow{\psi} G$ is an inner twist of $G$ and $x_{\tn{iso}}$ is a $G$-isocrystal on $\tn{Kott}$ whose image in $Z^{1}(F, G_{\tn{ad}})$ equals the cocycle corresponding to $\psi$.

We can now state the simplest form of the isocrystal refined LLC: For $G$ a quasi-split connected reductive group with fixed Whittaker datum $\mathfrak{w}$ and tempered $L$-parameter $\varphi$, there should be set of isomorphism classes of tempered representations of extended pure inner twists $\Pi^{\tn{iso}}_{\varphi}$ and bijection 
\begin{equation}\label{isoLLC} \Pi^{\text{iso}}_{\varphi} \xrightarrow{\iota_{\mathfrak{w}}^{\tn{iso}}} \tn{Irr}(S_{\varphi}^{\natural}),
\end{equation}
where
\begin{equation*}
S_{\varphi}^{\natural} := \frac{S_{\varphi}}{(S_{\varphi} \cap \widehat{G}_{\tn{der}})^{\circ}}.
\end{equation*}
There are natural maps $\Pi^{\text{iso}}_{\varphi} \to B(G)_{\tn{bas}}$ and $\tn{Irr}(S_{\varphi}^{\natural}) \to [Z(\widehat{G})^{\Gamma}]^{*}$ and, via these maps, the bijection \eqref{isoLLC} should be compatible with $\kappa$. 

There is also a notion of the endoscopic character identities for \eqref{isoLLC} whose formulation requires some further recollections about transfer factors. For an extended pure inner twist $(\psi, x_{\tn{iso}})$ with underlying inner twist $G'$ and an endoscopic datum $\mathfrak{e}$ with $z$-pair $\mathfrak{z}$ for $G'$, one can define a normalized transfer factor $\Delta'[\mathfrak{w}, \mathfrak{e}, \mathfrak{z}, (\psi, x_{\tn{iso}})]$ identically as in the mixed-characteristic case using the relative transfer factor for local function fields constructed in \cite[\S 7]{Dillery1}. As in the rigid situation, we can then define, for any $s \in S_{\varphi}$ semisimple, the virtual character
\begin{equation} \Theta_{\varphi, [x_{\tn{iso}}]}^{s} := e(G') \cdot \sum_{\pi \in \Pi^{\text{iso}}_{\varphi}([x_{\tn{iso}}])} \langle \pi, s \rangle \Theta_{\pi},
\end{equation}
where $\Pi^{\text{iso}}_{\varphi}([x_{\tn{iso}}])$ denotes all representations in $\Pi^{\text{iso}}_{\varphi}$ which are isomorphic to representations whose underlying element of $B(G)_{\tn{bas}}$ is $[x_{\tn{iso}}]$, the terms $e(G')$, $\Theta_{\pi}$ are as defined in \S \ref{EndsocopyTransfer}, and the pairing is from \eqref{isoLLC}. If $H \in \mathfrak{e}$ is the endoscopic group, $s \in \mathfrak{e}$ the semisimple element of $S_{\varphi}$, $H_{1} \in \mathfrak{z}$ the group from the $z$-pair, $\varphi_{H}$ the induced parameter for $H$, and $f \in \mathcal{C}^{\infty}_{c}(G'(F))$, $f^{H} \in \mathcal{C}^{\infty}_{c}(H_{1}(F))$ are $\Delta'[\mathfrak{w}, \mathfrak{e}, \mathfrak{z}, (\psi, x_{\tn{iso}})]$-matching functions, then the \mathdef{endoscopic character identities} are the equalities 
\begin{equation}
\Theta^{1}_{\varphi_{H}, \tn{triv}}(f^{H}) = \Theta^{s}_{\varphi, [x_{\tn{iso}}]}(f).
\end{equation}

\subsection{Comparison with rigid LLC}\label{IsoComp}

By \cite[Proposition 3.6]{Dillery1} there is a canonical map $u \to \mu$ such that the induced map $H^{2}(F, u) \to H^{2}(F, \mu)$ sends the canonical class of the left-hand side defining $\mathcal{E}$ to $1 \in H^{2}(F, \mu)$, and so we have a morphism $\mathcal{E} \to \tn{Kott}$ of fibered categories over $F$ which on bands is the composition $u \to \mu \to \mathbb{D}$; denote this composition by $\phi$. Pullback by this map gives a homomorphism 
\begin{equation}\label{cohomcomp} B(G)_{\tn{bas}} \to H^{1}_{\tn{bas}}(\mathcal{E}, G),
\end{equation}
which we warn the reader need not be injective or surjective in general.

There is an analogue of \eqref{cohomcomp} on the Galois side. Recall $Z_{n} := (Z(G)/Z_{\tn{der}})[n]$ and $G_{n} = G/Z_{n}$ as introduced in \S \ref{Endoscopy}; we observe that $G_{1} \xrightarrow{\sim} G_{\tn{ad}} \times \frac{Z(G)}{Z_{\tn{der}}}$ so that $G_{n} \xrightarrow{\sim} G_{\tn{ad}} \times \frac{Z(G)}{Z_{\tn{der}}}[n]$ and there is an identification
\begin{equation*} \widehat{G_{n}} \xrightarrow{\sim} \widehat{G}_{\tn{sc}} \times Z(\widehat{G})^{\circ},
\end{equation*}
where, for $n \mid m$ the dual transition map $\widehat{G_{m}} \to \widehat{G_{n}}$ is multiplication by $m/n$. This allows us to identify $\widehat{\bar{G}}$ with $\widehat{G}_{\tn{sc}} \times \widehat{C}_{\infty}$, where $\widehat{C}_{\infty} = \varprojlim Z(\widehat{G})^{\circ}$ with transition maps given as in the preceding sentence.

There is a canonical group homomorphism 
\begin{equation}\label{hatcomp}
\widehat{\bar{G}} = \widehat{G}_{\tn{sc}} \times \widehat{C}_{\infty} \to \widehat{G}, \hspace{1mm} (a, (b_{n})) \mapsto a_{\tn{der}} \cdot \frac{b_{1}}{N_{E/F}(b_{[E \colon F]})},
\end{equation}
where $a_{\tn{der}}$ denotes the image of $a$ in $G_{\tn{der}}$ and the right-hand factor is constant for all sufficiently large finite Galois $E/F$, and we take such an $E/F$ in \eqref{hatcomp}. This map restricts to a homomorphism
\begin{equation}\label{Zhatcomp}
\pi_{0}((Z(\widehat{G}_{\tn{sc}}) \times \widehat{C}_{\infty})^{+}) \to Z(\widehat{G})^{\Gamma},   
\end{equation}
and it is proved in \cite[Proposition 3.3]{kaletha18} that the comparison map \eqref{cohomcomp} becomes the dual of \eqref{Zhatcomp} after applying Tate-Nakayama duality (the argument loc. cit. is valid in any characteristic).

Similarly, for a fixed $L$-parameter $\varphi$ for $G$, if $S_{\varphi}^{+}$ denotes the preimage of $S_{\varphi}$ in $\widehat{\bar{G}}$, then \eqref{hatcomp} restricts to a map $S_{\varphi}^{+} \to S_{\varphi}$ whose post-composition with the projection to $S_{\varphi}^{\natural}$ is trivial on $(S_{\varphi}^{+})^{\circ}$ and thus gives a homomorphism
\begin{equation}\label{Sphicomp}
    \pi_{0}(S_{\varphi}^{+}) \to S_{\varphi}^{\natural},
\end{equation}
which is what we'll use to compare the two versions of the refined LLC.

The key result (whose proof is the same as \cite[Lemma 4.1]{kaletha18}, its mixed-characteristic analogue) is:

\begin{proposition} Let $[x_{\tn{iso}}] \in B(G)_{\tn{bas}}$ with image in $H^{1}_{\tn{bas}}(\mathcal{E},G)$ denoted by $[x_{\tn{rig}}]$, which we can view as an element of $H^{1}(\mathcal{E}, Z(G) \to G) := \varinjlim_{n} H^{1}(\mathcal{E}, Z_{n} \to G)$. Denote by $\tn{Irr}(\pi_{0}(S_{\varphi}^{+}), [x_{\tn{rig}}])$ (resp. $\tn{Irr}(S_{\varphi}^{\natural}, [x_{\tn{iso}}])$) all irreducible representations of $\pi_{0}(S_{\varphi}^{+})$ (resp. $S_{\varphi}^{\natural}$)  that map to $[x_{\tn{rig}}]$ (resp. $[x_{\tn{iso}}]$) via restriction followed by Tate-Nakayama duality.

Pullback by the map \eqref{Sphicomp} induces a bijection
\begin{equation}\label{Irrcomp}
\tn{Irr}(S_{\varphi}^{\natural}, [x_{\tn{iso}}]) \to \tn{Irr}(\pi_{0}(S_{\varphi}^{+}), [x_{\tn{rig}}]).
\end{equation}
\end{proposition}

It immediately follows that the composition
\begin{equation}\label{finalLLC}
\Pi_{\varphi}^{\tn{iso}}([x_{\tn{iso}}]) \xrightarrow{\eqref{cohomcomp}} \Pi_{\varphi}^{\tn{rig}}([x_{\tn{rig}}]) \xrightarrow{\eqref{rigidLLC}} \tn{Irr}(\pi_{0}(S_{\varphi}^{+}), [x_{\tn{rig}}]) \xrightarrow{\eqref{Sphicomp}^{-1}} \tn{Irr}(S_{\varphi}^{\natural}, [x_{\tn{iso}}])  
\end{equation}
splices across all $[x_{\tn{iso}}] \in B(G)_{\tn{bas}}$ to define a bijection as in the isocrystal refined LLC \eqref{isoLLC}.

For a fixed tempered $L$-parameter and $\dot{s}_{\tn{rig}} \in S_{\varphi}^{+}$ let $\dot{\mathfrak{e}}$ denote the corresponding refined endoscopic datum $(H, \mathcal{H}, \xi, \dot{s}_{\tn{rig}})$ with $z$-pair $(H_{1}, \xi_{1})$. It is a basic fact of refined endoscopic data that if $s$ is the image of $\dot{s}_{\tn{rig}}$ in $S_{\varphi}$ via the natural projection, then $(H, \mathcal{H}, \xi, \dot{s}_{\tn{iso}})$ is an endoscopic datum serving $\varphi$ and $s$, with $z$-pair $\mathfrak{z}$. Moreover, since the image of $\dot{s}_{\tn{rig}}$ in $S_{\varphi}$ via \eqref{hatcomp}, denoted by $s_{\tn{iso}}$, differs from $s$ by $N_{E/F}(b_{[E \colon F]}) \in Z(\widehat{G})^{\Gamma}$, the data $H, \mathcal{H}, \xi$ also serve $\varphi$ and $s_{\tn{iso}}$, and so $\mathfrak{e}_{\tn{iso}} := (H, \mathcal{H}, \xi, s_{\tn{iso}})$ is also an endoscopic datum served by $\mathfrak{z}$.

The final comparison result is:

\begin{proposition}
For $\dot{\mathfrak{e}} = (H, \mathcal{H}, \xi, \dot{s}_{\tn{rig}})$, $\mathfrak{e}_{\tn{iso}}$ as above and rigid inner twist $[x_{\tn{rig}}]$ of $G$, the rigid endoscopic character identities for $[x_{\tn{rig}}]$ and $\dot{\mathfrak{e}}$ hold if and only if the analogous endoscopic character identities \eqref{endchar} hold for the isocrystal refined LLC given by \eqref{finalLLC} for any $[x_{\tn{iso}}]$ mapping to $[x_{\tn{rig}}]$ and endoscopic datum $\mathfrak{e}_{\tn{iso}}$. 

In particular, if $Z(G)$ is connected, then the rigid endoscopic character identities for all rigid inner twists of $G$ are equivalent to the isocrystal endoscopic character identities for all extended pure inner twists of $G$.
\end{proposition}

\begin{proof}
We have by construction that $\Theta_{\varphi, [x_{\tn{rig}}]}^{\dot{s}_{\tn{rig}}} =  \Theta_{\varphi, [x_{\tn{iso}}]}^{s_{\tn{iso}}}$ so it suffices to show that 
\begin{equation}\label{endcomp}
\Delta'[\mathfrak{w}, \mathfrak{e}_{\tn{iso}}, \mathfrak{z}, (\psi, x_{\tn{iso}})] = \Delta'[\mathfrak{w}, \dot{\mathfrak{e}}, \mathfrak{z}, x_{\tn{rig}}]
\end{equation}
as functions on $H_{1,G-\tn{sr}}(F) \times G_{\tn{sr}}(F)$, where as usual we are taking some fixed Whittaker datum $\mathfrak{w}$ for $G$. For matching elements $(\gamma, \delta') \in H_{1,G-\tn{sr}}(F) \times G_{\tn{sr}}(F)$ the only possible discrepancy between the two sides of \eqref{endcomp} are the values coming from pairing (via Tate-Nakayama duality) the torsors $\tn{inv}[x_{\tn{rig}}](\gamma, \delta')$ and $\tn{inv}[x_{\tn{iso}}](\gamma, \delta')$ (cf. \cite[\S 7.1]{Dillery1} for the construction of this torsor in the rigid case---we will review both rigid and isocrystal constructions later in this argument) with $\dot{s}_{\tn{rig}}$ and $s_{\tn{iso}}$, respectively. 

Fix $\delta \in G(F)$ and $g \in G(F^{\tn{sep}})$ such that $\psi(g\delta g^{-1}) = \delta'$ and set $S:= Z_{G}(\delta)$. To facilitate computations we choose a representative $\xi \in u(\overline{F}^{\bigotimes_{F} 3})$ of the class representing the gerbe $\mathcal{E}$, which implies (using the first paragraph of this section) that $\phi(\xi)$ has image in $H^{2}(F, \mathbb{D})$ that represents $\tn{Kott}$; we can and do view $H$-torsors (for a group scheme $H$) on $\mathcal{E}$ (resp. on $\tn{Kott}$) as $\xi$-twisted (resp. $\phi(\xi)$-twisted) cocycles with coefficients in $H$ (as defined in \S \ref{PrelimGerbes}). In particular, we can write $x_{\tn{rig}} = (c_{\tn{rig}}, f)$, where $u \xrightarrow{f} Z(G)$ and $c_{\tn{rig}} \in G(\overline{F} \otimes_{F} \overline{F})$ is such that $dc_{\tn{rig}} = f(\xi)$, which means that $x_{\tn{iso}} = (c_{\tn{rig}}, f')$ where $\mathbb{D} \xrightarrow{f'} Z(G)$ is such that $f = f' \circ \phi$. 

By construction, $\tn{inv}[x_{\tn{rig}}](\gamma, \delta') \in H^{1}(\mathcal{E}, Z(G) \to S) := \varinjlim_{Z \subset_{\tn{finite}} Z(G)} H^{1}(\mathcal{E}, Z \to S)$ is represented by the $\xi$-twisted cocycle $(p_{1}^{\sharp}(g)c_{\tn{rig}}p_{2}^{\sharp}(g)^{-1}, f)$, and $\tn{inv}[x_{\tn{iso}}](\gamma, \delta') \in B(S)$ by the $\phi(\xi)$-twisted $(p_{1}^{\sharp}(g)c_{\tn{rig}}p_{2}^{\sharp}(g)^{-1}, f')$ (in the literature this second invariant has only been defined in the setting of Galois cocycles, see e.g. \cite[\S 4.1]{kaletha18}, but it is a routine verification that the \v{C}ech definition given here agrees under the usual passage from gerbe to Galois gerbe as in \cite[\S 6]{Taibi22}). The explicit description of these two twisted cocycles makes it obvious that the image of $\tn{inv}[x_{\tn{iso}}](\gamma, \delta')$ in $H^{1}(\mathcal{E}, Z(G) \to S)$ under \eqref{cohomcomp} equals $\tn{inv}[x_{\tn{rig}}](\gamma, \delta')$. From here the result follows from the compatibility of rigid and isocrystal Tate-Nakayama duality and the fact that, by definition, $s_{\tn{rig}}$ maps to $s_{\tn{iso}}$ under \eqref{hatcomp}.

It follows immediately from the above argument that the rigid endoscopic character identities imply their isocrystal analogues. For the converse when $Z(G)$ is connected, it suffices to show that, for any $x \in H^{1}(\mathcal{E}, Z(G) \to G)$, the rigid endoscopic character identity \eqref{endchar} for $x$ is equivalent to the same identity with $x$ replaced by any $H^{1}(\mathcal{E}, Z_{n})$-translate $x'$, for any choice of $n$. This equivalence follows from \cite[Lemma 6.2]{kaletha18} (and, in turn, Lemma 6.3 loc. cit.), the proofs of which hold in our situation due to the dual diagrams \eqref{bigdiag1} and \eqref{dualdiagram2} constructed in \S \ref{Tateduality}.
\end{proof}

\subsection{Functoriality conjecture revisited}\label{IsoFunc}
The goal of this subsection is to show that the first part of Conjecture \ref{Maartenbis} is not required by using isocrystals and the construction of semisimple $L$-parameters from \cite{FS21}. Fix $G \to G_{z}$ a weak $z$-embedding and $Z \subset Z(G)$ a finite central subgroup. For an $L$-parameter $\varphi$ for $^{L}G$, denote by $\varphi^{\tn{ss}}$ its \textsf{semisimplification}---the homomorphism $W_{F} \to \prescript{L}{}G$ given by pre-composing $\phi$ with the map $W_{F} \to W_{F} \times \SL_{2}$ defined by the identity on the first factor and $w \mapsto \begin{pmatrix} ||w||^{1/2} & 0 \\ 0 & ||w||^{-1/2} \end{pmatrix}$ on the second factor. 

\begin{remark}\label{field}In order to use the results of \cite{FS21} it is necessary to fix a field isomorphism $\mathbb{C} \cong \overline{\mathbb{Q}_{\ell}}$ (for $\ell \neq \tn{Char}(F)$ a prime) which we use to transfer results between their setting and the setting of this paper.
\end{remark}

Since the results that we will use from \cite{FS21} deal with semisimple parameters, we need the following result:

\begin{lemma}\label{sslemma} Let $\tilde{\varphi}$ be a parameter for $^{L}G$ with lift $\varphi$ to $^{L}G_{z}$ (which exists by Proposition \ref{liftparam}). Then if the induced map 
\begin{equation}\label{JMLemMap}
    S_{\varphi^{\tn{ss}}} \to S_{\tilde{\varphi}^{\tn{ss}}}
\end{equation}
is surjective, the map $\pi_{0}(S_{\varphi}^{+}) \to  \pi_{0}(S_{\tilde{\varphi}}^{+})$ is an isomorphism. 
\end{lemma}

\begin{proof}
We prove this for $Z =1$; the argument for general $Z$ is identical (cf. the proof of Corollary \ref{keycor}). Note that the map $S_{\varphi^{\tn{ss}}} \to S_{\tilde{\varphi}^{\tn{ss}}}$ has central kernel; it thus maps centers to centers and induces a bijection on unipotent elements. It follows from \cite{AMIY} that we have a canonical isomorphism
\begin{equation*}
    \pi_{0}(S_{\varphi}) \xrightarrow{\sim} \pi_{0}(Z_{\widehat{G_{z}}}(\mathdef{JM}(\varphi)))
\end{equation*}
where $\mathdef{JM}(\varphi)$ denotes the image of $\varphi$ under the \mathdef{Jacobson-Morozov} map (by \cite[Theorem 3.6]{AMIY} the JM map induces a bijection on Frobenius-semisimple parameters, which include all parameters considered in this paper), which is the pair $(\varphi^{\tn{ss}}, N:= d\varphi(e_{0}))$ (with $e_{0} := \begin{pmatrix} 0 & 1 \\ 0 & 0 \end{pmatrix} \in \mathfrak{sl}_{2}$), and $Z_{\widehat{G_{z}}}(\mathdef{JM}(\varphi))$ denotes the elements of $\widehat{G_{z}}$ that fix both components of the pair. If $\mathdef{JM}(\tilde{\varphi}) = (\tilde{\varphi}^{\tn{ss}}, N)$ then clearly $\mathdef{JM}(\varphi) = (\varphi^{\tn{ss}}, N)$.

We deduce that to show that $\pi_{0}(S_{\varphi}) \to  \pi_{0}(S_{\tilde{\varphi}})$ is surjective (it is already injective, cf. the proof of Proposition \ref{limsurj}) it suffices to prove surjectivity of the map
\begin{equation*}
    Z_{\widehat{G_{z}}}(\mathdef{JM}(\varphi)) \to Z_{\widehat{G}}(\mathdef{JM}(\tilde{\varphi})),
\end{equation*}
which follows from combining the fact that the map \eqref{JMLemMap} is surjective with the observation that the preimage of $Z_{\widehat{G}}(\mathdef{JM}(\tilde{\varphi})) = Z_{S_{\tilde{\varphi}^{\tn{ss}}}}(N)$ in $S_{\varphi^{\tn{ss}}}$ is exactly $Z_{S_{\varphi^{\tn{ss}}}}(N) = Z_{\widehat{G_{z}}}(\mathdef{JM}(\varphi))$. 
\end{proof}

The key input here is \cite[\S I.9]{FS21}, which, for $G'$ an extended pure inner form of $G$, associates to any smooth $\pi \in \tn{Irr}(G'(F))$ a semisimple parameter $\psi$, which we call the \mathdef{FS parameter} of $\pi$. If $\varphi$ is the $L$-parameter for $\pi$ obtained via the rigid refined LLC, it is conjecturally expected that these two parameters are related via $\varphi^{\tn{ss}} = \psi$ (after applying the coefficient field identification of Remark \ref{field}). A key fact about the Fargues-Scholze correspondence for our applications is:

\begin{theorem}\label{FSthm} (\cite[Theorem I.9.6.(v)]{FS21}) 
    For any $G_{1} \xrightarrow{\eta} G_{2}$ which induces an isomorphism on adjoint groups and $\pi \in \tn{Irr}(G_{2}'(F))$ with corresponding FS parameter $\psi$, any irreducible constituent of $\pi \circ \eta'$ has FS parameter $^{L}\eta \circ \psi$.
\end{theorem}

We deduce the following consequence:

\begin{corollary}\label{pi0isom}
Fix a good $z$-embedding system $\{G \to G_{z}^{(i)}\}$ and a smooth irreducible representation $\pi^{(i)}$ of $(G')_{z}^{(i)}(F)$ (with corresponding parameter $\varphi^{(i)}$ for $^{L}G_{z}^{(i)}$) whose pullbacks $\pi^{(j)}$ to $(G')_{z}^{(j)}(F)$ are irreducible for all $j \geq i$. Then for some $j \gg i$ the parameter $\varphi^{(j)}$ for $\pi^{(j)}$ is such that the natural map $\pi_{0}(S_{\prescript{L}{}p_{k,j} \circ \varphi^{(j)}}^{+}) \to \pi_{0}(S_{\varphi^{(j)}}^{+})$ is an isomorphism for any $k \geq j$.
\end{corollary}

\begin{proof}
Denote by $\varphi^{\tn{ss}}$ the parameter for $^{L}G$ induced by $\varphi^{(i),\tn{ss}}$; Corollary \ref{keycor} implies that the natural map $S_{\prescript{L}{}p_{j,i} \circ \varphi^{(i),\tn{ss}}} \to S_{\varphi^{\tn{ss}}}$ is surjective for all $j \gg 0$. Fixing such a $j$, denote by $\varphi^{(j)}$ the parameter for $^{L}G^{(j)}_{z}$ associated to $\pi^{(j)}$ (which we do \textit{not} assume is induced by $\varphi^{(i)}$) and $\varphi$ the induced parameter of $^{L}G$. It follows from Theorem \ref{FSthm} that $\varphi^{(j),\tn{ss}} = \prescript{L}{}p_{j,i} \circ \varphi^{(i),\tn{ss}}$, and so by Lemma \ref{sslemma} the map $\pi_{0}(S_{\varphi^{(j)}}^{+}) \to \pi_{0}(S_{\varphi}^{+})$ is an isomorphism, as are the natural maps $\pi_{0}(S_{\prescript{L}{}p_{k,j} \circ \varphi^{(j)}}^{+}) \to \pi_{0}(S_{\varphi}^{+})$ for any $k \geq j$. We conclude that the natural map $\pi_{0}(S_{\prescript{L}{}p_{k,j} \circ \varphi^{(j)}}^{+}) \to \pi_{0}(S_{\varphi^{(j)}}^{+})$ is an isomorphism for all $k \geq j$, since they're both compatibly isomorphic to $\pi_{0}(S_{\varphi}^{+})$.
\end{proof}

We may thus (using Proposition \ref{funcprop1}) replace all key uses of Conjecture \ref{Maartenbis} in \S \ref{rigidcomp} by the weaker Conjecture \ref{Maartenbisbis}:

\begin{corollary}\label{conjsuffices}
One can replace Conjecture \ref{Maartenbis} by Conjecture \ref{Maartenbisbis} in the justification that equation \eqref{limLLC} is well-defined and in the proof of Proposition \ref{uniqueness} (and thus also for Corollary \ref{fullunique}).
\end{corollary}

\begin{proof}
The only remaining thing to show is that the lifting hypothesis of Conjecture \ref{Maartenbisbis} holds in our above situation. One can deduce that $\varphi^{(k)}$ (corresponding to $\pi^{(k)}$) as in the proof of Corollary \ref{pi0isom} above always lifts to a parameter for $^{L}G_{z}^{(j)}$. For this argument we will replace $L_{F}$ with $W_{F}' = \mathbb{G}_{a} \rtimes W_{F}$ (where $W_{F}$ acts via scaling by $||w||$) and view parameters $\varphi$ as pairs of $W_{F} \xrightarrow{\varphi^{\tn{ss}}} \prescript{L}{}G$ semisimple and $N = d\varphi(e_{0}) \in \mathrm{Lie}(\widehat{G})$ satisfying certain properties (see e.g. \cite[\S 5.2]{Taibi22} for details, including its equivalence to our previous notion of parameters). We know by Theorem \ref{FSthm} that $\varphi^{(k),\tn{ss}} = \prescript{L}{}p_{k,j} \circ \varphi^{(j),\tn{ss}}$, and so it suffices to lift the nilpotent element $N \in \mathrm{Lie}(\widehat{G_{z}^{(k)}})$ to a nilpotent $N^{(j)} \in \mathrm{Lie}(\widehat{G_{z}^{(j)}})$ such that $\varphi^{(j),\tn{ss}}(w)N^{(j)}\varphi^{(j),\tn{ss}}(w)^{-1} = ||w||N^{(j)}$ for all $w \in W_{F}$.

Since the maps from both $\widehat{G_{z}^{(j)}}$ and $\widehat{G_{z}^{(k)}}$ to $\widehat{G}$ are surjective with central kernel, they induce bijections on nilpotent elements, and hence so does the map $\widehat{G_{z}^{(j)}} \to \widehat{G_{z}^{(k)}}$. It is easy to verify that the lift of $N$ to $\mathrm{Lie}(\widehat{G_{z}^{(j)}})$ is the desired $N^{(j)}$.
\end{proof}

\appendix
\section{Braided crossed modules and duality}\label{Braided}
The purpose of this appendix is to justify the duality results involving braided crossed modules for local function fields used in \S \S \ref{ConstructionDual}, \ref{EndoscopySystems}. The arguments are essentially identical to the $p$-adic case, but we outline them here for completeness. The bulk of the work has already been done in the case of tori in \cite[Appendix A]{Dillery2}, which we will refer to often.

\subsection{Crossed modules}\label{BraidedCrossed}
We recall some basic notions of crossed modules, following \cite{Noohi11} (see also \cite{Borovoi16}, \cite{Aviles12}, \cite{labesse99}).

\begin{definition}\label{crosseddef}
A \mathdef{crossed module} $H \xrightarrow{\partial} G$ is two groups such that $G$ acts on the right of $H$ via group automorphisms and $\partial$ is a group homomorphism which is $G$-equivariant with respect to the conjugation action of $G$ on itself. We also require the following identity, for all $h_{1}, h_{2} \in H$:
\begin{equation*}
h_{1}^{\partial{h_{2}}} = h_{2}h_{1}h_{2}^{-1}.
\end{equation*}
For a (topological) group $\Gamma$, if we equip $H \xrightarrow{\partial} G$ with a (continuous) left $\Gamma$-action such that for any $\sigma \in \Gamma$ we have
\begin{equation*}
\prescript{\sigma}{}(h^{g}) = (\prescript{\sigma}{}{h})^{\prescript{\sigma}{}g},
\end{equation*}
then we call it a \mathdef{$\Gamma$-crossed module}.
\end{definition}

\begin{definition}
In the context of Definition \ref{crosseddef}, a \mathdef{$\Gamma$-braiding} is a map $\{-, -\} \colon G \times G \to H$ such that 
\begin{equation*}
\partial \{g_{1}, g_{2}\} = g_{1}^{-1}g_{2}^{-1}g_{1}g_{2}
\end{equation*}
and $\{\prescript{\sigma}{}g_{1}, \prescript{\sigma}{}g_{2}\} = \prescript{\sigma}{}\{g_{1},g_{2}\}$ for $\sigma \in \Gamma$. We say that the braiding is \mathdef{symmetric} if every $\{g_{1},g_{2}\}\{g_{2},g_{1}\} = 1$.
\end{definition}

Given the above data, one can define sets $Z^{i}(\Gamma, H \xrightarrow{\partial} G)$ and $H^{i}(\Gamma, H \xrightarrow{\partial} G)$ for $i = -1,0,1$, see \cite[\S 3, \S 4]{Noohi11}. Since we will be giving complete definitions in the \v{C}ech setting we do not write out the full Galois versions here. 

\begin{remark}\label{keyapprem}
For our applications in this paper, the above notion of crossed $\Gamma$-modules is all that is required, for reasons that will be explained later in this appendix. Nevertheless, we work with the \v{C}ech versions here for completeness and full generality.
\end{remark}

One can adapt the above framework to the more general setting of a faithfully flat morphism of rings $R \to S$ (for applications in this paper, this will just be $F \to \overline{F}$). The author is unsure if this appears as given below in the literature, but it is certainly already known. We also alert the reader that there are more conceptual approaches to this topic using torsors on topoi, see for example \cite{Aviles12}. Recall that all algebraic groups are assumed to be affine over $R$ (but not necessarily commutative).

Recall from \S \ref{PrelimCohom} that for an $S$-algebra $S'$, we denote the two projections
\begin{equation*}
\mathrm{Spec}(S') \times_{R} \mathrm{Spec}(S') \to \mathrm{Spec}(S')  
\end{equation*}
by $p_{1}$ and $p_{2}$. Similarly, for the $3$-fold cover denote the projection onto the $(i,j)$-embedded copy of $\mathrm{Spec}(S') \times_{R} \mathrm{Spec}(S')$ by $p_{i,j}$, and analogously for all $n$-fold covers over $m$-fold bases.

\begin{definition}\label{crosseddefbis}
A \mathdef{crossed module} $H \xrightarrow{\partial} G$ is two algebraic groups over $R$ such that $G$ acts on the right of $H$ via $R$-group automorphisms and $\partial$ is an $R$-group homomorphism which is $G$-equivariant with respect to the conjugation action of $G$ on itself. We also require the following identity, for all $h_{1}, h_{2} \in H(R')$, where $R'$ is an $R$-algebra:
\begin{equation*}
h_{1}^{\partial{h_{2}}} = h_{2}h_{1}h_{2}^{-1};
\end{equation*}
the above identity could also be expressed diagramatically but we omit that here for brevity.

We call $H \xrightarrow{\partial} G$ an \mathdef{$S/R$-crossed module} if the following identity holds for all $R$-algebras $R'$, $\bar{h} \in H(S \otimes_{R} R')$, and $\bar{g} \in G(S \otimes_{R} R')$:
\begin{equation*}
p_{2}^{\sharp}(\bar{h}^{\bar{g}}) = p_{2}^{\sharp}(\bar{h})^{p_{2}^{\sharp}(\bar{g})}.
\end{equation*}
One checks easily that when $S/R$ is finite Galois then this recovers Definition \ref{crosseddef} (for $\Gamma = \mathrm{Aut}_{R}(S)$).
\end{definition}

\begin{definition}
In the context of Definition \ref{crosseddefbis}, an \mathdef{$S/R$-braiding} is a morphism of $S$-schemes $\{-\} \colon G_{S} \times G_{S} \to H_{S}$ such that for any $R$-algebra $R'$ and all $g_{i} \in G(S \otimes_{R} R')$ we have $\partial \{g_{1},g_{2}\} = g_{1}^{-1}g_{2}^{-1}g_{1}g_{2}$ and $\{p_{2}^{\sharp}(g_{1}), p_{2}^{\sharp}(g_{2})\} = p_{2}^{\sharp}\{g_{1},g_{2}\}$. We say that the braiding is \mathdef{symmetric} if every $\{g_{1},g_{2}\}\{g_{2},g_{1}\} = 1$.
\end{definition}

We can now define the desired cohomology sets (which we will give the structure of groups):

\begin{definition}\label{maincrossedef} Given the above setup, we can make the following definitions:
\begin{enumerate}
\item{Set $H^{-1}(S/R, H \xrightarrow{\partial} G) := \mathrm{Ker}(\partial)(R)$;}
\item{\begin{itemize}
    \item {Take $Z^{0}(S/R, H \xrightarrow{\partial} G)$, the set of \mathdef{$0$-cocycles}, to be all pairs $(g, \theta)$ where $g \in G(S)$ and $\theta \in H(S \otimes_{R} S)$ is a \v{C}ech $1$-cocycle such that we have $\partial \theta = p_{1}^{\sharp}(g)^{-1} \cdot p_{2}^{\sharp}(g)$.}
\item{Given $(g_{i}, \theta_{i}) \in Z^{0}(S/R, H \xrightarrow{\partial} G)$, we can define 
\begin{equation*}
(g_{1}, \theta_{1}) \cdot (g_{2}, \theta_{2}) = (g_{1}g_{2}, \theta_{1}^{p_{1}^{\sharp}(g_{2})} \cdot \theta_{2});
\end{equation*}
this gives $Z^{0}(S/R, H \xrightarrow{\partial} G)$ the structure of a group.}
\item{We say that $(g, \theta) \in Z^{0}(S/R, H \xrightarrow{\partial} G)$ is a \mathdef{coboundary} if it is of the form $(\partial \mu, \theta_{\mu})$ where $\mu \in H(S)$ and $\theta_{\mu}:= p_{1}^{\sharp}(\mu)^{-1} \cdot p_{2}^{\sharp}(\mu)$. The set of coboundaries is a normal subgroup and we define $H^{0}(S/R, H \xrightarrow{\partial} G)$ to be the corresponding quotient.}
\end{itemize}}

\item{
\begin{itemize}
\item{Take $Z^{1}(S/R, H \xrightarrow{\partial} G)$, the set of \mathdef{$1$-cocycles}, to be pairs $(q, \varepsilon)$ with $q \in G(S \otimes_{R} S)$ and $\varepsilon \in H(S \otimes_{R} S \otimes_{R} S)$ satisfying the identities
\begin{equation*}
p_{1,3}^{\sharp}(q) \cdot \partial \varepsilon = p_{1,2}^{\sharp}(q) \cdot p_{2,3}^{\sharp}(q)
\end{equation*}
and 
\begin{equation*}
p_{1,2,4}^{\sharp}(\varepsilon) \cdot p_{2,3,4}^{\sharp}(\varepsilon) = p_{1,3,4}^{\sharp}(\varepsilon) \cdot p_{1,2,3}^{\sharp}(\varepsilon)^{p_{3,4}^{\sharp}(q)}. 
\end{equation*}
We encourage the reader to work out the fact that the above identities reduce to analogous group-cohomological identities in \cite[\S 3.1]{Noohi11} when $S/R$ is finite Galois.}

\item{One can give $Z^{1}(S/R, H \xrightarrow{\partial} G)$ the natural structure of a groupoid by defining an arrow $(q_{1}, \varepsilon_{1}) \to (q_{2}, \varepsilon_{2})$ as a pair $(g, \theta)$, where $g \in H(S)$ and $\theta \in G(S \otimes_{R} S)$ are such that 
\begin{equation*}
q_{2} = p_{1}^{\sharp}(g)^{-1} \cdot q_{1} \cdot p_{2}^{\sharp}(g) \cdot \partial \theta^{-1}
\end{equation*}
and
\begin{equation*}
\varepsilon_{2} = p_{1,3}^{\sharp}(\theta) \cdot \varepsilon_{1}^{p_{3}^{\sharp}(g)} \cdot p_{2,3}^{\sharp}(\theta)^{-1} \cdot p_{1,2}^{\sharp}(\theta^{-1})^{p_{2,3}^{\sharp}(q_{2})}.
\end{equation*}
Again, one checks that this reduces to the analogous group-theoretic identity (as in, e.g., \cite[\S 4.1]{Noohi11}) when $S/R$ is finite Galois. We define $H^{1}(S/R, H \xrightarrow{\partial} G)$ to be the corresponding set of isomorphism classes.}
\item{To define a group structure in degree $1$ we need to use our (symmetric) $S/R$ braiding. Given $(q_{i}, \varepsilon_{i}) \in Z^{1}(S/R, H \xrightarrow{\partial} G)$, we define $(q_{1}, \varepsilon_{1}) \cdot (q_{2}, \varepsilon_{2}) = (q, \varepsilon)$, where $q:= q_{1}q_{2}$ and 
\begin{equation*}
\varepsilon = \varepsilon_{1}^{p_{1,3}^{\sharp}(q_{2})} \cdot \varepsilon_{2} \cdot \{p_{1,2}^{\sharp}(q_{2}), p_{2,3}^{\sharp}(q_{1})\}^{p_{2,3}^{\sharp}(q_{2})}.
\end{equation*}}
This gives a well-defined group structure which preserves the above isomorphism classes, and thus gives a group structure on $H^{1}(S/R, H \xrightarrow{\partial} G)$. Since the braiding is symmetric, it is furthermore an abelian group (cf. \cite[\S 4]{Noohi11} for a proof of these claims in the Galois setting---it is straightforward to adapt them to the \v{C}ech setting).
\end{itemize}
}
\end{enumerate}
\end{definition}

When $H$ and $G$ are commutative, one can consider $H \xrightarrow{\partial} G$ as a complex with two terms, typically taken to be concentrated in degrees $-1$ and $0$ or $0$ and $1$. This perspective (and the associated double complexes) can be used to define cohomology groups, as studied in \cite[A.1, A.2]{Dillery2}. More specifically, these groups will either be of the form $H^{i}(\overline{F}/F, H \to G)$ (on the ``automorphic side") or $H^{i}(W_{F}, \widehat{G} \to \widehat{H})$ (on the ``Galois side"). The degree convention in \cite{Dillery2} is that the complex is concentrated in degrees $0$ and $1$. However, to match the cohomological degrees given above, one should use degrees $-1$ and $0$ instead, and thus all cohomology groups loc. cit. should be shifted down by $1$ (for example, $H^{1}$ loc. cit. is our $H^{0}$). When working with such groups, we always take the trivial braiding. 

One key property of the above cohomology groups---justifying the ambiguous notation---is:

\begin{lemma}\label{crossedcomp1} The cohomology groups in Definition \eqref{maincrossedef} coincide with the analogous constructions in \cite[A.1]{Dillery2} (with degrees shifted down by $1$) when $H$ and $G$ are commutative.
\end{lemma}

\begin{proof} This follows from a basic unpacking of the explicit description of the cohomology groups associated to the total complex of a double complex. 
\end{proof}

\begin{definition} A \mathdef{morphism} of two $S/R$-braided crossed modules $[H' \xrightarrow{\partial'} G'] \to [H \xrightarrow{\partial} G]$ is a morphism $(f, \phi)$ of complexes $[H' \to G'] \to [H \to G]$ such that $f$ is $G'$-equivariant, where $G'$ acts on $H$ via $\phi$ and is compatible with the braidings (in the obvious sense). We say that such a morphism is a \mathdef{quasi-isomorphism} if the induced maps $\mathrm{Ker}(\partial') \to \mathrm{Ker}(\partial)$ and $\mathrm{Coker}(\partial') \to \mathrm{Coker}(\partial)$ are isomorphisms. 
\end{definition}

Our main application of the above definition uses:

\begin{proposition}\label{qisom}
If $[H' \to G'] \xrightarrow{(f,\phi)} [H \to G]$ is a quasi-isomorphism of $S/R$-braided crossed modules then it induces an isomorphism on all cohomology groups as in Definition \ref{maincrossedef}.   
\end{proposition}

\begin{proof} When $S/R$ is finite Galois and we use the classical definition for $\Gamma = \mathrm{Aut}_{R}(S)$, this is the content of \cite[Proposition 1.2.2]{labesse99}. In the general \v{C}ech setting this is explained in \cite[\S 2]{Aviles12}.  
\end{proof}

\begin{remark} In many applications of crossed cohomology groups, one can often replace the $S/R$-crossed module $[H \to G]$ by a quasi-isomorphic crossed module $[H' \to G'] \to [H \to G]$ where $H'$ and $G'$ are abelian in order to (after applying Proposition \ref{qisom}) use Lemma \ref{crossedcomp1} to view the cohomology groups in the much simpler context of \cite[Appendix A]{Dillery2} rather than Definition \ref{maincrossedef}. From this point of view, the main contribution of Definition \ref{maincrossedef} is to give a definition of $H^{1}(\overline{F}/F, H \to G)$ which is intrinsic to $G$ and makes it clear that it does not depend on choice of $[H' \to G']$.
\end{remark}

\subsection{Duality results}\label{BraidedDual}
Let $G$ be a connected reductive group over $F$; we take the $\overline{F}/F$-crossed module $G_{\tn{sc}} \to G$, where $G$ acts on $G_{\tn{sc}}$ by conjugation, with symmetric braiding $\{g,h\} = g_{\tn{sc}}h_{\tn{sc}}g_{\tn{sc}}^{-1}h_{\tn{sc}}^{-1}$ (choosing $g_{\tn{sc}}$, $h_{\tn{sc}} \in G_{\tn{sc}}(\overline{F})$ whose image in $G_{\tn{ad}}(\overline{F})$ coincides with that of $g,h \in G(\overline{F})$ respectively).

We can also consider the $W_{F}$-crossed module $\widehat{G}_{\tn{sc}} \to \widehat{G}$ with analogous symmetric braiding and define the associated cohomology groups using the group-theoretic construction given in \cite[\S3, \S4]{Noohi11} which we mentioned but did not explicitly define in the previous subsection, as well as the cohomology of the crossed $W_{F}$-module $\widehat{Z}_{1} \xrightarrow{\partial} \widehat{Z}_{2}$ for any two multiplicative groups over $\mathbb{C}$ with a $W_{F}$-action (compatible with $\partial$). In this vein, we have:

\begin{corollary}\label{crossedcomp2} (cf. Lemma \ref{crossedcomp1}) In the context of the previous paragraph, the cohomology groups given by considering $W_{F}$-crossed modules coincide with their analogues defined in \cite[A.2]{Dillery2} (with appropriate degree shifts).
\end{corollary}

The main statement we need is the analogue of \cite[Proposition 5.19, Proposition 5.20]{kaletha15}:

\begin{proposition}\label{appdual1}
There are canonical, functorial isomorphisms
\begin{enumerate}
\item{\begin{equation*} \Hom_{\tn{cts}}(Z(G)(F), \mathbb{C}^{\times}) \xrightarrow{\sim} H^{1}(W_{F}, \widehat{G}_{\tn{sc}} \to \widehat{G});
\end{equation*}}
\item{\begin{equation*}
\Hom_{\tn{cts}}(H^{0}(\overline{F}/F, G_{\tn{sc}} \to G), \mathbb{C}^{\times}) \xrightarrow{\sim} H^{1}(W_{F}, \widehat{Z}).
\end{equation*}}
\end{enumerate}
Moreover, the compositions
\begin{equation*} H^{1}(W_{F}, \widehat{G}) \to H^{1}(W_{F}, \widehat{G}_{\tn{sc}} \to \widehat{G}) \xrightarrow{\sim} \Hom_{\tn{cts}}(Z(G)(F), \mathbb{C}^{\times})
\end{equation*}
and
\begin{equation*}
H^{1}(W_{F}, \widehat{Z})\xrightarrow{\sim} \Hom_{\tn{cts}}(H^{0}(\ov{F}/F, G_{\tn{sc}} \to G), \mathbb{C}^{\times}) \xrightarrow{\sim} \Hom_{\tn{cts}}(G(F)/G_{\tn{sc}}(F), \mathbb{C}^{\times})
\end{equation*}
agree with Langlands' constructions (given, for example, in \cite{Borel77}).
\end{proposition}

\begin{remark}
The constructions of Langlands referred to above works for any non-archimedean local field, so the claim makes sense as stated. Observe also that the only expression in Proposition \ref{appdual1} which uses the \v{C}ech version of the cohomology of crossed modules is the group $H^{0}(\overline{F}/F, G_{\tn{sc}} \to G)$; for the applications in the above paper we only need the first duality isomorphism. 
\end{remark}

\begin{proof}
We re-emphasize that this uses the identical arguments of \cite{kaletha15}. The purpose of writing out the argument is to highlight which results of \cite{Dillery2} one needs to use rather than their analogues in \cite{KS99}, and also how one deals with the \v{C}ech version of crossed modules (when applicable, cf. the previous remark). 

Set $\widehat{Z} = Z(\widehat{G})$; for any subgroup $M \subset \widehat{G}$, denote by $M_{\tn{sc}}$ its preimage in $\widehat{G}_{\tn{sc}}$. If we fix $\widehat{T} \subset \widehat{G}$ a maximal torus which is part of an $F$-splitting then we have a morphism of crossed $W_{F}$-modules 
\begin{equation*}
    [\widehat{T}_{\tn{sc}} \to \widehat{T}] \to [\widehat{G}_{\tn{sc}} \to \widehat{G}]
\end{equation*}
which induces an isomorphism on $W_{F}$-crossed cohomology. We can thus prove the result for $H^{1}(W_{F}, \widehat{G}_{\tn{sc}} \to \widehat{G})$ replaced by $H^{1}(W_{F}, \widehat{T}_{\tn{sc}} \to \widehat{T})$. The latter is dual to $H^{-1}(\overline{F}/F, T \to T_{\tn{ad}})$, where $T$ is the minimal Levi subgroup of $G$ (recall that $G$ is quasi-split) by \cite[Lemma A.3]{Dillery2} (using the comparison from Corollary \ref{crossedcomp2}). Now the first duality claim follows from the fact that the complex $[T \to T_{\tn{ad}}]$ is quasi-isomorphic to $[Z(G) \to 1]$ (also using Proposition \ref{qisom}). We alert the reader that, for the validity of this argument, one can define $H^{-1}(\overline{F}/F, T_{\tn{ad}} \to T)$ via the cohomology associated to a complex of tori as in \cite[A.1]{Dillery2} in order to avoid the technical construction of Definition \ref{maincrossedef}.

For the map (2) one uses the quasi-isomorphism of $\overline{F}/F$-braided crossed modules
\begin{equation*}
[T_{\tn{sc}} \to T] \to [G_{\tn{sc}} \to G]   
\end{equation*}
which induces an isomorphism of $H^{1}(\overline{F}/F, T_{\tn{sc}} \to T)$ with $H^{1}(\overline{F}/F, G_{\tn{sc}} \to G)$. The former can be interpreted in the framework of \cite[Appendix A]{Dillery2} (using Lemma \ref{crossedcomp1}) where one can use the same arguments as in \cite{kaletha15} (replacing the use of \cite[Lemma A.3.B]{KS99} loc. cit. with \cite[Lemma A.4]{Dillery2}) to get the desired result.

The agreement of the above compositions with Langlands' constructions proceeds verbatim (using $z$-extensions, which exist over any field) as in the mixed characteristic case. In particular, they only use the cohomology of crossed modules in the \v{C}ech setting to construct the map 
\begin{equation*}
 G(F)/G_{\tn{sc}}(F) \to  H^{0}(\overline{F}/F, G_{\tn{sc}} \to G),
\end{equation*}
which follows from the long exact sequence in crossed cohomology (cf. \cite[Proposition 2.4]{Aviles12}) associated to the $\overline{F}/F$-crossed module $[G_{\tn{sc}} \to G]$.
\end{proof}

\bibliography{rigvsisolocalFFbib}

\providecommand{\bysame}{\leavevmode\hbox to3em{\hrulefill}\thinspace}
\providecommand{\MR}{\relax\ifhmode\unskip\space\fi MR }
\providecommand{\MRhref}[2]{%
  \href{http://www.ams.org/mathscinet-getitem?mr=#1}{#2}
}
\providecommand{\href}[2]{#2}
\begin{thebibliography}{BMIY23}

\bibitem[BMIY23]{AMIY}
Alexander Bertoloni~Meli, Naoki Imai, and Alex Youcis, \emph{{The
  Jacobson–Morozov Morphism for Langlands Parameters in the Relative
  Setting}}, International Mathematics Research Notices (2023), rnad217.

\bibitem[Bor79]{Borel77}
A.~Borel, \emph{Automorphic {$L$}-functions}, Automorphic forms,
  representations and {$L$}-functions ({P}roc. {S}ympos. {P}ure {M}ath.,
  {O}regon {S}tate {U}niv., {C}orvallis, {O}re., 1977), {P}art 2, Proc. Sympos.
  Pure Math., vol. XXXIII, Amer. Math. Soc., Providence, RI, 1979, pp.~27--61.
  \MR{546608}

\bibitem[Bor17]{Borovoi16}
Mikhail Borovoi, \emph{Extending the exact sequence of nonabelian {$H^1$},
  using nonabelian {$H^2$} with coefficients in crossed modules}, 2017, arXiv:
  1608.07366.

\bibitem[Bus90]{Bushnell90}
Colin~J. Bushnell, \emph{Induced representations of locally profinite groups},
  J. Algebra \textbf{134} (1990), no.~1, 104--114. \MR{1068417}

\bibitem[Con22]{Conrad1}
Brian Conrad, \emph{Lecture notes for ``{A}lgebraic {G}roups {I}{I}" at
  {S}tanford {U}niversity}, May 2022.

\bibitem[Dil21]{Dillery2}
Peter Dillery, \emph{Rigid inner forms over global function fields}, 2021,
  arXiv:2110.10820.

\bibitem[Dil23]{Dillery1}
\bysame, \emph{Rigid inner forms over local function fields}, Adv. Math.
  \textbf{430} (2023), Paper No. 109204. \MR{4617942}

\bibitem[DS23]{DS23}
Peter Dillery and David Schwein, \emph{Rigid inner forms and the {B}ernstein
  decomposition for {$L$}-parameters}, 2023, arXiv: 2308.11752.

\bibitem[FS21]{FS21}
Laurent Fargues and Peter Scholze, \emph{Geometrization of the local
  {L}anglands correspondence}, 2021.

\bibitem[GA12]{Aviles12}
Cristian~D. Gonz\'{a}lez-Avil\'{e}s, \emph{Quasi-abelian crossed modules and
  nonabelian cohomology}, J. Algebra \textbf{369} (2012), 235--255.
  \MR{2959794}

\bibitem[Gir71]{Giraud71}
Jean Giraud, \emph{Cohomologie non ab\'{e}lienne}, Die Grundlehren der
  mathematischen Wissenschaften, vol. Band 179, Springer-Verlag, Berlin-New
  York, 1971. \MR{344253}

\bibitem[Kal15]{kaletha15}
Tasho Kaletha, \emph{Epipelagic {$L$}-packets and rectifying characters},
  Invent. Math. \textbf{202} (2015), no.~1, 1--89. \MR{3402796}

\bibitem[Kal16a]{kaletha16a}
\bysame, \emph{The local {L}anglands conjectures for non-quasi-split groups},
  Families of automorphic forms and the trace formula, Simons Symp., Springer,
  [Cham], 2016, pp.~217--257. \MR{3675168}

\bibitem[Kal16b]{kaletha16}
\bysame, \emph{Rigid inner forms of real and {$p$}-adic groups}, Annals of
  Mathematics, Second Series \textbf{184} (2016), no.~2, 559--632. \MR{3548533}

\bibitem[Kal18]{kaletha18}
\bysame, \emph{Rigid inner forms vs isocrystals}, J. Eur. Math. Soc. (JEMS)
  \textbf{20} (2018), no.~1, 61--101. \MR{3743236}

\bibitem[Kal23]{KalethaICM}
\bysame, \emph{Representations of reductive groups over local fields},
  I{CM}---{I}nternational {C}ongress of {M}athematicians. {V}ol. {IV}.
  {S}ections 5--8, EMS Press, Berlin, [2023] \copyright 2023, pp.~2948--2975.
  \MR{4680348}

\bibitem[Kot83]{kottwitz83}
Robert~E. Kottwitz, \emph{Sign changes in harmonic analysis on reductive
  groups}, Trans. Amer. Math. Soc. \textbf{278} (1983), no.~1, 289--297.
  \MR{697075}

\bibitem[Kot86]{Kottwitz86}
\bysame, \emph{Stable trace formula: elliptic singular terms}, Math. Ann.
  \textbf{275} (1986), no.~3, 365--399. \MR{858284}

\bibitem[Kot97]{Kottwitz97}
\bysame, \emph{Isocrystals with additional structure. {II}}, Compositio Math.
  \textbf{109} (1997), no.~3, 255--339. \MR{1485921}

\bibitem[KS99]{KS99}
Robert~E. Kottwitz and Diana Shelstad, \emph{Foundations of twisted endoscopy},
  Ast\'{e}risque (1999), no.~255, vi+190. \MR{1687096}

\bibitem[KT22]{Taibi22}
Tasho Kaletha and Olivier Ta\"{i}bi, \emph{The local {L}anglands conjecture},
  2022, Notes from the 2022 IHES Summer School on the Langlands program,
  \url{https://otaibi.perso.math.cnrs.fr/kaletha-taibi-llc.pdf}.

\bibitem[Lab99]{labesse99}
Jean-Pierre Labesse, \emph{Cohomologie, stabilisation et changement de base},
  Ast\'{e}risque (1999), no.~257, vi+161, Appendix A by Laurent Clozel and
  Labesse, and Appendix B by Lawrence Breen. \MR{1695940}

\bibitem[Lan83]{Lan83}
R.~P. Langlands, \emph{Les d\'{e}buts d'une formule des traces stable},
  Publications Math\'{e}matiques de l'Universit\'{e} Paris VII [Mathematical
  Publications of the University of Paris VII], vol.~13, Universit\'{e} de
  Paris VII, U.E.R. de Math\'{e}matiques, Paris, 1983. \MR{697567}

\bibitem[Mil06]{Milne06}
J.~S. Milne, \emph{Arithmetic duality theorems}, second ed., BookSurge, LLC,
  Charleston, SC, 2006. \MR{2261462}

\bibitem[Noo11]{Noohi11}
Behrang Noohi, \emph{Group cohomology with coefficients in a crossed module},
  J. Inst. Math. Jussieu \textbf{10} (2011), no.~2, 359--404. \MR{2787693}

\bibitem[Poo17]{Poonen17}
Bjorn Poonen, \emph{Rational points on varieties}, Graduate Studies in
  Mathematics, vol. 186, American Mathematical Society, Providence, RI, 2017.
  \MR{3729254}

\bibitem[Sha64]{Shatz}
Stephen~S. Shatz, \emph{Cohomology of artinian group schemes over local
  fields}, Ann. of Math. (2) \textbf{79} (1964), 411--449. \MR{193093}

\bibitem[Sil79]{Silberger79}
Allan~J. Silberger, \emph{Isogeny restrictions of irreducible admissible
  representations are finite direct sums of irreducible admissible
  representations}, Proc. Amer. Math. Soc. \textbf{73} (1979), no.~2, 263--264.
  \MR{516475}

\bibitem[SZ18]{SZ}
Allan~J. Silberger and Ernst-Wilhelm Zink, \emph{Langlands classification for
  {$L$}-parameters}, J. Algebra \textbf{511} (2018), 299--357. \MR{3834776}

\end{thebibliography}
\bibliographystyle{amsalpha}

\end{document}